\documentclass[12pt]{amsart}
\usepackage{amsmath,amssymb,amsthm,graphicx,mathrsfs,url,bbm,mathtools}
\usepackage[hscale=0.7,vscale=0.7]{geometry}
\usepackage[utf8]{inputenc}
\usepackage[T1]{fontenc}
\usepackage[usenames,dvipsnames]{color}
\usepackage[colorlinks=true,linkcolor=Red,citecolor=Green,pagebackref]{hyperref}
\usepackage[open=true,openlevel=2,depth=3]{bookmark}
\hypersetup{pdfstartview=XYZ}

\newcommand{\C}{\mathbb{C}}
\newcommand{\R}{\mathbb{R}}
\newcommand{\T}{\mathbb{T}}

\newcommand{\Z}{\mathbb{Z}}

\newcommand{\N}{\mathbb{N}}
\newcommand{\V}{\mathbb{V}}

\newcommand{\HH}{\mathbb{H}}
\newcommand{\Ss}{\mathbb{S}}
\newcommand{\eps}{\varepsilon}

\newcommand{\mc}{\mathcal}
\newcommand{\dd}{\mathrm{d}}

\DeclareMathOperator{\vol}{\mathrm{vol}}
\DeclareMathOperator{\Tr}{Tr}
\DeclareMathOperator{\Op}{Op}

\DeclareMathOperator{\ran}{ran}

\DeclareMathOperator{\Div}{div}
\DeclareMathOperator{\supp}{supp}

\theoremstyle{plain}
\newtheorem{theorem}{Theorem}
\newtheorem*{theorem*}{Theorem}
\newtheorem{lemma}{Lemma}[section]
\newtheorem{proposition}{Proposition}[section]
\newtheorem{corollary}{Corollary}[section]

\theoremstyle{definition}
\newtheorem{definition}{Definition}[section]

\newtheorem{remark}{Remark}[section]

\begin{document}

\newpage

\title[Rigidity of manifolds with hyperbolic cusps]{Local rigidity of manifolds with hyperbolic cusps \\ I. Linear theory and microlocal tools}

\author{Yannick Guedes Bonthonneau}
\address{Univ Rennes, CNRS, IRMAR - UMR 6625, F-35000 Rennes, France}
\email{yannick.bonthonneau@univ-rennes1.fr}

\author{Thibault Lefeuvre}
\address{Université de Paris and Sorbonne Université, CNRS, IMJ-PRG, F-75006 Paris, France.}
\email{tlefeuvre@imj-prg.fr}

\begin{abstract}
This paper is the first in a series of two articles whose aim is to extend a recent result of Guillarmou-Lefeuvre on the local rigidity of the marked length spectrum from the case of compact negatively-curved Riemannian manifolds to the case of manifolds with hyperbolic cusps. In this first paper, we deal with the linear (or infinitesimal) version of the problem and prove that such manifolds are \emph{spectrally} \emph{rigid} for compactly supported deformations. More precisely, we prove that the X-ray transform on symmetric solenoidal 2-tensors is injective. In order to do so, we expand the microlocal calculus developed by Bonthonneau and Bonthonneau-Weich to be able to invert pseudo-differential operators on Sobolev and Hölder-Zygmund spaces modulo compact remainders. This theory has an interest on its own and will be extensively used in the second paper in order to deal with the nonlinear problem.
\end{abstract}

\maketitle

%\newpage
%
%\tableofcontents
%
%\newpage

\section{Introduction}

\subsection{Spectral rigidity}

On a smooth closed (i.e. compact without boundary) Riemannian manifold $(M,g)$, the famous question of Kac \emph{``Can one hear the shape of a drum?"} \cite{Kac-66} asks whether one can reconstruct the Riemannian structure from the knowledge of the spectrum of the (non-negative) Laplacian $-\Delta_g$
\begin{equation}
\label{equation:spectrum}
\mathrm{spec}(-\Delta_g) = \left\{ \lambda_0 = 0 < \lambda_1 \leq \lambda_2 \leq ... \right\}.
\end{equation}
This question has triggered a lot of work, especially on negatively-curved manifolds, and is known to be false globally: Vigneras \cite{Vigneras-80} exhibited pairs of hyperbolic isospectral surfaces which are not isometric. Nevertheless, one can ask for a weaker statement in the spirit of Guillemin-Kazhdan \cite{Guillemin-Kazhdan-80, Guillemin-Kazhdan-80-2}, such as the \emph{infinitesimal spectral rigidity} of the manifold $(M,g)$. We recall that $(M,g)$ is said to be \emph{infinitesimally spectrally rigid} if any smooth isospectral deformation $(g_{\lambda})_{\lambda \in (-1,1)}$ (i.e. metric deformations with same spectrum of the Laplacian \eqref{equation:spectrum}) of the metric $g$ is trivial, namely there exists an isotopy $(\phi_{\lambda})_{\lambda \in (-1,1)}$ such that $\phi_\lambda^*g_\lambda = g$. In this article, we are interested in the question of \emph{infinitesimal spectral rigidity} of negatively-curved manifolds with hyperbolic cusps. In this case, the $L^2$-spectrum of the Laplacian is no longer discrete. It decomposes into a pure point spectrum -- eigenvalues -- and a continuous spectrum. Associated to the continuous spectrum, one can define some resonances which are complex numbers quantifying the way energy is lost in the cusps, for example in the heat equation. Gathering resonances and eigenvalues one obtains the \emph{resonant set} --- see \cite{Muller-83} for more details. We will say that two manifolds with cusps are isospectral if they have the same resonant set (when counted with multiplicity). 

In the case of a closed manifold, the infinitesimal spectral rigidity usually boils down to proving that the X-ray transform $I_2^{g}$ — that is, the integration of symmetric $2$-tensors along closed geodesics in $(M,g)$ — is injective on symmetric \emph{solenoidal} or \emph{divergence-free} $2$-tensors. This will be called \emph{solenoidal injectivity} in the rest of the paper. It is now well-known (at least for closed negatively-curved manifolds) that the solenoidal injectivity of the X-ray transform is a first step in proving the infinitesimal and local rigidity of the \emph{marked length spectrum}. Recall that on a closed negatively-curved manifolds $(M,g)$, the set of closed geodesics is in $1$-to-$1$ correspondence with the set $\mc{C}$ of free homotopy classes: in other words, given $c \in \mc{C}$, there exists a unique closed geodesic $\gamma_g(c) \in c$. The marked length spectrum is then defined as the map
\[
L_g : \mc{C} \rightarrow \R_+, ~~~~ c \mapsto \ell_g(\gamma_g(c)),
\]
where $\ell_g$ denotes the Riemannian length computed with respect to $g$. Up to restricting to \emph{hyperbolic free homotopy classes} (see below), one can still define the same notion of marked length spectrum on negatively-curved manifolds with hyperbolic cusps.

Contrary to the case of the Laplace spectrum, it is conjectured that the marked length spectrum of closed negatively-curved manifolds should determine the metric up to isometries. This is known as the Burns-Katok \cite{Burns-Katok-85} conjecture. More precisely, if $(M,g_1)$ and $(M,g_2)$ are two closed negatively-curved Riemannian manifolds such that $L_{g_1} = L_{g_2}$, it is conjectured that there exists a smooth diffeomorphism $\phi : M \rightarrow M$, isotopic to the identity, such that $\phi^*g_1 = g_2$. This was proved for surfaces in \cite{Croke-90,Otal-90} and in some other particular cases (see \cite{Katok-88,Besson-Courtois-Gallot-95,Hamenstadt-99}) but the conjecture remains open in full generality in dimension $\geq 3$. The \emph{infinitesimal marked length spectrum rigidity} question consists in considering a family of metric $(g_\lambda)_{\lambda \in (-1,1)}$ such that $L_{g_\lambda} = L_{g_0}$ and proving that there exists an isotopy $(\phi_\lambda)_{\lambda \in (-1,1)}$ such that $\phi^*_\lambda g_\lambda = g_0$. This result is also implied by the solenoidal injectivity of the X-ray transform operator $I_2^g$, as was shown by Guillemin-Kazhdan \cite{Guillemin-Kazhdan-80}. More recently, Guillarmou and the second author \cite{Guillarmou-Lefeuvre-18} showed that on closed negatively-curved manifolds, the solenoidal injectivity of the X-ray transform also implies the \emph{local rigidity of the marked length spectrum}: if two negatively-curved metrics that are close enough share the same marked length spectrum, then they are isometric. This result locally solves the Burns-Katok conjecture \cite{Burns-Katok-85}. The goal of our series of two papers is to extend this result from the case of closed negatively-curved manifolds to the case of negatively-curved manifolds with hyperbolic cusps. The main Theorem, similar to that of \cite{Guillarmou-Lefeuvre-18}, will be eventually proved in the second paper \cite{Bonthonneau-Lefeuvre-19-2}. Hence, just as in the compact case, a first step is to establish is solenoidal injectivity of the X-ray transform, which is the goal of the present paper.

As far as the solenoidal injectivity of the X-ray transform operator $I_2^g$ is concerned, it was first obtained for negatively-curved closed \emph{surfaces} by Guillemin-Kazhdan in their celebrated paper \cite{Guillemin-Kazhdan-80}. More generally, their proof works for tensors of any order $m \in \N$. This result was then extended by Croke-Sharafutdinov \cite{Croke-Sharafutdinov-98} to negatively-curved closed manifolds of arbitrary dimension. More recently, \cite{Paternain-Salo-Uhlmann-14-2} obtained the solenoidal injectivity of $I_2^g$ for any \emph{Anosov Riemannian surface} $(M,g)$, namely surfaces for which the geodesic flow is \emph{Anosov} or \emph{uniformly hyperbolic} on the unit tangent bundle $SM$. Guillarmou \cite{Guillarmou-17-1} then extended the result on Anosov surfaces to tensors of arbitrary order $m \in \N$. More generally, it is conjectured that the X-ray transform $I_m^g$ is solenoidal injective on closed Anosov Riemannian manifolds but the question remains open in dimension $\geq 3$. \\

In this article, we are interested in the solenoidal injectivity of $I_2^g$ on noncompact complete manifolds of negative curvature whose ends are real hyperbolic cusps and this does not seem to have been considered before in the literature. More precisely, the case we will consider will be that of a complete negatively-curved Riemannian manifold $(M,g)$ with a finite numbers of ends of the form
\[
Z_{a,\Lambda} = [a,+\infty[_y \times (\R^d/\Lambda)_\theta,
\]
where $a>0$, and $\Lambda$ is a crystallographic group (i.e a cocompact discrete group of isometries of $\R^d$) with covolume $1$. On this end, we have the metric $g = y^{-2} (dy^2 + d\theta^2)$. The sectional curvature of $g$ is constant equal to $-1$, and the volume of $Z_{a,\Lambda}$ is finite. All ends with finite volume and curvature $-1$ take this form. In dimension two, all cusps are the same (we must have $\Lambda = \Z$). However, in higher dimensions, if $\Lambda$ and $\Lambda'$ are not in the same orbit of $SO(d,\Z)$, $Z_{a,\Lambda}$ and $Z_{a',\Lambda'}$ are never isometric. In the following, we will sometimes call \emph{cusp manifolds} such manifolds. Up to taking a finite cover, we can always assume that each $\Lambda$ is a lattice in $\R^d$ (i.e only acts by translations).

\begin{figure}[h!]
\begin{center}
\centering
\def\svgwidth{0.5\linewidth}
%% Accompanies image file 'schema.pdf' (pdf, eps, ps)
\begingroup%
  \makeatletter%
  \providecommand\color[2][]{%
    \errmessage{(Inkscape) Color is used for the text in Inkscape, but the package 'color.sty' is not loaded}%
    \renewcommand\color[2][]{}%
  }%
  \providecommand\transparent[1]{%
    \errmessage{(Inkscape) Transparency is used (non-zero) for the text in Inkscape, but the package 'transparent.sty' is not loaded}%
    \renewcommand\transparent[1]{}%
  }%
  \providecommand\rotatebox[2]{#2}%
  \newcommand*\fsize{\dimexpr\f@size pt\relax}%
  \newcommand*\lineheight[1]{\fontsize{\fsize}{#1\fsize}\selectfont}%
  \ifx\svgwidth\undefined%
    \setlength{\unitlength}{173.35972236bp}%
    \ifx\svgscale\undefined%
      \relax%
    \else%
      \setlength{\unitlength}{\unitlength * \real{\svgscale}}%
    \fi%
  \else%
    \setlength{\unitlength}{\svgwidth}%
  \fi%
  \global\let\svgwidth\undefined%
  \global\let\svgscale\undefined%
  \makeatother%
  \begin{picture}(1,0.41280851)%
    \lineheight{1}%
    \setlength\tabcolsep{0pt}%
    \put(0,0){\includegraphics[width=\unitlength,page=1]{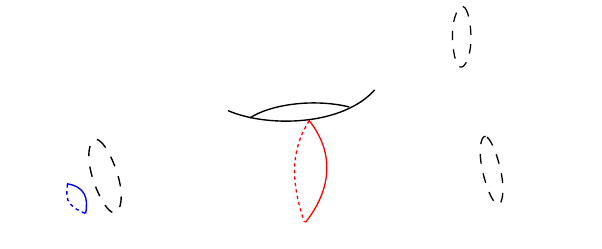}}%
    \put(0.79978635,0.32757294){\color[rgb]{0,0,0}\makebox(0,0)[lt]{\lineheight{1.25}\smash{\begin{tabular}[t]{l}$Z_1$\end{tabular}}}}%
    \put(0.84895247,0.12458668){\color[rgb]{0,0,0}\makebox(0,0)[lt]{\lineheight{1.25}\smash{\begin{tabular}[t]{l}$Z_2$\end{tabular}}}}%
    \put(0.02928133,0.15268167){\color[rgb]{0,0,0}\makebox(0,0)[lt]{\lineheight{1.25}\smash{\begin{tabular}[t]{l}$Z_3$\end{tabular}}}}%
    \put(0.33551669,0.11053919){\color[rgb]{0,0,0}\makebox(0,0)[lt]{\lineheight{1.25}\smash{\begin{tabular}[t]{l}$M_0$\end{tabular}}}}%
    \put(0,0){\includegraphics[width=\unitlength,page=2]{schema.pdf}}%
  \end{picture}%
\endgroup%

\caption{A surface with three cusps.}\label{figure:schema}
\end{center}
\end{figure}

In our case, we denote by $\mathcal{C}$ the set of hyperbolic free homotopy classes on $M$, which is in one-to-one correspondence with the set of hyperbolic conjugacy classes of $\pi_1(M,\cdot)$. These are the set of conjugacy classes of hyperbolic elements of the fundamental group which can be seen as isometric transformations on the universal cover $\widetilde{M}$ of $M$ ($\widetilde{M}$ is diffeomorphic to a ball) with two distinct fixed points on the boundary at infinity. From elementary Riemannian geometry, just as in the closed case, we know that for each such class $c\in \mathcal{C}$ of $C^1$ curves on $M$, there is a unique closed curve $\gamma_g(c)$ which is a \emph{geodesic} for $g$. We define the X-ray transform $I_2^g : C_{\mathrm{comp}}^\infty(M,S^2 T^*M) \rightarrow \ell^\infty(\mc{C})$ (where $\ell^\infty(\mc{C})$ denotes bounded sequences indexed by $\mc{C}$) as the following operator: if $h \in C_{\mathrm{comp}}^\infty(M,S^2 T^*M)$ then
\[
I_2^g h(c) = \frac{1}{\ell(\gamma_g(c))} \int_{0}^{\ell(\gamma_g(c))}  h_{\gamma(t)}(\dot{\gamma}(t),\dot{\gamma}(t)) dt,
\]
where $\gamma$ is a parametrization by arc-length. Of course, it is possible to define the X-ray transform for a larger class of tensors than those in $C_{\mathrm{comp}}^\infty(M,S^2 T^*M)$. The main result of our paper is the following:

\begin{theorem}
\label{theorem:xray-injectivite}
Let $d \geq 1$ and $(M^{d+1},g)$ be a negatively-curved complete manifold whose ends are real hyperbolic cusps. Let $-\kappa_0 < 0$ be the maximum of the sectional curvature. Then, for all $\alpha \in (0,1)$ and $\beta \in (0,\sqrt{\kappa_0}\alpha)$, the X-ray transform $I^{g}_2$ is injective on
\[
y^\beta C^\alpha(M, S^2 T^*M) \cap H^1(M, S^2 T^*M) \cap \ker D^*
\]
\end{theorem}

Here, $D^*$ denotes the divergence on $2$-tensors: as usual, a tensor $f$ is declared to be \emph{solenoidal} if and only if $D^*f = 0$ (see \S\ref{ssection:xray} for further details). It is defined as the formal transpose (for the $L^2$-scalar product) of the operator $D:=\sigma \circ \nabla$ acting on $1$-forms, where $\nabla$ is the Levi-Civita connection and $\sigma$ is the operator of symetrization of $2$-tensors. In turn, the previous Theorem implies the spectral rigidity for smooth compactly supported isospectral deformations.

\begin{corollary}
\label{corollary:xray-injectivite}
Let $d \geq 1$ and $(M^{d+1},g)$ be a negatively-curved complete manifold whose ends are real hyperbolic cusps. Let $(g_\lambda)_{\lambda \in (-1,1)}$ be a smooth isospectral deformation of $g=g_0$ with compact support in $M$. Then, there exists an isotopy $(\phi_\lambda)_{\lambda \in (-1,1)}$ such that $\phi_\lambda^* g_\lambda = g$.
\end{corollary}

This is the equivalent for cusp manifolds of the classical Guillemin-Kazhdan result \cite{Guillemin-Kazhdan-80-2}.

As mentioned earlier, Theorem \ref{theorem:xray-injectivite} is the first step towards proving the local rigidity of the marked length spectrum on such manifolds, as the X-ray transform on symmetric $2$-tensors turns out to be the differential of the marked length spectrum, and this program will be carried out in the following paper \cite{Bonthonneau-Lefeuvre-19-2}. We will refer to the local marked length spectrum rigidity question as the \emph{nonlinear problem}. In contrast, the \emph{linear} or \emph{infinitesimal problem} will be that of the solenoidal injectivity of the X-ray transform which is dealt in the present paper.

In order to prove Theorem \ref{theorem:xray-injectivite}, we will need — together with a Livsic-type theorem which does not really differ from the compact case — to study the decomposition of symmetric $2$-tensors into a \emph{potential part} and a \emph{solenoidal part} (or \emph{divergence-free part}). Namely, we will need to prove that any symmetric $2$-tensor $f$ can be written as $f = Dp + h$, where $p$ is a $1$-form and $h$ is solenoidal. The existence of such a decomposition relies on the analytic properties of the elliptic differential operator $D$ and in particular on the existence of a parametrix with compact remainder. Since the manifold $M$ is not compact, this theory is made harder (smoothing operators are no longer compact) and one has to resort to a careful analysis of the behaviour of the operator on the infinite ends of the manifold. A large part of this article is devoted to this study as the next paragraph explains.

\subsection{Pseudo-differential calculus on manifolds with hyperbolic cusps}

A careful study of the operators on the infinite ends of the models will be needed. The relevant techniques are that of Melrose's b-calculus \cite{Melrose-APS-93} which we will adapt to our setting. They will also be used in our second paper \cite{Bonthonneau-Lefeuvre-19-2}. While the operators $D$ and $D^*D$ studied in this first article are very likely to belong to the ``fibered cusp calculus'' introduced by Mazzeo-Melrose \cite{Mazzeo-Melrose-98}, we rather chose to expand the microlocal calculus developed in \cite{Bonthonneau-16} and \cite{Bonthonneau-Weich-17} and this for two main reasons:

\begin{enumerate}

\item First of all, in order to deal with the nonlinear problem in \cite{Bonthonneau-Lefeuvre-19-2}, we use the resolvent of the generator $X$ of the geodesic flow on the unit tangent bundle $SM$, as it was studied in \cite{Bonthonneau-Weich-17}. Since $X$ is \emph{not elliptic}, the techniques of Melrose \cite{Melrose-APS-93} \emph{cannot} be applied to study its analytic properties (at least, not in a straightforward manner ...) and to prove, in particular, the meromorphic extension of $(X\pm\tau)^{-1}$  to the whole complex plane. It was the purpose of \cite{Bonthonneau-Weich-17} to expand the relevant calculus introduced in \cite{Bonthonneau-16} in order to deal with such a non-elliptic operator.

\item Secondly, we will mostly be interested in the analytic behaviour of the operator $D^*D$ on \emph{weighted Hölder-Zygmund} spaces. On the one hand, this type of spaces does not seem to have been considered so far by the microlocal school working on noncompact manifolds neither in the general context of Melrose calculi, or specifically with cusps and for which we refer to \cite{Muller-83, Mazzeo-Melrose-98, Vaillant-01}. In particular, we prove boundedness results of pseudo-differential operators on such manifolds and show how to construct a parametrix on these spaces modulo a compact remainder. On the other hand, boundedness on this type of spaces for pseudo-differential operators on manifolds with \emph{bounded geometry} seems to have been considered by various authors (see \cite{Skrzypczak-98, Taylor-III-96edition} for instance). Roughly speaking, this assumption asserts that the manifold is \emph{uniformly comparable} to $\R^{d}$ and that the usual results known on $\R^d$ can be transferred to such manifolds. However, in our case, since the radius of injectivity collapses to $0$ in the cusps, we are not dealing with a bounded geometry and we cannot use such results. Hence, combining all the existing literature on the subject, it seems that our detailed study of pseudo-differential operators acting on Hölder-Zygmund spaces on manifolds with cusps is new. We refer to Section \S\ref{sec:fibred-cusp} for a more extensive discussion.

\end{enumerate}

Let us now briefly explain our analysis of pseudo-differential operators on cusp manifolds. Recall that the cusps are of the form
\[
Z_{a,\Lambda} = [a,+\infty[_y \times (\R^d/\Lambda)_\theta,
\]
where $y$ denotes the height variable and $\theta$ the slice variable, and we will take advantage of this product decomposition. Given $f \in C^\infty_{\mathrm{comp}}(M)$, we can always decompose $f|_{Z_{a,\Lambda}}$ in restriction to $Z_{a,\Lambda}$ as a sum $f = f_0 + f_\bot$, where $f_0 \in C^\infty_{\mathrm{comp}}(Z_{a,\Lambda})$ is independent of the slice variable $\theta$,
\[
f_0(y) := \int_{\R^d/\Lambda} f(y,\theta) d \theta
\]
is called the \emph{$0$-th Fourier mode of $f$} (note that we do not need to divide by the volume of the slice, since by assumption the volume of $\R^d/\Lambda$ is equal to $1$), and $f_\bot := f-f_0 \in C^\infty_{\mathrm{comp}}(Z_{a,\Lambda})$ is called the \emph{non-zero Fourier modes of $f$}. Of course, such a decomposition can be naturally extended to sections of vector bundles.

In the core of this article, we will be working with \emph{admissible pseudo-differential operators} (see Definition \ref{def:R-L^2-admissible-operator} for a precise statement) acting on sections of vector bundles over the cusp manifold. Roughly speaking, these operators are of \emph{geometric nature} (they will be constructed out of the metric, just like the Laplace-Beltrami operator acting on functions for instance) and, as a consequence of their geometric definition, they will act diagonally on the decomposition of sections into zero and non-zero Fourier modes, modulo some compact remainders. As we will see, the construction of a parametrix for an elliptic operator is made difficult due to the zero Fourier mode and this is really where the non-compactness of the manifold comes into play. 

The precise study of the operator on the zero Fourier mode is done through an \emph{indicial operator} as in Melrose \cite{Melrose-APS-93} which reveals some \emph{indicial roots}, playing a crucial role in the description of the elliptic operator as a Fredholm operator (i.e. an operator with finite kernel and cokernel). In particular, according to the \emph{weighted spaces} considered (these are functional spaces where the norm includes a weight $y^\rho$ for some $\rho \in \R$), the behaviour of an operator might be drastically different. As for the introduction, we state a simpler version of our main theorem of inversion (see Theorem \ref{theorem:parametrix-compact}) in the case where the fiber over the cusp is trivial (it is a point) and the operator acts on distributions on $M$.

\begin{theorem}
Let $(M^{d+1},g)$ be a negatively-curved complete manifold whose ends are real hyperbolic cusps. Let $P$ be a differential operator on $M$. Assume that it is $\R$-admissible in the sense of Definition \ref{def:R-L^2-admissible-operator}. Also assume that it is uniformly elliptic in the sense of Definition \ref{def:elliptic}. Then there is a discrete set $S(P)\subset \R$ such that for each connected component $I \subset \R\setminus S(P)$, and $\rho\in I$, $P$ is Fredholm on every space $y^{\rho-d/2}H^s(M)$ and $y^\rho C^s_\ast(M)$, with a Fredholm index which only depends on $I$.
\end{theorem}

The spaces $H^s(M)$ are the usual Sobolev spaces built from the metric. The spaces $C^s_\ast(M)$ are the Hölder-Zygmund spaces, introduced in Section \S\ref{section:holder-zygmund}. They coincide with the usual Hölder spaces $C^s(M)$ built from the distance (induced by the metric) for $s \in \R_+ \setminus \N$ (see Proposition \ref{proposition:correspondence-espaces-holder}). By a little more effort, we also obtain a relative Fredholm index Theorem (see Formula \eqref{equation:fredholm-index-theorem}), allowing to describe the jump of the Fredholm index as one crosses an indicial root (see Propositions \ref{prop:bigger-kernel} and \ref{prop:smaller-image}). This also holds in the Hölder-Zygmund category.

\subsection{Outline of the paper}

In Section \S\ref{section:pseudo}, we introduce the basic functional spaces and the class of pseudo-differential operators we will be working with. In Section \S\ref{section:holder-zygmund}, we prove boundedness properties of our class of pseudo-differential operators on Hölder-Zygmund spaces. Section \S\ref{section:parametrices1} is dedicated to the notion of \emph{indicial operator} and to the inversion of an elliptic pseudo-differential operators on weighted Sobolev and Hölder-Zygmund spaces. 

In the last Section \S\ref{section:geometry}, we show how the previous theory can be applied to the operators $\nabla_S$ (the gradient of the Sasaki metric on the unit tangent bundle $SM$), $D$ and $D^*D$. This will provide the decomposition of tensors into a potential and a solenoidal part. We also obtain a Livsic Theorem (see Theorem \ref{theorem:livsic}) which is rather similar to the compact case. In the end, gathering all these different pieces together, we will deduce Theorem \ref{theorem:xray-injectivite} and Corollary \ref{corollary:xray-injectivite}. \\

\noindent \textbf{Acknowledgements:} We warmly thank the anonymous referee for his helpful comments which improved the exposition of the article. We also thank Viviane Baladi, Sébastien Gouëzel, Colin Guillarmou, Sergiu Moroianu, Davi Obata, Frédéric Paulin, Frédéric Rochon, for helpful remarks and useful discussions. T.L. also thanks the reading group on b-calculus in Orsay for sharing their knowledge and enthusiasm. T.L. has received funding from the European Research Council
(ERC) under the European Union’s Horizon 2020 research and innovation programme
(grant agreement No. 725967). This material is based upon work supported by the National Science Foundation under Grant No. DMS-1440140 while T.L. was in residence at the Mathematical Sciences Research Institute in Berkeley, California, during the Fall 2019 semester.

\section{Pseudo-differential operators on manifolds with hyperbolic cusps}

\label{section:pseudo}

\subsection{The geometric setup and main result}

\subsubsection{Admissible bundles}

Throughout the article, we will rely on constructions from \cite{Bonthonneau-Weich-17}, itself based on \cite{Bonthonneau-16}. In the former paper, the techniques from Melrose \cite{Melrose-APS-93} had to be adapted to deal with operators that are \emph{not} elliptic. In Section \S\ref{sec:fibred-cusp}, we will compare our setup to that of Mazzeo and Melrose's \emph{fibred cusp calculus}. Since we want to state results in some generality, we will consider in this whole section the following situation: we are given a non-compact manifold $N$ with a finite number of ends $N_\ell$, which take the form 
\begin{equation}\label{eq:trivial-fibred-end}
Z_{\ell,a}\times F_\ell.
\end{equation}
Here, $Z_{\ell,a} = \{ z \in Z_\ell\ |\ y(z)> a\}$, and
\[
Z_\ell = ]0,+\infty[_y \times \left(\R^d/\Lambda_\ell\right)_\theta.
\]
In all generality, $\Lambda_\ell \subset O(d)\ltimes \R^d$ is a crystallographic group. However, according to Bieberbach's Theorem, up to taking a finite cover, we can assume that $\Lambda\subset \R^d$ is a lattice of translations. We will work with that case, and check that the results are stable by taking quotients under free actions of finite groups of isometries. The slice $(F_\ell,g_{F_\ell})$ is assumed to be a closed Riemannian manifold. We will use the variables $(y,\theta,\zeta) \in Z_\ell \times F_\ell$ where $(y,\theta) \in [a,+\infty) \times \R^d/\Lambda_\ell$. We assume that $N$ is endowed with a metric $g_N$, equal over the cusps to the product metric
\[
g_N|_{Z_{\ell,a}} = \frac{dy^2 + d\theta^2}{y^2} + g_{F_\ell}.
\]
We therefore have a decomposition
\[
N = N_0 \cup_\ell N_\ell,
\]
where $N_0$ is a compact submanifold of $N$ with boundary. We will refer to $N_0$ as the compact part of $N$ and to the $N_{\ell}$'s as the cuspidal parts of $N$. An example of such a manifold is given by $N = SM$, the unit tangent bundle of our manifold $(M,g)$ with hyperbolic cusps. 

We will also have a vector bundle $L\to N$, and will assume that for each $\ell$, there is a vector bundle $L_\ell \to F_\ell$, so that
\[
L_{|N_\ell} \simeq Z_\ell \times L_\ell.
\]
Whenever $L$ is a hermitian vector bundle with metric $g_L$, a compatible connection $\nabla^L$ is one that satisfies
\[
X g_L(f_1,f_2) = g_L(\nabla^L_X f_1, f_2) + g_L(f_1, \nabla^L_X f_2),
\]
where $f_{1,2} \in C^\infty_{\mathrm{comp}}(N,L)$. Taking advantage of the product structure, we impose that when $X$ is tangent to $Z$,
\begin{equation}\label{eq:structure-connection}
\nabla_X^L f(y,\theta,\zeta) = d_{(y,\theta)} f(X)+ A_{(y,\theta)}(X)\cdot f,
\end{equation}
(since the cylinder $Z$ has a flat structure, the differential $d_{(y,\theta)}$ is well defined). Here the connection form $A_x(X)$ is an anti-symmetric endomorphism depending linearly on $X$, and $A(y\partial_y)$, $A(y\partial_\theta)$ do not depend on $y,\theta$. In particular, we get that the curvature of $\nabla^L$ is bounded, as are all its derivatives.

\begin{definition}
\label{definition:admissible-bundle}
Such data $(L\to N, g, g_L,\nabla^L)$ will be called an \emph{admissible bundle}.
\end{definition}

Given a cusp manifold $(M,g)$, the bundle of differential forms over $M$ is an admissible bundle. Since the tangent bundle of a cusp is trivial, any linearly constructed bundle over $M$ is admissible. For example, the bundle of forms over the Grassmann bundle of $M$, or over the unit cosphere bundle $S^\ast M$. The Sasaki metric on $SM$ is not a product metric, however, it is uniformly equivalent to the natural product metric in the cusps (see Appendix C in \cite{Bonthonneau-16}). In particular, it defines the same classes of regularity.

\subsubsection{Functional spaces}

\label{sec:functional-spaces}

Let $f$ be a function on $N$. We define for an integer $k \geq 0$:
\[
\|f\|_{C^k(N)} :=  \sup_{0 \leq j \leq k} \sup_{z \in N}\|\nabla^j f(z)\|,
\]
where $\nabla$ is the Levi-Civita connection induced by $g_N$. We write $f \in C^\infty(N)$ if all the derivatives of $f$ are bounded. If $f$ is infinitely many times differentiable, but its derivatives are not bounded, we simply say that $f$ is \textit{smooth}.

The Christoffel coefficients of the metric in the cusp in the frame
\[
X_y := y\partial_y,\ X_\theta:= y\partial_\theta,\ X_\zeta:= \partial_\zeta
\]
are independent of $(y,\theta)$ (see \cite[Appendix A.3.2]{Bonthonneau-thesis}). As a consequence, in the cusp, there are uniform constants such that
\begin{equation}\label{eq:equivalent-local-Ck-norm}
\sup_{0 \leq j \leq k} \|\nabla^j f(z)\| \asymp \sup_{|\alpha| \leq k} |X_\alpha f(z)|,
\end{equation}
for all $z=(y,\theta,\zeta) \in N$. Here, $\alpha$ is an ordered multiindex with values in $\{y,\theta,\zeta\}^k$. 

We now introduce Hölder spaces. Let $0 < \beta < 1$. We will write $f \in C^\beta(N)$ if:
\[ 
\|f\|_{C^\beta} := \sup_{z \in N} |f(z)| + \sup_{z,z' \in N, z \neq z'} \frac{|f(z)-f(z')|}{d(z,z')^\beta} = \|f\|_\infty + \|f\|_\beta < \infty, 
\]
where $d(\cdot,\cdot)$ refers to the Riemannian distance induced by the metric $g_N$. In particular, a function $f$ may be $\beta$-Hölder continuous, with a uniform Hölder constant of continuity (i.e. $\|f\|_\beta < \infty$), but may not be in $C^\beta(N)$ if $\|f\|_\infty = \infty$ for instance. It also makes sense to define $C^\beta$ for $\beta \in \R_+ \setminus \N$ by asking that $f \in C^{[\beta]}(N)$ and that the $[\beta]$-th derivatives of $f$ are $\beta-[\beta]$ Hölder-continuous.

The Lebesgue spaces $L^p(N)$, for $p \geq 1$, are the usual spaces defined with respect to the Riemannian measure $d\mu$ induced by the metric $g_N$. Over the cusp, it has the particular expression $d\mu = y^{-d-1} dy d\theta d\vol_{F_\ell}(\zeta)$, where $d\vol_{F_\ell}$ denotes the Riemannian measure induced by the metric $g_{F_\ell}$. For $s \in \R$, we define (via the spectral theorem): 
\begin{equation}
\label{equation:norm-sobolev}
\|f\|_{H^s(N)}:= \|(-\Delta+1)^sf\|_{L^2(N)},
\end{equation}
and $H^s(N)$ is the completion of $C_{\mathrm{comp}}^\infty(N)$ with respect to this norm. Here $\Delta$ is the Laplacian induced by the metric $g_N$.

Let $\widetilde{y} : N \rightarrow \R_+$ be a smooth positive function such that $\widetilde{y}(z) = y$, for $z \in Z_{\ell,a}$. \textbf{In the following, we will abuse notations and confuse the function $\widetilde{y}$ defined globally on $M$ with the $y$ coordinate on the cusps.} We can now introduce weighted spaces $y^\rho \mathfrak{X}^s(N)$ (for $\mathfrak{X}=C,H$) by the following formulas:
\[
\|f\|_{y^\rho \mathfrak{X}^s(N)} := \|y^{-\rho}f\|_{\mathfrak{X}^s(N)}.
\]
For the reader to get familiar with these spaces, let us mention here some embedding lemmas.

\begin{lemma}
\begin{enumerate}
\item Let $0\leq s < s' <1$ and $\rho-d/2< \rho'$. Then
\[
y^{\rho} C^{s'}(N) \hookrightarrow y^{\rho'}H^s(N)
\]
is a continuous embedding.
\item Let $k \in \N$, $s > \frac{d+1}{2}+k$. Then
\[
y^{-d/2}H^{s}(N) \hookrightarrow C^k(N)
\]
is a continuous embedding.
\end{enumerate}
\end{lemma}

The shift by $y^{d/2}$ will often appear throughout the article and is due to the fact that Sobolev spaces are built from the $L^2$ space induced by the hyperbolic measure $dyd\theta d\vol(\zeta)/y^{d+1}$. We will prove (and even refine) these embedding lemmas in Section \S\ref{sec:embeddings}.

So far, we have only introduced the functional spaces for functions. When working with sections of vector bundles, one needs a connection to compute derivatives. Then one can define the relevant functional spaces in essentially the same way. For those spaces to behave in a reasonable fashion, one needs that the connection itself is uniformly $C^\infty$ with respect to itself. It is in particular the case when it is the connection of an admissible bundle.

\subsection{Pseudo-differential operators on cusps}

We now introduce our algebra of pseudo-differential operators. We want to consider the action of operators on sections of $L\to N$ or more generally from sections of $L_1\to N$ to sections of $L_2\to N$ where $L_{1,2}$ are admissible bundles. In the paper \cite{Bonthonneau-Weich-17}, an algebra of \emph{semi-classical} operators was described using results from \cite{Bonthonneau-16}. It consists of families of operators depending on a small parameter $h>0$. In this paper, most of the time, we will be using \emph{classical} operators, which is equivalent to fixing the value of $h$ to $1$. The description of the algebra of operators we are using relies on two points: first, we need to say which types of smoothing remainders are allowed; second, we need to describe the quantization we will manipulate.

\subsubsection{Various types of smoothing operators}

\label{sssection:smoothing}

The non-compactness of the manifold, and the fact that we consider Sobolev and H\"older-Zygmund spaces makes it unavoidable to use several classes of smoothing operators in the paper. Let us take the time to properly present them. 

\begin{enumerate}
\item The smallest class of smoothing operators, which we will call \textbf{$\R$-residual operators} and denote by $\dot{\Psi}^{-\infty}_\R$ comprises the operators $R$ that are bounded as maps
\[
R : y^{\rho} H^{-k}(N,L_1) \to y^{-\rho} H^k(N,L_2),
\]
for any $\rho > 0, k \geq 0$. We will define in Section \S\ref{sec:functional-spaces} some H\"older-Zygmund spaces $C^s_\ast$, and we could replace the spaces $H^k$ by the spaces $C^s_\ast$ in the definition of $\R$-residual operators, and \emph{obtain the same class of operators}, as is stated in Proposition \ref{prop:maximallyresidualholder}.

\item Another class of operators of interest is that of \textbf{$\R$-smoothing operators}, denoted $\Psi^{-\infty}_\R$, which comprises the operators $R$ that are bounded as maps
\[
R : y^{\rho} H^{-k}(N,L_1) \to y^{\rho} H^k(N,L_2),
\]
for any $\rho \in \R, k \geq 0$. If one replaces Sobolev spaces by H\"older-Zygmund spaces, one obtains a notion of $\R$-smoothing operators in Hölder-Zygmund. \emph{In this case, it is not clear to us whether both classes are the same.} We will make further comments on this right after the proof of Proposition \ref{prop:relation-symbol-geometric}.

\item Now, it will turn out that the range of $\rho$'s that are allowed may have to be restricted, so that one has to introduce a bit more of notations. Specifically, given a non-trivial interval $I=  (\rho_-, \rho_+)\subset \R$, we will say that $R$ is an \textbf{$I$-residual operator} if for any $\rho, \rho'\in I$, and any $k\in \R$, $R$ is bounded as a map
\[
R : y^{\rho - d/2}H^{-k} \to y^{\rho' -d/2} H^k.
\]
The $-d/2$ factor is here to take into account that the measure in the cusps is $y^{-d-1}dyd\theta d\zeta$.  The class of such operators is denoted $\dot{\Psi}^{-\infty}_I$. We will see that it is equivalent to requiring that for the same range of parameters, $R$ is bounded as a map
\[
R : y^{\rho}C^{-k}_\ast \to y^{\rho'}C^k_\ast.
\]

\item We also have \textbf{$I$-smoothing operators}, which are operators $R$ such that for all $\rho\in I$, and all $k\in \R$, $R$ is bounded as a map 
\[
R : y^{\rho - d/2}H^{-k} \to y^{\rho -d/2} H^k.
\]
The class of these operators is denoted $\Psi^{-\infty}_I$. Again, if one replaces Sobolev spaces with H\"older-Zygmund spaces, it defines a class of $I$-$L^\infty$-smoothing operators. \emph{It is not clear to us if this defines the same class.}

\end{enumerate}

The reader familiar with Melrose's b-calculus will observe that $\R$-residual here is the equivalent of Melrose's maximally residual operators.

\subsubsection{Hyperbolic quantization}

\label{sssection:hyperbolic-quantization}

Let us now describe our quantization. In the compact part, we will use usual pseudo-differential operators with symbols $\sigma$ in the Kohn-Nirenberg class, satisfying usual estimates in charts of the form
\[
|\partial_x^\alpha \partial_\xi^\beta \sigma| \leq C_{\alpha,\beta} \langle \xi\rangle^{m-|\beta|}.
\]
For this, we refer to \cite{Zworski-book} for instance. As a consequence, it suffices to explain what we will be calling a pseudo-differential operator in the \emph{ends}. We will recall the constructions of \cite{Bonthonneau-16,Bonthonneau-Weich-17}. For this, we consider one end, and we drop the $\ell$'s. Instead of quantizing $Z_{a}$, we work with the full cusp $Z$.

Let us denote by $\Op^w$ the usual Weyl quantization on $\R^{d+1}\times \R^{k}$. Given $\chi\in C^\infty_{\mathrm{comp}}(\R)$ equal to $1$ around $0$, and $a\in \mathcal{S}'(\R^{2d+2k+2})$, we denote by $\Op^w(a)_\chi$ the operator (on $\R^{+\ast}_y \times \R^d \times \R^k$ now, to avoid $y, y' =0$) whose kernel is
\begin{equation}\label{equation:cutoff}
K(y, \theta, \zeta;y', \theta', \zeta') = \chi\left[ \frac{y'}{y} - 1 \right]K_{\Op^w(a)}(y,\theta,\zeta;y',\theta', \zeta').
\end{equation}
Next, we can associate $a \in C^\infty( T^\ast (Z\times \R^k), \mathrm{Hom}(\R^{n_1},\R^{n_2}))$ with its periodic lift 
\[
\tilde{a} \in C^\infty( T^\ast (\R_y \times \R^{d}_\theta\times \R^k_\zeta),\mathrm{Hom}(\R^{n_1},\R^{n_2})).
\]
(supported for $y>0$). Given $f\in C^\infty(Z\times \R^k,\R^{n_1})$, denoting by $\tilde{f}$ the periodic lift to $\R^{d+1}\times \R^k$, it follows from the explicit expression of $\Op^w(\tilde{a})_\chi \tilde{f}$ that it is again periodic. In particular, $\Op^w(\tilde{a})_\chi$ defines an operator from compactly supported smooth sections of $\R^{n_1} \to Z\times \R^k$ to distributional sections of $\R^{n_1} \to Z\times \R^k$.

As a consequence, it makes sense to set
\[
\Op_{\R^k}(a) f =   y^{(d+1)/2} \Op^w(a)_\chi[ y^{-(d+1)/2} f ].
\]
Using a partition of unity on $F_\ell$, we can globalize this to a Weyl quantization $\Op^w_{N_\ell,L_1\to L_2}$, and then on the whole manifold $\Op^w_{N,L_1\to L_2}$ --- the arguments in \cite[Section 14.2.3]{Zworski-book} apply. We will write $\Op$ for this Weyl quantization on the whole manifold. Since $F$ is compact, one checks that the resulting operators are uniformly properly supported above each cusp.

Now, we need to say more about the symbol estimates that we will require. We introduce
\[
\langle \xi \rangle := \sqrt{1+g^*_N(\xi,\xi)},
\]
the Japanese bracket of $\xi \in T^*N$ with respect to the natural metric $g_N^\ast$ on $T^\ast N$ (this is the dual metric to the Sasaki metric). Due to the product structure of the metric $g_N$ on $N$ over the cuspidal part, it is equal to $g_{Z_\ell}^\ast + g_{F_\ell}^\ast$. Over the cusps, we denote by $(Y,J,\eta) \in \R \times \R^d \times T^*F_{\ell}$ the dual variables to $(y,\theta,\zeta) \in [a,+\infty) \times \R^d/\Lambda \times F_{\ell}$ (such global coordinates are possible because the (co)tangent bundle of the cusp is trivial). In the case that $F_\ell$ is a point, we then have:
\[
\langle \xi \rangle = \sqrt{1+y^2|\xi|_{\mathrm{euc}}^2}  = \sqrt{1+y^2(Y^2+J^2)}.
\]
Let $\pi : T^*N \rightarrow N$ be the canonical projection on the base. The vector bundles $L_{1,2} \to N$ can be lifted to $T^*N$ in order to define vector bundles $\pi^*L_{1,2} \rightarrow T^*N$. In the following, we will consider such lifts (and canonical vector bundles constructed from it) but we will drop the $\pi^*$ notation.

\begin{definition}
A \emph{symbol of order $m$} is a smooth section $a$ of $\mathrm{Hom}(L_1,L_2)\to T^\ast N$, that satisfies the usual estimates over $N_0$ (the compact core of $N$), and above each $N_\ell$ (the cuspidal parts), and in local charts in $F_\ell$, for each $\alpha,\beta,\gamma,\alpha',\beta',\gamma'$, there is a constant $C>0$:
\[
\left|(y\partial_y)^\alpha (y\partial_\theta)^\beta (\partial_\zeta)^\gamma \ (y^{-1}\partial_Y)^{\alpha'}(y^{-1}\partial_J)^{\beta'}  (\partial_\eta)^{\gamma'} a \right|_{\mathrm{Hom}(L_1,L_2)} \leq C \langle \xi \rangle^{m - \alpha' - |\beta'| - |\gamma'|}.
\]
We write $a\in S^m(T^\ast N, \mathrm{Hom}(L_1,L_2))$.
\end{definition}

Observe that this does not actually depend on the order in which the derivatives were taken, nor does it depend on the choice of charts on $F_\ell$ since it is compact.

\subsubsection{Algebra of small pseudo-differential operators}

We can now introduce the class of \emph{small pseudo-differential operators}:

\begin{definition}
\label{definition:small}
The class of \emph{small pseudo-differential operators} is defined as:
\[
\Psi^m_{\text{small}}(N, L_1\to L_2) := \left\{ \Op(a) + R ~|~ a\in S^m(T^\ast N, \mathrm{Hom}(L_1,L_2)), R\in \Psi^{-\infty}_{\R} \right\}.
\]
\end{definition}

The choice of the adjective ``small'' refers to Melrose's small calculus \cite{Melrose-APS-93}. This class of operators satisfies the usual properties that we present now.

\begin{proposition}\label{prop:microlocal-calculus}
\begin{enumerate}
	\item Let $m \in \R$. Then $A \in \Psi^m_{\text{small}}(N, L_1\to L_2)$ is continuous from $y^\rho H^s(N,L_1)$ to $y^\rho H^{s-m}(N,L_2)$ for all $s,\rho\in \R$. 
	\item  Let $m, m' \in \R$, $a\in S^m(T^\ast N,\mathrm{Hom}(L_1,L_2))$ and $b\in S^{m'}(T^\ast N,\mathrm{Hom}(L_2,L_3))$. Then $\Op(a)\Op(b)\in \Psi^{m+m'}_{\text{small}}$, and
\[
\Op(a)\Op(b) - \Op(ab) \in \Psi^{m+m' - 1}_{\text{small}}.
\]
In particular, $\Psi^*_{\text{small}} := \cup_{m \in \R} \Psi^m_{\text{small}}$ is an algebra.
\end{enumerate}

\end{proposition}

For the proof, we refer to \cite[Proposition A.8]{Bonthonneau-Weich-17} and \cite[Proposition 1.19]{Bonthonneau-16}. The only difficulty may be that in \cite[Proposition A.8]{Bonthonneau-Weich-17}, the boundedness was only considered on spaces $H^s$, instead of $y^\rho H^s$. However, since $\Op(\sigma)$ is uniformly properly supported (due to the introduction of the cutoff function $\chi$ in the quantization \eqref{equation:cutoff}), we can conjugate it by $y^\rho$, and observe that this gives an operator still in the class $\Psi^*_{\text{small}}$.

It will be useful to recall that the proof of the boundedness relies on building a symbol $\sigma_s$ such that $H^s = \Op(\sigma_s)L^2$ (so that $\Op(\sigma_s)$ approximates $(-\Delta + C)^{-s/2}$ for some large constant $C>0$). In the following, we will write $\Lambda_{-s} = \Op(\sigma_s)$.

\begin{definition}\label{def:elliptic}
Let $a\in S^m(T^\ast N,\mathrm{Hom}(L_1,L_2))$. We will say that $a$ is \emph{left (resp. right) uniformly elliptic} if there exists $C > 0$ and $b\in S^{-m}(T^\ast N,\mathrm{Hom}(L_2,L_1))$ such that $b$ is a left (resp. right) inverse for $a$, in the sense that $b(z,\xi)a(z,\xi) = \mathbbm{1}_{L_1}$ (resp. $a(z,\xi)b(z,\xi) = \mathbbm{1}_{L_2}$) for all $(z,\xi) \in T^*N$ such that $g_N^*(\xi,\xi) \geq C$.

When $L_1$ and $L_2$ have the same dimension, both definitions are equivalent and we just say that $a$ is uniformly elliptic.
\end{definition}

\begin{proposition}\label{prop:smoothing-remainder}
Let $a\in S^m(T^\ast N, \mathrm{Hom}(L_1,L_2))$ be left (resp. right) uniformly elliptic. Then we can find $Q \in \Psi^{-m}_{\text{small}}(N,\mathrm{Hom}(L_2,L_1))$ such that
\[
Q \Op(a) = \mathbbm{1} + R \quad (\text{resp. } \Op(a) Q = \mathbbm{1} + R),
\]
with $R \in \Psi^{-\infty}_{\text{small}}(N,\mathrm{Hom}(L_1,L_1))$.
\end{proposition}

Before going on with the proof, observe that the remainder here is not necessarily a compact operator, contrary to the case of a closed manifold.

\begin{proof}
It suffices to deal with the left elliptic case. Here, we can apply the usual parametrix construction. By definition of (left/right) ellipticity, we can choose a $q_0 \in S^{-m}(T^\ast N,\mathrm{Hom}(L_2,L_1))$ such that for all $(z,\xi) \in T^*N$ such that $g_N^*(\xi,\xi)> C$,
\[
q_0(z,\xi) a(z,\xi) = \mathbbm{1}_{L_1}.
\]
Then
\[
\Op(q_0) \Op(a) = \mathbbm{1}_{L_1} + \Op(r_1) + R_1.
\]
Here, $r_1 \in S^{-1}(T^\ast N, \mathrm{Hom}(L_1))$, and $R_1$ is small smoothing. Then
\[
\Op((1-r_1)q_0) \Op(a)  = \mathbbm{1} + \Op(r_2) + R_2,
\]
where $r_2 \in S^{-2}(T^\ast N,\mathcal{L}(L_1))$ and $R_2$ is again small smoothing. Now, we can iterate this construction, and find by Borel's Lemma a symbol $\tilde{q} \in S^{-m}(T^\ast N,\mathrm{Hom}(L_2,L_1))$ with
\[
\tilde{q} \sim q_0 - \sum_{i = 1}^{+\infty} r_i q_0,
\]
where $r_i \in S^{-i}(T^\ast N, \mathrm{Hom}(L_1))$, in the sense that for all $N \geq 0$,
\[
\tilde{q} - \left(q_0 - \sum_{i= 1}^N r_i q_0\right) \in S^{-m-N-1}(T^*N , \mathrm{Hom}(L_2,L_1)),
\]
and such that
\[
\Op(q)\Op(a) = \mathbbm{1} + R,
\]
with $R$ small smoothing.
\end{proof}

\subsection{Comparison to the fibered-cusp calculus}
\label{sec:fibred-cusp}

To study Fredholm properties of differential operators on ends of the type \eqref{eq:trivial-fibred-end}, the so-called \emph{fibered-cusp calculus} was introduced by Mazzeo and Melrose in \cite{Mazzeo-Melrose-98}. We now explain the main difference between their calculus and ours, and some reasons why using ours is relevant.

So far, we have presented an algebra of pseudo-differential operators which is an extension of an algebra of \emph{differential} operators. The latter is itself the algebra generated by $\mathcal{V}_0$, the Lie algebra of vector fields of the form
\begin{equation}\label{eq:formV_0}
a y\partial_y + by\partial_\theta + X,
\end{equation}
where the coefficients $a,b$ are $C^\infty$-bounded on $Z\times F$ and $X$ is a smooth and $C^\infty$-bounded vector field tangent to $F$. Here $C^\infty$-bounded is understood with respect to the Levi-Civita connection, or the local vector fields $y\partial_y$, $y\partial_\theta$ and $\partial_\zeta$. A crucial observation is that the Laplacian associated with the metric of $Z\times F$ is in this algebra.

Let us recall on the other hand the setup of the \emph{fibred-cusp calculus} developed by Mazzeo-Melrose \cite{Mazzeo-Melrose-98}. Usually, it is presented as a calculus on a manifold with boundary. However, we can remove the boundary, and replace the distance to the boundary $u$ by $y=1/u$ as local coordinates. After such a change of coordinates, the fibered-cusp algebra $\Psi^{\mathrm{diff}}_{\mathrm{fc}}$ is seen to be the algebra of differential operators generated by the algebra $\mathcal{V}_{\mathrm{fc}}$ of vector fields of the form (in local coordinates $(y,\theta,\zeta)$ near the boundary)
\begin{equation}\label{eq:formV_fc}
\frac{1}{y}\left( a y \partial_y + b y \partial_\theta+ c  \partial_\zeta \right),
\end{equation}
where $a,b,c$ are $C^\infty$-bounded functions of $u=1/y,\zeta,\theta$ (including at $u=0$). At first glance, it seems that $\mathcal{V}_{\mathrm{fc}}$ is the same algebra as $(1/y) \mathcal{V}_0$. However, the regularity required in the Mazzeo-Melrose calculus is quite different. Indeed, the vector fields measuring the regularity of the coefficients are now $y^2 \partial_y$ ($\partial_u$ in the $u$-coordinate), $\partial_\theta$ and $\partial_\zeta$ (observe how they are \emph{not} in $\mathcal{V}_{\mathrm{fc}}$). As a consequence, the inclusion $\mathcal{V}_{\mathrm{fc}} \subset (1/y)\mathcal{V}_0$ is strict. For example, $\sin(\log y) y \partial_y$ is an element of $\mathcal{V}_0$, but $\sin(\log y) \partial_y$ is not in  $\mathcal{V}_{\mathrm{fc}}$. The reason is that it oscillates with a wavelength $\sim 1$ uniformly in the cusp, but in the fibred-cusp geometry of Mazzeo-Melrose, it oscillates at wavelength $|\log u|$, where $u$ is the distance to the boundary.

Our algebra of operators is thus a priori larger than the fibred-cusp algebra. However, it can only be defined when working with hyperbolic cusps, which is not the case of Mazzeo and Melrose's algebra. To explain this, we have to expand a bit on the geometric difference behind the definition of these algebra of vector fields. The very fact that these \emph{form} an algebra is a delicate point. 

Let us assume for a moment that the $\theta$ variable is not a variable in a torus, but rather an arbitrary compact manifold $W$, so that the end takes the form $\R^+ \times W \times F$. We want to consider then vector fields of the form \eqref{eq:formV_0}, and impose some condition on the coefficients for them to form an algebra. For this, we consider $y X_1$ and $y X_2$, directed along $\partial_\theta$. Then $X_1 = b_1 \partial_\theta$, and $X_2 = b_2 \partial_\theta$ so that
\[
[yX_1, yX_2] = y^2[X_1,X_2].
\]
This has to be of order $\mathcal{O}(y)$, so that $[X_1, X_2] =\mathcal{O}(1/y)$. This means that the vector fields $X_{1,2}$, on $W$ with parameters $y, \zeta$, are almost commuting as $y\to +\infty$. Let us further impose the non-degeneracy condition that for each point in $W$ we can find vector fields $y X_1,\dots, yX_d$ in our putative algebra which form a bounded frame around that point. Heuristically, when iterating these brackets, we find that in the region where they are non-degenerate, they are almost constant as $y\to +\infty$. Eventually, this implies that they have to be supported on the whole of $W$, and almost commute. Declaring them as an orthonormal basis defines then a metric on $W$ whose curvature goes to $0$ as $y\to + \infty$. It follows that $W$ is almost flat. Manifolds that are almost flat are classified, and are all bundles constructions with torii. As a side remark here, if we considered complex (instead of real) hyperbolic cusps, we would have a $\theta$ variable describing (instead of a torus) a torus bundle over a torus. For more insight on this type of question, we refer to \cite{Gromov-78} and \cite{Ballmann-Bruning-Carron-12}.

On the other hand,the set of vector fields $\mathcal{V}_{\mathrm{fc}}$ forms an algebra without any particular condition on the topology of the manifold $W$. Another way of putting this is that the description \eqref{eq:formV_fc} is completely local at the boundary, while the description \eqref{eq:formV_0} is semi-local, since it is actually a global description the $\theta$ variable.

Let us now explain why in our case, it is worthwhile to use our algebra instead of the fibred-cusp one. The arguments are of different nature. 

The purpose of \cite{Mazzeo-Melrose-98} was to analyze whether operators in $\Psi^{\mathrm{diff}}_{\mathrm{fc}}$ have parametrices modulo compact remainders when acting on $L^2(N')$. This involves the inversion of an \emph{indicial operator}, which is a family of operators $\hat{P}(\zeta,\eta)$, parametrized by $(\zeta,\eta)\in T^\ast F$, acting on the fiber $p^{-1}(\zeta)$ (here $\R^d/\Lambda$). To invert general elliptic operators in our class would certainly involve something more complicated. 

However, we will restrict our attention to operators which respect the geometry of the cusp. For such operators, we will be able to concentrate mostly on functions which do not depend on the variable $\theta$. For this reason, the non-compactness of the manifold will effectively appear to take the form $\R^+\times F$. The criterion for being Fredholm is then much simpler. Indeed, we only need to invert a family of operators $I(P,\lambda)$, with $\lambda \in i\R$, each such operator acting on $F$ (the base instead of the fiber). One can note here that this is very similar to another calculus of Melrose, the b-calculus, which is adapted to cylindrical ends like $\R^+\times F$.

Another argument of convenience is that the theory of Mazzeo and Melrose has so far been developped entirely in the framework of Sobolev spaces. Extending those results to the class of H\"older-Zygmund spaces would require a considerable effort, more so than developping our tools. 

The other arguments are a bit less practical, and related to the question of studying dynamics. If $P$ is a differential operator of order $m$ in our class satisfying the ``geometric condition'' yet to be stated, $u^m P\in \Psi^{\mathrm{diff}}_{\mathrm{fc}}$, so one could apply the results in \cite{Mazzeo-Melrose-98}. However
\begin{itemize}
	\item It is a bit unsatisfying that in the fibred-cusp calculus, the most natural operator in our setting, i.e the Laplace operator, is not in the algebra, but can still be dealt with via some trick --- multiplication by $u^2$. 
	\item In the case that $P$ is not differential but pseudo-differential of varying order, it is not quite obvious what would replace the correspondence $P\mapsto u^m P$. This is crucial when dealing with anisotropic spaces as in \cite{Bonthonneau-Weich-17}. This will also intervene in our second paper \cite{Bonthonneau-Lefeuvre-19-2}. 
	\item For the study of propagators and propagation of singularities (which will be important in the sequel \cite{Bonthonneau-Lefeuvre-19-2}), it is important to consider the flow of the principal symbol of the relevant operators. In our framework, such a flow is uniformly smooth with respect to the geometry, while it is not the case of the fibred-cusp calculus. 
\end{itemize}

\section{Pseudo-differential operators for Hölder-Zygmund spaces on cusps}

\label{section:holder-zygmund}

H\"older spaces, with their naive definition, are not very compatible with the tools of microlocal analysis. The usual solution is to give an alternative definition of these spaces via a Littlewood-Paley decomposition, defining thus the so-called H\"older-Zygmund spaces. On $\R^n$, or on compact manifolds, this is a classical subject. However, this has not been investigated so far for manifolds with cusps. 

In this section, we are going to define such spaces $C^s_*$ on manifolds with cusps, and then prove that the class of pseudo-differential operators defined in the previous section acts on them in a reasonnable fashion. On a compact manifold, this is a well-known fact and we refer to the arguments before \cite[Equation (8.22)]{Taylor-III-96edition} for more details. The subtleties come from the non-compactness of the manifold.

At this stage, we insist on the fact that the Euclidean Littlewood-Paley decomposition is rather remarkable insofar as it only involves Fourier multipliers (and not ``really'' pseudo-differential operators), which truly simplify all the computations. This is not the case in the hyperbolic world and some rather tedious integrals have to be estimated.

Then, we will be able to prove that the previously defined pseudo-differential operators of order $m \in \R$ map continuously $C_*^{s+m}$ to $C^s_*$, just as in the compact setting. Since we can always split the operator in different parts that are properly supported in cusps or in a fixed compact subset of the manifold (modulo a smoothing operator), we can directly restrict ourselves to operators supported in a cusp as long as we know that smoothing operators enjoy the boundedness property. 

\subsection{Definitions and properties}

In the paper \cite{Bonthonneau-16}, only Sobolev spaces were considered. So we will have to prove several basic results of boundedness of the calculus, acting now on Hölder-Zygmund spaces. We will give the proofs in the case of cusps, and leave the details of extending to products of cusps with compact manifolds to the reader.

We consider a smooth non-negative function $\psi \in C^\infty_{\mathrm{comp}}(\R)$ such that $\psi(s) = 1$ for $|s| \leq 1$ and $\psi(s) = 0$ for $|s| \geq 2$. We define, for $j \in \N^\ast$, the following function on the cotangent bundle to the hyperbolic space $T^* \HH^{d+1} \simeq \HH^{d+1}\times \R^{d+1}$:
\begin{equation}
\label{equation:definition-phi-j}
\varphi_j(x,\xi) = \psi(2^{-j}\langle \xi \rangle)-\psi(2^{-j+1}\langle \xi \rangle),
\end{equation}
where $(x,\xi) \in T^*\HH^{d+1}$, $x = (y,\theta) \in \HH^{d+1}$ and $\langle \xi \rangle := \sqrt{1 + y^2|\xi|^2}$ is the hyperbolic Japanese bracket with $|\xi| = |\xi|_{\mathrm{euc}}$ being the euclidean norm of the (co)vector $\xi \in \R^{d+1}$. Observe that 
\[
\supp \varphi_j \subset \left\{ (x,\xi) \in \HH^{d+1} \times \R^{d+1} ~|~ 2^{j-1} \leq \langle \xi \rangle \leq 2^{j+1}\right\}.
\]
Then, with $\varphi_0 = \psi(\langle \xi\rangle)$, $\sum_{j=0}^{+\infty} \varphi_j(x,\xi) = 1$. Of course, the functions $\varphi_j$ are translation invariant and thus descend to a cusp $Z = ]0,+\infty[ \times \R^d/\Lambda$. We will still denote them by $\varphi_j$. We then introduce the 
\begin{definition}
We define the \emph{Hölder-Zygmund space of order $s$} as:
\[
C^s_*(Z) := \left\{ u \in \Delta^N L^\infty(Z)+L^\infty(Z)\ |\ \|u\|_{C^s_*} < \infty \right\},
\]
where:
\[ 
\|u\|_{C^s_*} := \sup_{j \in \N} 2^{js} \|\Op(\varphi_j)u\|_{L^\infty(Z)} 
\]
and $N =0$ for $s>0$ and $N> (|s|+d+1)/2$ when $s\leq 0$.
\end{definition}

It will be shown in Lemma \ref{lemma:boundedness-decaying-symbols} below that $\Op(\varphi_j)$ is indeed bounded on $L^\infty$, hence proving that this definition is legitimate. One can check that the definition of these spaces does not depend on the choice of the initial function $\psi$ (as long as it satisfies the aforementioned properties), see Remark \ref{remark:paley-independence}. This mainly follows from Lemma \ref{lemma:paley-control} below. Note that, although a cutoff function $\chi$ around the ``diagonal'' $y=y'$ has been introduced in (\ref{equation:cutoff}) in the quantization $\Op$, we still have $\mathbbm{1} = \sum_{j \in \N} \Op(\varphi_j)$. Thus, given $u \in C^s_*$ with $s>0$, one has $u = \sum_{j \in \N} \Op(\varphi_j) u$, with normal convergence in $L^\infty$ and
\[ \|u\|_{L^\infty} \leq \sum_{j \in \N} \|\Op(\varphi_j) u\|_{L^\infty} \leq \sum_{j \in \N} 2^{-js} \underbrace{2^{js} \|\Op(\varphi_j) u\|_{L^\infty}}_{\leq \|C^s_*\|} \lesssim \|u\|_{C^s_*}. \]

We denote by $[x]$ the floor function, i.e. the integer part of $x \in \R$. It can be checked that the definition of Hölder-Zygmund space locally coincides with the usual definition on a compact manifold, that is for\footnote{For $s \in \N$, this does not exactly coincide with the set of functions that have exactly $[s]$ derivatives in $L^\infty$.} $s \notin \N$, $C^s_*$ contains the functions that have $[s]$ derivatives which are locally $L^\infty$ and such that the $[s]$-th derivatives are $s-[s]$ Hölder continuous. Indeed, if we choose a function $f$ that is localized in a strip $y \in [a,b]$, then the size of the annulus in the Paley-Littlewood decomposition is uniform in $y$ and can be estimated in terms of $a$ and $b$, so the definition of the Hölder-Zygmund spaces boils down to that of $\R^{d+1}$. This will be made precise in Proposition \ref{proposition:correspondence-espaces-holder}.

We have the equivalent of Proposition \ref{prop:microlocal-calculus}:

\begin{proposition}
\label{proposition:pseudo-borne-cs}
Let $P = \Op(\sigma)$ with $\sigma\in S^m(T^*N,\mathrm{Hom}(L_1,L_2))$ be a pseudo-differential operator in the class $\Psi^m_{\text{small}}(N, L_1 \rightarrow L_2)$. Then:
\[
P : y^\rho C_*^{s+m}(N,L_1) \rightarrow y^\rho C_*^s(N,L_2),
\]
is bounded for all $s \in \R$. If $\sigma'\in S^{m'}(T^*N,\mathrm{Hom}(L_2,L_3))$ is another symbol,
\[
\Op(\sigma) \Op(\sigma') - \Op(\sigma\sigma') \in \Psi^{m+m'-1}_{\text{small}}.
\]
\end{proposition}

We make the following important remark. By definition, an arbitrary small pseudo-differential operator $P \in \Psi^m_{\text{small}}(N, L_1 \rightarrow L_2)$ can be written as $P = \Op(\sigma)+R$, where $\sigma \in S^m(T^*N,\mathrm{Hom}(L_1,L_2))$ and $R$ is $\R$-smoothing i.e. it maps $y^\rho H^{-s} \rightarrow y^\rho H^s$ (for any $\rho, s \in \R$). The previous Proposition \ref{proposition:pseudo-borne-cs} shows that the first term $\Op(\sigma)$ is bounded on the Hölder-Zygmund spaces $y^\rho C_*{s+m} \rightarrow y^\rho C_*^s$. However, for the second term $R$, this is not clear. In the following, we will mostly be interested by \emph{admissible operators} (see Definition \ref{def:R-L^2-admissible-operator}) and for this restricted class of operators, we will show that the remainder is indeed bounded on Hölder-Zygmund spaces (see Proposition \ref{prop:equivalenceL2Linfty-admissible}).

As usual, since we added a cutoff function on the kernel of the operator around the diagonal $y=y'$, the statement boils down to $\rho=0$, which we are going to prove in the next paragraph.

\subsection{Basic boundedness}

The first step here is to derive a bound on $L^\infty(Z)$ spaces. We follow the notations in \cite{Bonthonneau-16}, $\tilde{f}$ denoting the lifting of a function $f$ on $Z$ to $\Lambda$-periodic functions in $\HH^{d+1}$. If $f$ is a function on the full cusp $Z$, then for $P = \Op(\sigma)$ with $\sigma \in S^m(T^*Z)$, one has:
\[
Pf(x) = \int_{\HH^{d+1}} \chi(y'/y-1) \left(\dfrac{y}{y'}\right)^{\frac{d+1}{2}} K^w_{\sigma}(y,\theta,y',\theta')\tilde{f}(y',\theta') \dd y' \dd \theta',
\]
where the kernel $K^w_\sigma$ can be written:
\[ 
K^w_{\sigma}(x,x') = \int_{\R^{d+1}} e^{i \langle x-x',\xi \rangle} \sigma\left( \dfrac{x+x'}{2}, \xi \right) \dd\xi,
\]
with $x=(y,\theta)$. If $P : L^\infty(Z) \rightarrow L^\infty(Z)$ is bounded, then:
\begin{equation}
\label{equation:schur}
\begin{split}
\|P\|_{\mc{L}(L^\infty,L^\infty)} & \leq \sup_{(y,\theta) \in \HH^{d+1}} \int_{\HH^{d+1}} \chi(y'/y-1) \left(\dfrac{y}{y'}\right)^{\frac{d+1}{2}} |K^w_{\sigma}(y,\theta,y',\theta')| \dd y' \dd\theta' \\
& \lesssim \sup_{(y,\theta) \in \HH^{d+1}} \int_{y'=y/C}^{y'=Cy} \int_{\theta' \in \R^d} |K^w_{\sigma}(y,\theta,y',\theta')| \dd y' \dd\theta'.
\end{split}
\end{equation}
Thus, we will look for bounds on $|K^w_{\sigma}(y,\theta,y',\theta')|$. A rather immediate computation shows that:
\begin{equation}
\label{equation:iteration-kernel}
\dfrac{x_\ell-x_\ell'}{i\frac{y+y'}{2}} K^w_\sigma = K^w_{X_{\ell} \sigma},
\end{equation}
where $x = (x_0,x_1,...,x_d) = (y,\theta)$ and $X_{0} = y^{-1} \partial_Y, X_{\ell}= y^{-1} \partial_{J_\ell}$ for $\ell =1,...,d$ and we will iterate many times this equality, denoting $X^\alpha = X_0^{\alpha_0}\dots X_d^{\alpha_d}$ for each multiindex $\alpha = (\alpha_0,...,\alpha_d) \in \N^{d+1}$. Since
\[
|K^w_{\sigma}(y,\theta,y',\theta')| \lesssim \int_{\R^{d+1}} |\sigma\left(\frac{x+x'}{2},\xi\right)|\dd \xi,
\]
we also get
\[
|K^w_{\sigma}(y,\theta,y',\theta')| \lesssim \left|\dfrac{x-x'}{y+y'}\right|^{-\alpha} \int_{\R^{d+1}} |X^\alpha \sigma| \dd\xi.
\]
\begin{lemma}\label{lemma:boundedness-decaying-symbols}
Let $\sigma \in S^{-m}(T^*Z)$ with $m> d+ 1$. Then $\Op(\sigma)$ is bounded on $L^\infty(Z)$.
\end{lemma}

\begin{proof}
Under the assumptions, $\sigma$ is integrable in $\xi$, and so are its derivatives. In particular, we get for all multiindex $\alpha$,
\[
|K^w_{\sigma}(y,\theta,y',\theta')| \lesssim \frac{C_\alpha}{(y+y')^{d+1}} \left|\frac{y+y'}{x-x'}\right|^\alpha.
\]
From this we deduce
\[
|K^w_{\sigma}(y,\theta,y',\theta')| \lesssim \frac{1}{(y+y')^{d+1}}\frac{1}{1 + \left|\frac{\theta-\theta'}{y+y'}\right|^{d+1}}
\]
and
\begin{align*}
\| \Op(\sigma)\|_{L^\infty\to L^\infty} &\lesssim \sup_{y} \int_{y/C}^{yC} \dd y' \int_{\R^d} \dd\theta \frac{1}{(y+y')^{d+1}}\frac{1}{1 + \left|\frac{\theta}{y+y'}\right|^{d+1}} \\
					&\lesssim \sup_{y} \int_{y/C}^{yC} \dd y' \frac{1}{y+y'} < \infty.
\end{align*}
\end{proof}

We now use the previous dyadic partition of unity. Given a symbol $\sigma \in S^m(T^*Z)$, we define $\sigma_j := \sigma \varphi_j \in S^{-\infty}$. Observe that
\[
P = \Op(\sigma) = \sum_{j=0}^{+\infty} \Op(\underbrace{\sigma \varphi_j}_{=\sigma_j}) = \sum_{j=0}^{+\infty} P_j,
\]
where $P_j := \Op(\sigma_j)$. We will need the following refined version of the previous lemma:
\begin{lemma}\label{lemma:paley-ordrem}
Assume that $\sigma \in S^m(T^*Z)$. Then, $\|P_j\|_{\mc{L}(L^\infty,L^\infty)} \lesssim 2^{jm}$
\end{lemma}

In particular, if $u \in L^\infty(Z)$, we find that $u \in C^0_\ast(Z)$ (but the converse is not true !).

\begin{proof}
The proof is similar to the proof of Lemma \ref{lemma:boundedness-decaying-symbols}, but we have to be careful to obtain the right bound in terms of power of $2^j$. Since $\varphi_j$ has support in $\left\{2^{j-1}\leq \langle \xi \rangle \leq 2^{j+1}\right\}$, the kernel $K^w_{\sigma_j}$ of $P_j$ satisfies:
\begin{equation}
\label{equation:estimee-kernel-1}
|K^w_{\sigma_j}(x,x')|  \lesssim \int_{\left\{2^{j-1} \leq \langle \xi \rangle \leq 2^{j+1}\right\}} \langle \xi \rangle^{m} \dd\xi \lesssim \dfrac{2^{j(m+d+1)}}{(y+y')^{d+1}}
\end{equation}
Differentiating in $\xi$, we get for all multiindex $\alpha$,
\begin{equation}
\label{equation:estimee-kernel-1bis}
|K^w_{\sigma_j}| \lesssim \left|\frac{y+y'}{x-x'}\right|^\alpha \dfrac{2^{j(m-|\alpha|+d+1)}}{(y+y')^{d+1}},
\end{equation}
Combining with (\ref{equation:iteration-kernel}) (we iterate the equality $k'$ times in $y$ and $k$ times in $\theta$ that is in each $\theta_i$ coordinate), we obtain:
\begin{equation}\label{equation:estimee-kernel-2}
 |K^w_{\sigma_j}(x,x')| \lesssim \dfrac{2^{j(m+d+1)}}{(y+y')^{d+1}\left(1 + 2^{jk'}\left|\dfrac{y-y'}{y+y'}\right|^{k'} + 2^{jk}\left|\dfrac{\theta-\theta'}{y+y'}\right|^{k}\right)}
 \end{equation}
Then, integrating in (\ref{equation:schur}), we obtain:
\[ 
\begin{split}
& \|P_j\|_{\mc{L}(L^\infty,L^\infty)} \\
& \lesssim \sup_{(y,\theta) \in \HH^{d+1}} \int_{y'=y/C}^{y'=Cy} \int_{\theta' \in \R^d} |K^w_{\sigma}(y,\theta,y',\theta')| \dd y' \dd\theta' \\
& \lesssim 2^{j(m+d+1)} \sup_{(y,\theta) \in \HH^{d+1}} \int_{y'=y/C}^{y'=Cy} \int_{\theta' \in \R^d}  \dfrac{\dd y' \dd\theta' }{(y+y')^{d+1}\left(1 + 2^{jk'}\left|\frac{y-y'}{y+y'}\right|^{k'} + 2^{jk}\left|\dfrac{\theta-\theta'}{y+y'}\right|^{k}\right)} \\
& \lesssim 2^{j(m+d+1)} \sup_{(y,\theta) \in \HH^{d+1}} 2^{-jd}  \int_{y'=y/C}^{y'=Cy} \dfrac{\dd y'}{(y+y')\left(1+2^{jk'}\left|\frac{y-y'}{y+y'}\right|^{k'}\right)^{1-d/k}} \\
& \lesssim 2^{j(m+1)} \int_{1/C}^{C} \dfrac{1}{(1+u)\left(1+2^{jk'}\left|\frac{u-1}{u+1}\right|^{k'}\right)^{1-d/k}} \dd u, 
\end{split}
\]
where we have done the change of variable $u=y'/y$. We let $v = 2^{j}\frac{1-u}{1+u}$, so that $u = (1- 2^{-j} v)/(1+ 2^{-j} v)$,
\[
1/(1+u) = (1+ 2^{-j}v)/2,\quad \dd u = - \frac{2^{1-j}}{(1+2^{-j}v)^2} \dd v.
\]
and we get the bound
\[
\begin{split}
\int_{1/C}^{C} \dfrac{1}{(1+u)\left(1+2^{jk'}\left|\frac{u-1}{u+1}\right|^{k'}\right)^{1-d/k}} \dd u 	\lesssim 2^{-j}\int_{-2^{j}(C-1)/(C+1)}^{2^j (C-1)/(C+1)} \frac{1}{(1+ |v|^{k'})^{1-d/k}}  \frac{\dd v}{1+ 2^{-j} v}.
\end{split}
\]
Let now $k=d+1$ and $k' = d+2$. We can bound the term $1/(1+2^{-j}v)$ by $(C+1)/2$, and we get
\[
\|P_j\|_{\mc{L}(L^\infty,L^\infty)} \lesssim 2^{jm} \int_{\R} \frac{ \dd v }{(1 + |v|^{d+2})^{1/(d+1)}} \lesssim 2^{jm}.
\]
Here, it was crucial that the kernel is uniformly properly supported.
\end{proof}

\begin{lemma}\label{lemma:paley-control}
Let $\sigma \in S^m(T^*Z)$. For all $N \in \N$, there exists a constant $C_N > 0$ such that for all integers $j, k \in \N$ such that $|j-k| \geq 3$,
\[
\|P_j \Op(\varphi_k)\|_{\mc{L}(L^\infty, L^\infty)}, \|\Op(\varphi_k)P_j\|_{\mc{L}(L^\infty, L^\infty)} \leq C_N 2^{-N \max(j,k)},
\]
where $P_j = \Op(\sigma \varphi_j)$.  
\end{lemma}

\begin{proof}
This is a rather tedious computation and we only give the key ingredients. It is actually harmless to assume that $\sigma = 1$, which we will assume to hold for the sake of simplicity. We use \cite[Proposition 1.19]{Bonthonneau-16}. We know that 
\[
\Op(\varphi_j) \Op(\varphi_k) f(x) = \int_{x' \in \HH^{d+1}} \left(\dfrac{y}{y'}\right)^{\frac{d+1}{2}} K^w_{\varphi_j \sharp \varphi_k}(x,x') f(x') \dd x'
\]
where, by definition,
\begin{equation}\label{equation:kernel-composee}
K^w_{\varphi_j \sharp \varphi_k}(x,x') = \int e^{i\langle x-x',\xi\rangle} \varphi_j \sharp \varphi_k\left(\frac{x+x'}{2},\xi\right) \dd\xi 
\end{equation}
and 
\begin{equation}\label{equation:symbole-sharp}
\begin{split}
\varphi_j \sharp \varphi_k(x,\xi) = 2^{-2d-2}\int & e^{2i\left(-\langle x-x_1,\xi-\xi_1 \rangle + \langle x-x_2,\xi-\xi_2\rangle\right)}  \\
& \varphi_j(x_2,\xi_1) \varphi_k(x_1,\xi_2) \chi(y,y_1,y_2) \dd x_1 \dd x_2 \dd \xi_1 \dd \xi_2,
\end{split}
\end{equation}
where, for fixed $y$, $\chi(y,\cdot,\cdot)$ is supported in the rectangle $\left\{y/C \leq y_{1,2} \leq yC \right\}$ ($C$ not depending on $y$). To prove the claimed boundedness estimate, it is thus sufficient to prove that 
\[ 
\sup_{x \in \HH^{d+1}} \int_{x' \in \HH^{d+1}} \left(\dfrac{y}{y'}\right)^{\frac{d+1}{2}} |K^w_{\varphi_j \sharp \varphi_k}(x,x')| \dd x' \lesssim C_N 2^{-N \max(j,k)},
\]
and we certainly need bounds on the kernel $K^w_{\varphi_j \sharp \varphi_k}$. First observe that it is supported in some region $\left\{ y/C' \leq y' \leq yC'\right\}$ so, as before, the term $(y/y')^{\frac{d+1}{2}}$ is harmless in the integral. Then, we follow the same strategy as in the proof of Lemma \ref{lemma:paley-ordrem}. We deduce that it suffices to obtain bounds of the form
\[
|K^w_{X^{\alpha}(\varphi_j \sharp \varphi_k)}|, |K^w_{\varphi_j \sharp \varphi_k}| \lesssim C_N\frac{ 2^{-N \max(j,k)}}{(y+y')^{d+1}}.
\]
for $|\alpha|\leq d+2$. 

For the sake of simplicity, we only deal with the bound on $|K^w_{\varphi_j \sharp \varphi_k}|$, the others being similar. To obtain a bound on this kernel, it is sufficient to prove that $|\varphi_j \sharp \varphi_k(x,\xi)| \lesssim C_N 2^{-N\max(j,k)}\langle \xi \rangle^{-N}$ (where $N$ has to be chosen large enough). Indeed, one then obtains:
\[
|K^w_{\varphi_j \sharp \varphi_k}(x,x')| \lesssim C_N 2^{-N\max(j,k)} \int_{\R^{d+1}} \dfrac{\dd \xi}{\left(1+\left(\frac{y+y'}{2}\right)^{2}|\xi|^2\right)^{N/2}} \lesssim C_N \frac{2^{-N \max(j,k)}}{(y+y')^{d+1} }.
\]
We denote by $y_1 D_{x_{1,\ell}} := \dfrac{y_1}{2i} \partial_{x_{1,\ell}}$ the operator of derivation and we use in (\ref{equation:symbole-sharp}) the identity
\begin{equation}\label{equation:identite-classique}
(1+y_1^2|\xi-\xi_1|^2)^{-N}(1+y^2_1D_{x_1}^2)^{N}(e^{2i\langle x-x_1,\xi-\xi_1\rangle}) = e^{2i\langle x-x_1,\xi-\xi_1\rangle}
\end{equation}
where $D_{x_1}^2 = \sum_\ell D_{x_{1,\ell}}^2$. In terms of Japanese bracket, this can be rewritten shortly $\langle \xi-\xi_1 \rangle^{-2N}\langle D_{x_1} \rangle^{2N}(e^{2i\langle x-x_1,\xi-\xi_1\rangle}) = e^{2i\langle x-x_1,\xi-\xi_1\rangle}
$. We thus obtain:
\[
\begin{split}
\varphi_j \sharp \varphi_k(x,\xi) =& 2^{-2d-2}  \int e^{2i\left(\langle x-x_1,\xi-\xi_1 \rangle + \langle x-x_2,\xi-\xi_2\rangle\right)}  \langle \xi-\xi_1 \rangle^{-2N} \langle \xi-\xi_2 \rangle^{-2N} \\
&   \langle D_{x_1} \rangle^{2N} \langle D_{x_2} \rangle^{2N} \left( \varphi_j(x_2,\xi_1) \varphi_k(x_1,\xi_2) \chi(y,y_1,y_2) \right) \dd x_1 \dd x_2 \dd\xi_1 \dd\xi_2,
\end{split}
\]
We also need to use this trick in the $x$ variable (more precisely on the $\theta$ variable) to ensure absolute convergence of this integral. This yields the formula:
\[
\begin{split}
 \varphi_j \sharp \varphi_k(x,\xi)& = 2^{-2d-2} \int e^{2i\left(\langle x-x_1,\xi-\xi_1 \rangle + \langle x-x_2,\xi-\xi_2\rangle\right)} \\ 
& \langle \theta-\theta_1 \rangle^{-2M} \langle \theta-\theta_2 \rangle^{-2M}  \langle D_{J_1} \rangle^{2M} \langle D_{J_2} \rangle^{2M}  \\
& \left[ \langle \xi-\xi_1 \rangle^{-2N} \langle \xi-\xi_2 \rangle^{-2N} \langle D_{x_1} \rangle^{2N} \langle D_{x_2} \rangle^{2N} (\varphi_j(x_2,\xi_1)  \varphi_k(x_1,\xi_2) \chi(y,y_1,y_2)) \right] \\
& \qquad \dd x_1 \dd x_2 \dd\xi_1 \dd\xi_2,
\end{split}
\]
where $M$ is chosen large enough. We here need to clarify a few things. First of all, the notation is a bit hazardous insofar as $\langle \theta-\theta_1 \rangle^2 := 1+ \frac{|\theta-\theta_1|^2}{y_1^2}$ this time. This comes from the fact that the natural operation of differentiation (which preserves the symbol class) is $\langle D_{J_1} \rangle^2 := 1 + \sum_{\ell=1}^d (y^{-1}_1 \partial_{J_{1,\ell}})^2$. If ones formally develops the previous formula, one obtains a large number of terms involving derivatives — coming from the brackets 
\[
\langle D_{J_1} \rangle^{2M} \langle D_{J_2} \rangle^{2M}  \langle D_{x_1} \rangle^{2N} \langle D_{x_2} \rangle^{2N}\]
— of $\varphi_j$ and $\varphi_k$. These derivatives obviously do not change the supports of these functions and can only better the estimate (there is a $2^{-j}$ that pops up out of the formula each time one differentiates, stemming from the very definition of $\varphi_j$). As a consequence, it is actually sufficient to bound the integral if one forget about these brackets of differentiation. We are thus left to bound
\[
\begin{split}
&\int e^{2i\left(\langle x-x_1,\xi-\xi_1 \rangle + \langle x-x_2,\xi-\xi_2\rangle\right)} \langle \theta-\theta_1 \rangle^{-2M} \langle \theta-\theta_2 \rangle^{-2M}    \\
& \langle \xi-\xi_1 \rangle^{-2N} \langle \xi-\xi_2 \rangle^{-2N}  \varphi_j(x_2,\xi_1)  \varphi_k(x_1,\xi_2) \chi(y,y_1,y_2)   \dd x_1 \dd x_2 \dd \xi_1 \dd \xi_2.
\end{split}
\]
We can now assume without loss of generality that $k \geq j+3$. Then, $\varphi_j$ and $\varphi_k$ are supported in two distinct annulus whose interdistance is bounded below by $2^{k-1}-2^{j+1} \geq 2^{k-2}$. Using this fact, one can bound the integrand by 
\[ 
\begin{split}
\langle \xi -\xi_1 \rangle^{-2N} & \langle \xi -\xi_2 \rangle^{-2N} \langle \theta-\theta_1 \rangle^{-2M}  \langle \theta -\theta_2 \rangle^{-2M}  \chi(y,y_1,y_2) \\
&\lesssim  C_N 2^{-N k}\langle \xi \rangle^{-4N}  \langle \theta-\theta_1 \rangle^{-2M}  \langle \theta -\theta_2 \rangle^{-2M}  \chi(y,y_1,y_2), 
\end{split}
\]
where the last bracket is $\langle \xi \rangle := \sqrt{1+y^2|\xi|^2}$. (The estimates actually come out with a Japanese bracket in terms of $y_{1,2}$ but these are uniformly comparable to the Japanese bracket in terms of $y$ because $\chi$ is supported in the region $\left\{y/C \leq y_{1,2} \leq yC\right\}$.) We thus obtain:
\[
\begin{split}
&\Bigg|\int e^{2i\left(\langle x-x_1,\xi-\xi_1 \rangle + \langle x-x_2,\xi-\xi_2\rangle\right)} \langle \theta-\theta_1 \rangle^{-2M} \langle \theta-\theta_2 \rangle^{-2M}    \\
& \langle \xi-\xi_1 \rangle^{-2N} \langle \xi-\xi_2 \rangle^{-2N}  \varphi_j(x_2,\xi_1)  \varphi_k(x_1,\xi_2) \chi(y,y_1,y_2))   \dd x_1 \dd x_2 \dd\xi_1 \dd\xi_2\Bigg| \\
& \lesssim C_N 2^{-N k}\langle \xi \rangle^{-4N}  \int_{\substack{x_1\in \R^{d+1}\\ x_2 \in \R^{d+1} \\ 2^{j-1} \leq \langle \xi_1 \rangle \leq 2^{j+1} \\ 2^{k-1} \leq \langle \xi_2 \rangle \leq 2^{k+1} }} \langle \theta-\theta_1 \rangle^{-2M}  \langle \theta -\theta_2 \rangle^{-2M}  \chi(y,y_1,y_2) \dd\xi_1 \dd\xi_2 \dd x_1 \dd x_2 
\end{split}
\]
We simply use a volume bound of the annulus (the ball in which it is contained actually) for the $\xi_1,\xi_2$ integrals which provides:
\[
\int_{2^{j-1} \leq \langle \xi_1 \rangle \leq 2^{j+1}}  \dd\xi_1 \lesssim 2^{j(d+1)} /y_1^{d+1}
\]
As a consequence, the bound in the previous integral becomes:
\[
\begin{split}
&C_N \frac{2^{-N k}}{\langle \xi \rangle^{4N}}  \int_{\substack{x_1\in \R^{d+1} \\x_2 \in \R^{d+1}  \\ 2^{j-1} \leq \langle \xi_1 \rangle \leq 2^{j+1} \\ 2^{k-1} \leq \langle \xi_2 \rangle \leq 2^{k+1} }} \frac{\chi(y,y_1,y_2) \dd \xi_1 \dd \xi_2 \dd x_1 \dd x_2}{\langle \theta-\theta_1 \rangle^{2M}  \langle \theta -\theta_2 \rangle^{2M} } \\
& \lesssim C_N \frac{2^{-N k+(j+k)(d+1)}}{\langle \xi \rangle^{4N}}  \int_{\substack{x_1\in \R^{d+1} \\ x_2 \in \R^{d+1} }} \langle \theta-\theta_1 \rangle^{-2M}  \langle \theta -\theta_2 \rangle^{-2M}  \chi(y,y_1,y_2) \dfrac{\dd x_1 \dd x_2}{y_1^{d+1}y_2^{d+1}}
\end{split}
\]
Now, the last integral can be bounded by
\[ 
\begin{split}
\int_{y_1=y/C}^{Cy}  \int_{\theta_1 \in \R^d} \int_{y_2=y/C}^{Cy} \int_{\theta_2 \in \R^d}  \langle \theta-\theta_1 \rangle^{-2M}  \langle \theta -\theta_2 \rangle^{-2M}  \dfrac{\dd x_1 \dd x_2}{y_1^{d+1}y_2^{d+1}} \lesssim 1,
\end{split}
\]
where $M$ is large enough, which eventually yields the estimate
\[
|\varphi_j \sharp \varphi_k (x,\xi)| \lesssim C_N  2^{-Nk} 2^{(j+k)(d+1)} \langle \xi \rangle^{-4N}.
\]
Since $N$ was chosen arbitrary, we can always take it large enough so that it swallows the term $2^{(j+k)(d+1)}$. In the end, concluding by symmetry of $j$ and $k$, we obtain the sought estimate
\begin{equation}
\label{equation:composition-symbole}
|\varphi_j \sharp \varphi_k (x,\xi)| \lesssim C_N  2^{-N\max(j,k)}  \langle \xi \rangle^{-N}.
\end{equation}
This implies the estimate on the kernel $K^w_{\varphi_j \sharp \varphi_k}$ and concludes the proof.
\end{proof}

\begin{remark}
\label{remark:paley-independence}
Following the same scheme of proof, one can also obtain the independence of the definition of the Hölder-Zygmund spaces with respect to the cutoff function $\psi$ chosen at the beginning. If $\widetilde{\psi} \in C^\infty_0(\R)$ is another cutoff function such that $\widetilde{\psi} \equiv 1$ on $[-a,a]$ and $\widetilde{\psi} \equiv 0$ on $\R \setminus[-b,b]$ (and $0 < a < b$), we denote by $\Op(\widetilde{\varphi}_j)$ the operators built from $\widetilde{\psi}$ like in (\ref{equation:definition-phi-j}). Then, in order to show the equivalence of the $C^s_*$- and $\widetilde{C}^s_*$-norms respectively built from $\psi$ or $\widetilde{\psi}$, one has to compute quantities like $\|\Op(\varphi_j)\Op(\widetilde{\varphi}_k)\|_{\mc{L}(L^\infty,L^\infty)}$. If $k \in \N$ is fixed, then the terms $\Op(\varphi_j)\Op(\widetilde{\varphi}_k)$ ``interact'' (in the sense that one will not be able to obtain a fast decay estimate like (\ref{equation:composition-symbole})) for $j \in [k-1+[\log_2(a)],k+1+([\log_2(b)]+1)]$. We can improperly call these terms ``diagonal terms". Note that the number of such terms is independent of both $j$ and $k$. The content of Lemma \ref{lemma:paley-control} can be interpreted by saying that when taking the same cutoff function (that is $\psi=\widetilde{\psi}$), the diagonal terms are $\left\{j,k \in \N ~|~ |j-k|\leq 2\right\}$. In the following, we will use the definition of Hölder-Zygmund spaces with the rescaled cutoff functions $\widetilde{\psi}_h := \psi(h \cdot)$. The diagonal terms are then shifted by $\log_2(h^{-1})$.
\end{remark}

A consequence of the previous Lemma \ref{lemma:paley-control} is the following estimate. Note that it is not needed for the proof of Proposition \ref{proposition:pseudo-borne-cs} but will appear shortly after when comparing the Hölder-Zygmund spaces $C^s_\ast$ with the usual spaces $C^s$.

\begin{lemma}
\label{lemma:paley-control2}
Let $P = \Op(\sigma)$ for some $\sigma \in S^m(T^*Z), m \in \R$ and let $0 < s < m$. Then, there exists a constant $C > 0$ such that for all $j \in \N$:
\[
\|P \Op(\varphi_j)\|_{\mc{L}(C^s_\ast,L^\infty)} \leq C2^{j(m-s)}
\]
\end{lemma}

\begin{proof}
This is a straightforward computation using both Lemma \ref{lemma:paley-control} (to deal with the terms $|j-k| \geq 3$ below) and Lemma \ref{lemma:paley-ordrem} (to deal with the terms $|j-k| \leq 2$):
\[
\begin{split}
\|P\Op(\varphi_j)f\|_{L^\infty} & \lesssim \sum_{k \in \N} \|P_k\Op(\varphi_j)f\|_{L^\infty} \\
& \lesssim \sum_{|k-j|\geq 3} \|P_k\Op(\varphi_j)f\|_{L^\infty} + \sum_{|k-j|\leq 2} \|P_k\Op(\varphi_j)f\|_{L^\infty} \\
& \lesssim \sum_{|k-j|\geq 3} C_N 2^{-N\max(j,k)}\|f\|_{L^\infty} + 2^{jm}\|\Op(\varphi_j)f\|_{L^\infty} \\
& \lesssim \|f\|_{L^\infty} + 2^{j(m-s)} \underbrace{2^{js}\|\Op(\varphi_j)f\|_{L^\infty}}_{\lesssim \|f\|_{C^s_\ast}} \\
& \lesssim 2^{j(m-s)} \|f\|_{C^s_\ast},
\end{split}
\]
where $N \geq 1$ is arbitrary.
\end{proof}

We can now start the proof of Proposition \ref{proposition:pseudo-borne-cs}.

\begin{proof}[Proof of Proposition \ref{proposition:pseudo-borne-cs}, case $s+m>0$, $s>0$]
We look at:
\[
\|\Op(\varphi_j)Pu\|_{L^\infty} \lesssim \sum_{|j-k| \geq 3} \|\Op(\varphi_j)P_k u\|_{L^\infty} + \|\Op(\varphi_j)\sum_{|j-k|\leq2} P_k u\|_{L^\infty} 
\]
The first term can be bounded using Lemma \ref{lemma:paley-control} and for $N \geq [s]+1$:
\[
\begin{split}
\sup_{j \in \N} 2^{js} \sum_{|j-k| \geq 3} \|\Op(\varphi_j)P_k u\|_{L^\infty}  & \leq \sup_{j \in \N} 2^{js} C_N \sum_{|j-k| \geq 3} 2^{-N \max(j,k)} \|u\|_{L^\infty} \\
& \lesssim \|u\|_{L^\infty} \lesssim \|u\|_{C^{s+m}_*}
\end{split}
\]
Concerning the second term, we use the same trick, writing $u_k := \Op(\varphi_k)u$.
\[
\begin{split}
\|\Op(\varphi_j)\sum_{|j-k|\leq 2} P_k u\|_{L^\infty} & \lesssim \|\sum_{|j-k|\leq 2} P_k u\|_{L^\infty} \\
& \lesssim \sum_{|j-k|\leq 2 } \sum_{|j-l| \geq 5} \|P_k u_l\|_{L^\infty} +  \sum_{|j-k|\leq 2 } \sum_{|j-l| \leq 4} \|P_k u_l\|_{L^\infty}
\end{split}
\]
The first term can be bounded just like before, using Lemma \ref{lemma:paley-control}. As to the second term, we use Lemma \ref{lemma:paley-ordrem}, which gives that
\[
\sup_{j \in \N} 2^{js} \sum_{|j-k|\leq 2 } \sum_{|j-l| \leq 4} \|P_k u_l\|_{L^\infty} \lesssim \sup_{j \in \N} 2^{js} 2^{jm} \sum_{|j-l|\leq4} \|\Op(\varphi_l)u\|_{L^\infty} \lesssim \|u\|_{C^{s+m}_*}
\]
Combining the previous inequalities, we obtain the desired result. Observe that the proof above also gives that for $P\in \Psi^m$, $m\in \R$,
\[
\| Pu \|_{C^{-m}_\ast} \lesssim \| u \|_{L^\infty}.
\] 
\end{proof}

Next, we want to deal with the case of negative $s$. To this end, we need to have some rough space on which 
our operators are bounded. Consider the space of distributions (for some constant $h>0$ small enough).
\[
C^{-2n}(Z) := (-h^2\Delta + \mathbbm{1})^n L^\infty(Z).
\]
equipped with the norm
\[
\|u \|_{C^{-2n}(Z)} := \inf\{ \| v \|_{L^\infty}\ |\ (-h^2\Delta + \mathbbm{1})^n v = u\}.
\]
\begin{lemma}
For $n \geq 1$ and $h$ small enough, $s>0$, and $\sigma\in S^{-2n +1-s}(T^*Z)$, $\Op(\sigma)$ is bounded on $C^{-2n}(Z)$. Also, for $n>n'$, $C^{-2n'}(Z) \subset C^{-2n}(Z)$.
\end{lemma}

\begin{proof}
First of all, we prove that $L^\infty \subset C^{-2n}$. To this effect, we consider parametrices
\[
(-h^2 \Delta + \mathbbm{1})^n\Op(q_n) = \mathbbm{1} + h^N \Op'(r_n),
\]
with $q_n$ of order $-2n$, and $r_n$ of order $-N$. Taking $N$ larger than $d+1$, by Lemma \ref{lemma:boundedness-decaying-symbols}, $\Op(r_n)$ is bounded on $L^\infty$ and $\Op(q_n)$ is bounded from $L^\infty$ to $C^{2n}_\ast\subset L^\infty$ by the previous Lemma. We get that for $v\in L^\infty$,
\[
(-h^2\Delta + \mathbbm{1})^n \underset{:=P_n}{\underbrace{\Op(q_n)(\mathbbm{1}+h^N\Op(r_n))^{-1}}} v = v,
\]
the inverse being defined by Neumann series for $h$ small enough and $P_n$ is of order $-2n$ so $P_nv \in C^{2n}_* \subset L^\infty$. The inclusion $C^{-2n'}\subset C^{-2n}$ follows decomposing $(-h^2\Delta+\mathbbm{1})^{n} = (-h^2\Delta+\mathbbm{1})^{n'}(-h^2\Delta+\mathbbm{1})^{n-n'}$.

For $f = (-h^2\Delta+\mathbbm{1})^n \widetilde{f} \in C^{-2n}$ (with $\widetilde{f} \in L^\infty$), observe that
\[
\Op(\sigma)f = \Op(\sigma)(-h^2\Delta+\mathbbm{1})^n \widetilde{f} = \Op'(\sigma_h') \widetilde{f} + (-h^2\Delta+\mathbbm{1})^n \Op(\sigma)\widetilde{f},
\]
with $\sigma \in S^{-2n+1-s}$ --- here $\Op'$ is a quantization with cutoffs around the diagonal with a larger support and $\sigma'_h \in S^{-s}$. By the last remark in the proof of the previous lemma, this is in $C^s_* + (-h^2\Delta+\mathbbm{1})^nC^{2n+s-1}_* \subset L^\infty + (-h^2\Delta+\mathbbm{1})^nL^\infty \subset C^{-2n}$.

\end{proof}

\begin{proof}[Proof of Proposition \ref{proposition:pseudo-borne-cs}, general case]
Given $p\in S^m$ and $n$, we can build parametrices
\[
(-h^2\Delta + \mathbbm{1})^k \Op(q_k) \Op(p) = \Op(p) + \Op(r_k),
\]
with $q_n \in S^{-2k}$, $r_n \in S^{-2n- d-1}$. With $k \geq n + (m+d+1)/2$, we get that for $u\in C^{-2n}$,
\[
\Op(p)u = (-h^2\Delta + \mathbbm{1})^k \Op(q_k) \Op(p) u - \Op(r_k) u \in C^{-2(n+k)} + C^{-2n} = C^{-2(n+k)}.
\]
In particular, $\Op(p)$ is continuous from $C^{-2n}$ to $C^{-4n - 2\lceil (m - d - 1)/2\rceil}$. Next, inspecting the proof of Lemma \ref{lemma:paley-control}, we find that it also applies to the spaces $C^{-2n}$. In particular, we obtain that for all $n\geq 0$, and every $s\in \R$,
\begin{equation}\label{eq:rough-bound}
\| \Op(p) u \|_{C^s_\ast} \leq C \| u \|_{C^{s+m}_\ast} +  C \| u \|_{C^{-2n}}.
\end{equation}
So far, we have proved that for $n\geq 0$, $s\in \R$, $m\in \R$, $\Op(p)$ is continuous as a map
\[
\{ u \in C^{-2n}\ |\ \|u\|_{C^{s+m}_\ast} <\infty\} \to \{ u\in C^{-4n - 2\lceil (m - d - 1)/2\rceil}\ |\ \| u\|_{C^{s}_\ast} < \infty \}.
\]
We would like to replace $-4n - 2\lceil (m - d - 1)/2\rceil$ by a number that only depends on $s$. To this end, we pick $u \in C^{-4n - 2\lceil (m - d - 1)/2\rceil}$ such that $\| u\|_{C^{s}_\ast} < \infty$. First off, if $s> 0$, then $u \in L^\infty$. So we assume that $s\leq 0$. Then for all $\epsilon>0$, using the estimate \eqref{eq:rough-bound},
\[
\| \Op( \langle \xi \rangle^{s - \epsilon}) u \|_{C^\epsilon_\ast} < \infty.
\]
Using parametrices again, we can find $r_N \in S^{- s - \epsilon}$ and $q \in S^{s+\epsilon}$ so that
\[
u = \Op(q_{s+\epsilon}) \Op(\langle \xi \rangle^{-s - \epsilon}) u + \Op(r_N)u.
\]
Since $\Op(r_N)u, \Op(\langle \xi \rangle^{-s - \epsilon}) u \in L^\infty$, we can apply the first part of the proof and obtain $u \in C^{-2 \lceil (s+\epsilon+d+1)/2\rceil}$.
\end{proof}

\subsection{Correspondence between Hölder-Zygmund spaces and usual Hölder spaces}
\label{sec:embeddings}

We prove that the Hölder-Zygmund spaces $C^s_\ast(Z)$ coincide with the usual spaces $C^s(Z)$ when $s \in \R_+ \setminus \N$.

\begin{proposition}
\label{proposition:correspondence-espaces-holder}
For all $s \in \R_+ \setminus \N$, $C^s_\ast(Z) = C^s(Z)$. More precisely, there exists a constant $C_s \geq 1$ such that:
\[
1/C_s \times \|f\|_{C^s_\ast(Z)}  \leq  \|f\|_{C^s(Z)} \leq C_s \|f\|_{C^s_\ast(Z)}.
\]
\end{proposition}

For the sake of simplicity, we prove the previous proposition in the case $s \in (0,1)$, the general case being handled in a similar fashion. This will require a preliminary

\begin{lemma}
\label{lemma:op-phi-j-constante}
For all $N \geq 0$, there exists a constant $C_N > 0$ such that for all $j \in \N$:
\[
\|\Op(\varphi_j)\mathbf{1}\|_{L^\infty} \leq C_N 2^{-jN}
\]
\end{lemma}

\begin{proof}
Let us start by giving an explicit expression:
\[
\Op(\varphi_j)\mathbf{1} = \int_{\HH^{d+1}} \int_{\R^{d+1}} \chi(y'/y-1)(y/y')^{\frac{d+1}{2}} e^{i\langle x-x',\xi\rangle} \sigma_j\left(\dfrac{y+y'}{2},\xi\right) \dd \xi \dd y' \dd\theta'.
\]
Since there is no dependence in $\theta$, we can remove $\theta$ and $J$ and get
\[
\int_{y'=0}^{+\infty} \int_{Y=-\infty}^{+\infty} \chi(y'/y-1)(y/y')^{\frac{d+1}{2}} e^{i\langle y-y',Y\rangle} \sigma_j\left(\dfrac{y+y'}{2},Y,J=0\right) \dd Y \dd y'
\]
The first (symplectic) change of variable $y' = y u, Y = Z/y$ yields:
\[
\Op(\varphi_j)\mathbf{1} = \int_{u=1/C}^{C} \left(\int_{Z = -\infty}^{+\infty} \chi(u-1)u^{-\frac{d+1}{2}} e^{i\langle 1-u,Z\rangle} \sigma_j\left(\dfrac{1+u}{2},Z\right) \dd Z\right) \dd u,
\]
where
\[
\sigma_j\left(\dfrac{1+u}{2},Z\right) =  \left[\psi\left(2^{-j}\sqrt{1+\left(\frac{1+u}{2} Z\right)^2}\right) - \psi\left(2^{-j+1}\sqrt{1+\left(\frac{1+u}{2} Z \right)^2}\right)\right]
\]
The second change of variable $Z' = \frac{1+u}{2} Z$ gives:
\[
\begin{split}
\Op(\varphi_j)\mathbf{1}  = \int_{u=1/C}^{C} \frac{2\dd u}{1+u}\int_{Z' = -\infty}^{+\infty}& \chi(u-1) u^{-\frac{d+1}{2}} e^{2i\langle \frac{1-u}{1+u},Z' \rangle} \\
			& \left(\psi(2^{-j}(1+Z'^2)^{1/2}) - \psi(2^{-j+1}(1+Z'^2)^{1/2})\right) \dd Z' ,
\end{split}
\]
and for the sake of simplicity, we now forget in this expression the term $u^{-\frac{d+1}{2}}$ (as $u$ is integrates between $1/C$ and $C$, this is uniformly bounded as well as all its derivatives). The third change of variable $h := \frac{1-u}{1+u}$ (i.e. $u = \frac{1-h}{1+h}$) leads to 
\[
\begin{split}
& \int_{\R} \left(\psi(2^{-j}(1+Z'^2)^{1/2}) - \psi(2^{-j+1}(1+Z'^2)^{1/2})\right) \left( \int_{\R} e^{2i\langle h,Z' \rangle} \chi\Big( \frac{-2h}{1+h}\Big) \frac{2 \dd h}{1+h} \right) \dd Z' \\
 & = \int_{Z' = -\infty}^{+\infty} \left(\psi(2^{-j}(1+Z'^2)^{1/2}) - \psi(2^{-j+1}(1+Z'^2)^{1/2})\right) s(Z') \dd Z',
 \end{split}
\]
for some Schwartz functions $s \in \mc{S}(\R)$ (since $h \mapsto  \chi(-2h/(1+h))$ is smooth with compact support). Then, given $N \geq 0$, there exists $C_N >0$ such that $s(Z') \leq C_N \langle Z' \rangle^{-N-2020}$ and thus, using the fact that $Z' \mapsto \psi(2^{-j}(1+Z'^2)^{1/2}) - \psi(2^{-j+1}(1+Z'^2)^{1/2}$ is supported on an annulus at height $\sim 2^j$, we obtain:
\[
\begin{split}
& \left|  \int_{Z' = -\infty}^{+\infty} \left(\psi(2^{-j}(1+Z'^2)^{1/2}) - \psi(2^{-j+1}(1+Z'^2)^{1/2})\right) s(Z') \dd Z' \right| \\
&  \leq \int_{2^{j-2} \leq |Z'| \leq 2^{j+2}} \underbrace{\langle Z' \rangle^{-N}}_{\lesssim 2^{-jN}} \langle Z' \rangle^{-2020} \dd Z' \lesssim 2^{-jN}.
\end{split}
\]

\end{proof}

We can now prove Proposition \ref{proposition:correspondence-espaces-holder}:

\begin{proof}
We first prove that there exists $C > 0$ such that for all functions $f \in C^s_\ast$, $\|f\|_{C^s} \leq C\|f\|_{C^s_\ast}$. For $x,x' \in Z$ such that $d(x,x') \leq 1$, we write:
\[
\begin{split}
|f(x)-f(x')| & = \left| \sum_{j \in \N} \left(\Op(\varphi_j)f\right)(x) - \left(\Op(\varphi_j)f\right)(x') \right|  \\
& \leq \sum_{j \in \N} \left|\left(\Op(\varphi_j)f\right)(x) - \left(\Op(\varphi_j)f\right)(x')\right|
\end{split}
\]
Let $N \in \N \setminus \left\{0\right\}$ be the unique integer such that $2^{-N} \leq d(x,x') \leq 2^{-N+1}$. We split the previous sum between $j \geq N$ and $j < N$. First:
\[
\begin{split}
\sum_{j \geq N} \left|\left(\Op(\varphi_j)f\right)(x) - \left(\Op(\varphi_j)f\right)(x')\right| & \lesssim \sum_{j \in \N} \|\Op(\varphi_j)f\|_{L^\infty} \\
& \lesssim \sum_{j \geq N} 2^{-js} \|f\|_{C^s_\ast} \\
& \lesssim 2^{-sN} \|f\|_{C^s_\ast} \lesssim \|f\|_{C^s_\ast} d(x,x')^s
\end{split}
\]
Now, using Lemma \ref{lemma:paley-control2} with $P=\nabla$ (note that $0 < s < m=1$), one has:
\[
\begin{split}
\sum_{j < N} \left|\left(\Op(\varphi_j)f\right)(x) - \left(\Op(\varphi_j)f\right)(x')\right| & \lesssim \sum_{j < N} \|\nabla \Op(\varphi_j) f\|_{L^\infty} d(x,x') \\
& \lesssim 2^{N(1-s)} \|f\|_{C^s_\ast} d(x,x') \lesssim \|f\|_{C^s_\ast} d(x,x')^s 
\end{split}
\]
Eventually, using the obvious estimate $\|f\|_{L^\infty} \lesssim \|f\|_{C^s_\ast}$, one obtains $\|f\|_{C^s} \lesssim \|f\|_{C^s_\ast}$.

Let us now prove the other estimate. We start with:
\[ 
\begin{split}
\Op(\varphi_j)f(x) & = \int_{\HH^{d+1}} \chi(y'/y-1)(y/y')^{\frac{d+1}{2}} K^w_{\varphi_j}(x,x')f(x')\dd x' \\
& = \int_{\HH^{d+1}} \chi(y'/y-1)(y/y')^{\frac{d+1}{2}} K^w_{\varphi_j}(x,x')(f(x')-f(x))\dd x' \\
& \qquad + f(x) \Op(\varphi_j)\mathbf{1}
\end{split}
\]
According to Lemma \ref{lemma:op-phi-j-constante} (with $N=1$), the last term is bounded by $\lesssim \|f\|_{L^\infty} 2^{-j} \lesssim \|f\|_{C^s}2^{-j}$. As to the first term, using the Hölder property of $f$:
\[
\begin{split}
& \left| \int_{\HH^{d+1}} \chi(y'/y-1)(y/y')^{\frac{d+1}{2}} K^w_{\varphi_j}(x,x')(f(x')-f(x))\dd x' \right| \\
& \lesssim \int_{\HH^{d+1}} \chi(y'/y-1)(y/y')^{\frac{d+1}{2}} \left|K^w_{\varphi_j}(x,x')\right| d(x,x')^s \dd x' \|f\|_{C^s}
\end{split}
\]
Now, following the exact same arguments as the ones developed in Lemma \ref{lemma:paley-ordrem} and using the crucial fact that on the support of the kernel of the pseudo-differential operator (namely for $y' \in [y/C,yC]$) one can bound the distance $d(x,x') \lesssim |\log(y/y')| + \frac{|\theta-\theta'|}{y}$, one can prove the estimate
\[
\sup_{x \in \HH^{d+1}} \int_{\HH^{d+1}} \chi(y'/y-1)(y/y')^{\frac{d+1}{2}} \left|K^w_{\varphi_j}(x,x')\right| d(x,x')^s \dd x' \lesssim 2^{-js}
\]
The sought estimate $\|f\|_{C^s_\ast} \lesssim \|f\|_{C^s}$ then follows immediately.

\end{proof}

\subsection{Embedding estimates}

Using the Paley-Littlewood decompositions in the cusps, we are going to prove the embedding estimates. We can actually strengthen them to the following two Lemmas:

\begin{lemma}
\label{lemma:embedding-holder-to-sobolev-spaces}
For all $s,s' \in \R$ such that $s' > s$, $\rho, \rho' \in \R$ such that $\rho' > \rho-d/2$,
\[
y^\rho C_*^{s'}(N,L) \hookrightarrow y^{\rho'} H^s(N,L)
\]
is a continuous embedding. 
\end{lemma}

\begin{lemma}
\label{lemma:embedding-hs-in-ck}
For all $s, \rho \in \R$, 
\[
y^\rho H^{s}(N,L) \hookrightarrow y^{\rho+d/2}C_*^{s-(d+1)/2}(N,L)
\]
is a continuous embeddings.
\end{lemma}

Observe that the two lemmas are locally true so that it it is sufficient to prove them when the function is supported on a single fibered cusp. The key lemma here is the following

\begin{lemma}
For all $s \in \R$,
\[
\|u\|^2_{H^s(Z)} \asymp \sum_{j \in \N} \|\Op(\varphi_j)u\|^2_{L^2(Z)} 4^{js}
\]
\end{lemma}

\begin{proof}
The proof is done using semiclassical estimates and then concluding by equivalence of norms when $h$ is bounded away from $0$. For $h > 0$, we start from
\[
\begin{split}
&\| u\|^2_{H_h^s(N)} 	\asymp  \| \Op_h(\langle \xi\rangle^s) u\|_{L^2}^2 \\
					&\ = \sum_{j,k} \langle \Op_h( \langle \xi\rangle^{s} \varphi_{j}) u, \Op_h( \langle \xi\rangle^{s} \varphi_{k}) u\rangle \\
					&\ = \sum_{|j-k|\leq 2} \langle \Op_h( \langle \xi\rangle^{s} \varphi_{j}) u, \Op_h( \langle \xi\rangle^{s} \varphi_{k}) u\rangle + \sum_{|j-k| \geq 3} \langle \Op_h( \langle \xi\rangle^{s} \varphi_{k}) \Op_h( \langle \xi\rangle^{s} \varphi_{j}) u, u\rangle.
\end{split}
\]
The first term is obviously controled by $ \sum_j \|\Op_h(\langle \xi \rangle^s\varphi_j)u\|^2_{L^2(Z)}$. To bound the last term we can first use the estimate (\ref{equation:composition-symbole}) in the proof of Lemma \ref{lemma:paley-control} which yields 
\[
\langle \Op_h( \langle \xi\rangle^{s} \varphi_{k}) \Op_h( \langle \xi\rangle^{s} \varphi_{j}) u, u\rangle  \leq C_{N} 2^{-N \max(j,k)} \|u\|_{H^{-N}_h}^2,
\]
where $N > |s|$ is taken arbitrary large and thus 
\[
\sum_{|j-k| \geq 3} \langle \Op_h( \langle \xi\rangle^{s} \varphi_{k}) \Op_h( \langle \xi\rangle^{s} \varphi_{j}) u, u\rangle \lesssim \|u\|^2_{H_h^{-N}(Z)}.
\]
Now, we also have that
\[
\begin{split}
\|u\|^2_{H_h^{-N}} & = \|\Op_h(\langle \xi \rangle^{-N}) u\|^2_{L^2} \\
& = \|\sum_j \Op_h(\langle \xi \rangle^{-N} \varphi_j ) u\|^2_{L^2} \\
& \lesssim \sum_j 2^{-j} \|\Op_h(2^j \langle \xi \rangle^{-N-s}\langle \xi\rangle^{s}  \varphi_j ) u\|^2_{L^2} \\
& \lesssim \sum_j 2^{-j} (\|\Op_h(\langle \xi\rangle^{s}  \varphi_j)u\|^2_{L^2} + h\|u\|^2_{H^{-N}_h(Z)}) \\
& \lesssim \sum_j \|\Op_h(\langle \xi\rangle^{s}  \varphi_j)u\|^2_{L^2} + h\|u\|^2_{H^{-N}_h(Z)},
\end{split}
\]
where the penultimate inequality follows from Gårding's inequality \cite[Lemma A.15]{Bonthonneau-Weich-17} for symbols of order $-(2N-1)$ since $2^j \langle \xi \rangle^{-N-s}\langle \xi\rangle^{s}  \varphi_j \in S^{-(2N-1)}$ is controlled by $\lesssim \langle \xi\rangle^{s}  \varphi_j$. For $h$ small enough, we can swallow the term $h\|u\|^2_{H^{-N}_h(Z)}$ in the left-hand side and we eventually obtain that $\|u\|^2_{H^s_h} \lesssim \sum_j \|\Op_h(\langle \xi\rangle^{s}  \varphi_j)u\|^2_{L^2}$,
where the constant hidden in the $\lesssim$ notation is independent of $h$. Actually, since $\langle \xi \rangle^s \varphi_j \lesssim 2^{js} \varphi_j$, the same arguments involving Gärding's inequality also yield
\[
\|u\|^2_{H^s_h} \lesssim \sum_j \|2^{js} \Op_h(\varphi_j)u\|^2_{L^2}.
\]

On the other hand,
\[
\sum_j \|\Op_h(\langle \xi\rangle^{s}\varphi_j)u\|^2_{L^2(Z)} = \Big \langle \sum_j \Op_h(\langle \xi\rangle^{s}\varphi_j)^2 u, u \Big \rangle. 
\]
Using expansions for products, we find that this is $\lesssim \langle \Op_h( \langle \xi\rangle^{2s}\sum \varphi_j^2) u, u\rangle$. This itself is controlled by the $H_h^{s}$ norm. Eventually, we conclude by equivalence of norms when $h$ is bounded away from $0$ (see Remark \ref{remark:paley-independence}).
\end{proof}

We will also need the following observation: $\Op(\varphi_j)(y^\rho u) = y^\rho \Op'(\varphi_j)(u)$ for some other quantization $\Op'$ (the cutoff function $\chi(y'/y-1)$ in the quantization $\Op$ is changed to $(y'/y)^\rho \chi(y'/y-1)$). In the following proof, we will denote by $\Op'$ and $\Op''$ other quantizations than $\Op$ which are produced by multiplying the cutoff function $\chi$ by some power of $y'/y$. Eventually, one last remark is that Proposition \ref{proposition:correspondence-espaces-holder} implies in particular that the spaces $C^s_*(N)$ defined for $s \in \R_+ \setminus \N$ do not depend on the choice of the cutoff function $\chi$ in the quantization (insofar as they can be identified to the usual Hölder spaces $C^s(N)$).

\begin{proof}[Proof of Lemma \ref{lemma:embedding-holder-to-sobolev-spaces}]
We fix $\rho < \rho'+d/2$ and $\eps > 0$ small enough so that $\rho < \rho'+d/2-\eps$. Then:
\[
\begin{split}
\|u\|^2_{y^{\rho'}H^s} & = \sum_{j \in \N} \|\Op(\varphi_j)(y^{-\rho'}u)\|_{L^2}^2 4^{js} \\
& \lesssim \sum_{j \in \N} \|y^{-\rho'}\Op'(\varphi_j)u\|_{L^2}^2 4^{js} 
\\
& \lesssim \sum_{j \in \N} \|y^{-\rho'-d/2+\eps}\Op'(\varphi_j)u\|_{L^\infty}^2 4^{js} \\
& \lesssim \sum_{j \in \N} \|\Op''(\varphi_j)(y^{-\rho'-d/2+\eps}u)\|_{L^\infty}^2 4^{js} \\
& \lesssim \sum_{j \in \N} 4^{j(s-s')} \underbrace{\|\Op''(\varphi_j)(y^{-\rho'-d/2+\eps}u)\|_{L^\infty}^2 4^{js'}}_{\leq\|u\|^2_{y^{\rho'+d/2-\eps}C^{s'}_*}} \lesssim \|u\|^2_{y^{\rho'+d/2-\eps}C^{s'}_*} \lesssim \|u\|^2_{y^\rho C^{s'}_*},
\end{split}
\]
since $s  < s'$.
\end{proof}

\begin{proof}[Proof of Lemma \ref{lemma:embedding-hs-in-ck}]
Let us sketch the proof for the embedding $y^{-d/2} H^{(d+1+\eps)/2} \hookrightarrow C^0$, the general case being handled in the same fashion with a little bit more work. We start by computing a $L^1(Z)\to L^\infty(Z)$ norm for $\Op(\sigma)$ when $\sigma \in S^{-(d+1+\epsilon)}(T^*Z)$. We find
\[
\| \Op(\sigma)\|_{y^\rho L^1 \to L^\infty}^2 \leq \sup_{x,x'} y^{d+1} {y'}^{2\rho}\left|\sum_{\gamma\in \Lambda} K_\sigma^w(x,y',\theta'+\gamma) \right|.
\] 
Going through the arguments of proof for equation \eqref{equation:estimee-kernel-2}, we deduce that 
\[
|K_\sigma^w(x,y',\theta')| \leq \left[(y+y')^{d+1}\left(1 + \left|\dfrac{y-y'}{y+y'}\right|^{k'} + \left|\dfrac{\theta-\theta'}{y+y'}\right|^{k}\right)\right]^{-1}.
\]
As a consequence, we have to estimate:
\begin{align*}
\sum_{\gamma\in\Lambda} |K_{\sigma}^w(x,y',\theta'+\gamma)| &\leq \sum_{\gamma\in \Lambda} \left[(y+y')^{d+1}\left(1 +\left|\dfrac{y-y'}{y+y'}\right|^{k'} +\left|\dfrac{\theta-\theta'+\gamma}{y+y'}\right|^{k}\right)\right]^{-1} \\
					& \leq \left[(y+y')^{d+1}\left(1 +\left|\dfrac{y-y'}{y+y'}\right|^{k'}\right)\right]^{-1} \sum_{\gamma\in\Lambda} \left[1 + \frac{\left|\frac{\theta-\theta'+\gamma}{y+y'}\right|^{k} }{1 +\left|\frac{y-y'}{y+y'}\right|^{k'} }\right]^{-1}\\
					\intertext{Since $y+y'>a$ the function in the sum has bounded variation, so we can apply a series-integral comparison, and replace it by the integral.}
					& \leq \frac{  C(y+y')^d}{(y+y')^{d+1}\left(1 + \left|\dfrac{y-y'}{y+y'}\right|^{k'}\right)}{\displaystyle\int_{\gamma \in \R^d}} \left[1 + \frac{\left|\frac{\theta-\theta'}{y+y'}+|\gamma|\right|^{k} }{1 + \left|\frac{y-y'}{y+y'}\right|^{k'} }\right]^{-1}\\
					& \leq \frac{1 }{y+y'}\left(1 + \left|\dfrac{y-y'}{y+y'}\right|^{k'}\right)^{-1+d/k}.
\end{align*}
We deduce that
\[
\| \Op(\varphi_j)\|_{y^\rho L^1 \to L^\infty}^2 \leq \sup_{x,x'} y^{d+1}{y'}^{\rho} \left[(y+y')\left(1 + \left|\dfrac{y-y'}{y+y'}\right|^{k'}\right)^{1-d/k}\right]^{-1}.
\]
This is bounded for $\rho=-d$. We conclude that $\Op(\sigma)$ is bounded from $y^{-d} L^1$ to $L^\infty$. Now, we recall that for $h>0$ small enough, $(-\Delta + h^{-2})^{-(d+1+\epsilon)/2} = \Op( \sigma_{d+1+\epsilon}) + R$, with $R$ smoothing, and $\sigma_{d+1+\epsilon} \in S^{-d-1-\epsilon}$. For $f \in y^{-d} W^{d+1+\epsilon,1}$, writing
\[
f = (-\Delta + h^{-2})^{-(d+1+\epsilon)/2}\underbrace{(-\Delta + h^{-2})^{+(d+1+\epsilon)/2}f}_{\in y^{-d}L^1},
\]
we deduce that $y^{-d} W^{d+1+\epsilon,1}\hookrightarrow C^0$ for $\epsilon>0$. By interpolation, we then deduce that $y^{-d/2} W^{(d+1+\epsilon)/2, 2} = y^{-d/2} H^{(d+1+\eps)/2} \hookrightarrow C^0$.
\end{proof}

To close this section, we make the important observation that $\R$-residual operators defined in terms of Hölder-Zygmund spaces $C^s_\ast$ are actually the same as the ones defined in terms of Sobolev spaces. More precisely:

\begin{proposition}\label{prop:maximallyresidualholder}
The set of operators $R$ which are bounded as
\[
R  : y^\rho C^{-s}_\ast \to y^{-\rho} C^s_\ast
\]
for all $\rho\in \R$, $s\in \R$ is equal to $\dot{\Psi}^{-\infty}_{\R}$.
\end{proposition}

\begin{proof}
This follows directly from Lemmas \ref{lemma:embedding-holder-to-sobolev-spaces} and \ref{lemma:embedding-hs-in-ck}.
\end{proof}

\section{Parametrices modulo compact operators.}\label{sec:parametrix-modulo-compact}

\label{section:parametrices1}

In this section, we will consider elliptic parametrices for pseudo-differential operators that preserve in a precise sense the geometry of the cusps. These operators will be called \emph{admissible operators}. As explained in the introduction, we will also have to introduce, on top of usual ellipticity, a concept of \emph{ellipticity at infinity}, which is related to the inversion of a model operator induced by a given admissible operator, and called the \emph{indicial operator}.

\subsection{Decomposition at infinity}

\subsubsection{Black-box formalism} 

\label{ssection:black-box}

To start with, we will present a black-box formalism in the sense of \cite{Sjostrand-Zworski-91}, as it was introduced in \cite{Bonthonneau-Weich-17}. 

Associated with each cusp $Z_\ell$, we define extension and restriction operators. For the sake of simplicity, we drop the subscript $\ell$ and denote the cusp by $Z$. The discussion is carried out here as if there were only one cusp, but the generalization is of course straightforward. First of all, we introduce the pullback operator
\[
\imath^* : C^\infty_{\mathrm{comp}}(]a,+\infty[ \times F_Z, L_Z) \rightarrow C^\infty_{\mathrm{comp}}(N,L),
\]
defined by $\imath^*f|_{N \setminus Z} = 0$ and $\imath^*f|_{Z}(y,\theta,\zeta) = f(y,\zeta)$. Consider the restriction operator on the zero Fourier mode:
\[
\mc{P}_Z : \mc{D}'(N,L) \rightarrow \mc{D}'(]a,+\infty[ \times F_Z, L_Z),
\]
defined for $f \in \mc{D}'(N,L)$ and $\phi \in C^\infty_{\mathrm{comp}}(]a,+\infty[ \times F_Z, L_Z)$ by:
\begin{equation}
\label{equation:restriction}
\langle \mc{P}_Z f, \phi \rangle := \langle f, \imath^* \phi \rangle.
\end{equation}
Consider the space $\mc{E}'_0(]a,+\infty[ \times F_Z, L_Z) \subset \mc{D}'(]a,+\infty[ \times F_Z, L_Z)$ defined in the following way: $f \in \mc{E}'_0(]a,+\infty[ \times F_Z, L_Z)$ if and only if there exists $A > a$ such that $\supp(f) \subset ]A,+\infty[ \times F_Z$. We introduce the extension operator 
\[
\mc{E}_Z : \mc{E}'_0(]a,+\infty[ \times F_Z, L_Z) \rightarrow \mc{D}'(N,L),
\]
defined for $f \in  \mc{E}'_0(]a,+\infty[ \times F_Z, L_Z) $ and $\phi \in C^\infty_{\mathrm{comp}}(N,L)$ by:
\begin{equation}
\label{equation:extension}
\langle \mc{E}_Z f, \phi \rangle = \langle f, \chi_\phi \mc{P}_Z \phi \rangle,
\end{equation}
where $\chi_\phi \in C^\infty(]a,+\infty[ \times F_Z, \R)$ is a smooth cutoff function defined in the following way: $\chi_\phi \equiv 1$ on $[a + (A-a)/2,+\infty[ \times F_Z$ and $\chi_\phi \equiv 0$ on $]a,a+(A-a)/2020] \times F_Z$. It can be checked that \eqref{equation:extension} does not depend on the choice of $\chi_\phi$ due to the support condition on $f$. 

We fix a smooth cutoff function $\chi \in C^\infty(N,\R)$ such that $\chi|_Z \equiv 1$ for $y > 3a$ and $\chi \equiv 0$ for $y < 2a$ (recall that $y : N \rightarrow \R$ is considered to be a global positive function on $N$ which is the height function in the cusps and an arbitrary smooth positive function in the compact core). Given $f \in \mc{D}'(N,L)$, we will often consider the distribution $\mc{E}_Z \chi \mc{P}_Z f \in \mc{D}'(N,L)$. This is a distribution supported in $\left\{y \geq 2a\right\}$. Roughly speaking, it corresponds to the projection on the zero Fourier mode of $f$ on $\left\{y \geq 2a\right\}$.

Consider a smooth function $\phi$ defined on $]a,+\infty[ \times F_Z$ (and typically unbounded). We define the multiplication of the zero Fourier mode of a distribution $f \in \mc{D}'(N,L)$ by $\phi$ in the following way:
\[
\mathcal{Z}(\phi) f :=  (\mathbbm{1}-\mathcal{E}_Z \chi \mathcal{P}_Z)f +  \mathcal{E}_Z ( \chi \phi \mathcal{P}_Z f ) \in \mc{D}'(N,L).
\]
Of course, if $\phi \equiv 1$, then $\mathcal{Z}(1)f = f$. The operators $\mathcal{E}_Z$, $\mathcal{P}_Z$ and $\mathcal{Z}(\phi)$ together form a black box formalism, as it was introduced by Sj\"ostrand and Zworski in \cite{Sjostrand-Zworski-91}. Once again, if there are more than one cusp, all the previous operations are naturally extended.

\begin{definition}\label{def:Hsrhospaces}
For $s,\rho_0, \rho_\bot \in \R$, we define
\[
H^{s,\rho_0,\rho_\bot}(N,L) = \mathcal{Z}( \tilde{y}^{\rho_0-\rho_\bot} )\left(y^{\rho_\bot} H^s\right).
\]
These are \emph{weighted} Sobolev spaces, with weight $y^{\rho_0}$ on the zero Fourier mode and weight $y^{\rho_\bot}$ on the non-zero Fourier modes. We also introduce the notations:
\[
\begin{array}{l}
H^{s,-\infty,\rho_\bot}(N,L)  = \cap_{\rho_0 \in \R} H^{s,\rho_0,\rho_\bot}(N,L), \\
 H^{s,\rho_0,-\infty}(N,L) = \cap_{\rho_\bot \in \R} H^{s,\rho_0,\rho_\bot}(N,L), \\
H^{s,-\infty,-\infty}(N,L) = \cap_{\rho_0, \rho_\bot \in \R} H^{s,\rho_0,\rho_\bot}(N,L)
\end{array}
\]
When both weights are equal, namely $\rho_0 = \rho_\bot =: \rho$, we will simply write $H^{s,\rho} := H^{s,\rho,\rho}$.
\end{definition}

Note that we take the same weight on each cusps, this will suffice for our purposes. To obtain compact remainders in parametrices, the following observation going back to \cite{Lax-Phillips-Automorphic-76} is essential: for any $\rho_\bot \in \R, s > s'$, the restriction of the injection $y^{\rho_\bot} H^s(N,L) \hookrightarrow y^{\rho_\bot} H^{s'}(N,L)$ to non-constant Fourier modes is compact. More precisely:

\begin{lemma}
\label{lemma:compact-injection-non-constant-mode}
For all $s > s'$, for all $k \in \N$, the operator 
\[
\mathbbm{1}-\mathcal{E}_Z \chi \mathcal{P}_Z : H^{s,\rho_0,\rho_\bot}(N,L) \rightarrow H^{s',-k,\rho_\bot}(N,L)
\]
is compact.
\end{lemma}

In the following, we will sometimes say by simplicity that
\[
\mathbbm{1}-\mathcal{E}_Z \chi \mathcal{P}_Z : H^{s,\rho_0,\rho_\bot}(N,L) \rightarrow H^{s',-\infty,\rho_\bot}(N,L)
\]
is compact, meaning exactly that this operator maps to any $H^{s',-k,\rho_\bot}(N,L)$ and that it is compact when mapping on such spaces.

\begin{proof}
The value of $\rho_0$ is inessential here, so we take $\rho_0=\rho_\bot=\rho$. Since $[\mathbbm{1}-\mathcal{E}_Z \chi \mathcal{P}_Z,y^\rho]=0$, the lemma boils down to the case $\rho=0$. For the sake of simplicity, we assume that there is a single cusp and that $L \rightarrow N$ is a trivial bundle, the general case is handled in a similar fashion. Let $\psi_n \in C^\infty_{\mathrm{comp}}(N)$ be a smooth cutoff function such that $\psi_n \equiv 1$ on $y < n$ and $\psi_n \equiv 0$ on $y > 2n$. The operators of injection
\[
T_n := \psi_n (\mathbbm{1}-\mathcal{E}_Z \chi \mathcal{P}_Z) \in \mc{L}(H^s(N), H^{s'}(N))
\]
are compact, so it is sufficient to prove that the injection
\[
T := \mathbbm{1}-\mathcal{E}_Z \chi \mathcal{P}_Z \in \mc{L}(H^s(N), H^{s'}(N))
\]
is the norm-limit of the operators $T_n$. In other words, if we can prove that for all $n \in \N$, there exits a constant $C_n > 0$ such that: for all $f \in H^s(N)$ such that $\chi \mc{P}_Zf \equiv 0$ (we denote by $H^s_0(N)$ the space of such functions endowed with the norm $\|\cdot\|_{H^s}$), we have
\[ \|(\mathbbm{1}-\psi_n)f\|_{H^{s'}} \leq C_n \|f\|_{H^s}, \]
and that $C_n \rightarrow_{n \rightarrow +\infty} 0$, then we are done. Using Wirtinger's inequality, one can obtain like in \cite[Lemma 4.9]{Bonthonneau-Weich-17} that
\[
\|\mathbbm{1}-\psi_n\|_{\mc{L}(H_0^1,L^2_0)} \leq C/n
\]
for some uniform constant $C > 0$ (depending on the lattice $\Lambda$). Since the $(\psi_n)_{n \geq 0}$ can be designed so that $\|\mathbbm{1}-\psi_n\|_{\mc{L}(H_0^1,H_0^1)} \leq 1$, we obtain by interpolation that $\|\mathbbm{1}-\psi_n\|_{\mc{L}(H_0^1,H_0^s)} \leq (C/n)^{1-s}$ for all $s \in [0,1]$. Since $\|\mathbbm{1}-\psi_n\|_{\mc{L}(H_0^k,H_0^k)} \leq 1$ for all $k \in \Z$, we can interpolate once again to conclude.
\end{proof}

\begin{lemma}\label{lemma:compact-injection}
Consider $\rho_\bot\in\R$, $\rho_0< \rho_0'$, and $s > s'$. Then $H^{s,\rho_0,\rho_\bot}(N,L)\hookrightarrow H^{s',\rho_0',\rho_\bot}(N,L)$ is a compact injection.
\end{lemma}

\begin{proof}
One can write $f = (\mathbbm{1}-\mathcal{E}_Z \chi \mathcal{P}_Z)f + \mathcal{E}_Z \chi \mathcal{P}_Zf$. The first term is dealt by applying the previous lemma, namely
\[
(\mathbbm{1}-\mathcal{E}_Z \chi \mathcal{P}_Z) : H^{s,\rho_0,\rho_\bot}(N,L)\rightarrow H^{s',\rho_0',\rho_\bot}(N,L)
\]
compact. As to $\mathcal{E}_Z \chi \mathcal{P}_Z : H^{s,\rho_0,\rho_\bot}(N,L)\rightarrow H^{s',\rho_0',\rho_\bot}(N,L)$, we see it as the operator:
\[
\begin{split}
H^{s,\rho_0,\rho_\bot}(N,L) \ni f  & \mapsto \chi \mc{P}_Zf \in H^{s,\rho_0,\rho_\bot}(]a,+\infty[\times F_Z,L_Z)  \\
& \hookrightarrow  \chi \mc{P}_Zf \in H^{s',\rho'_0,\rho_\bot}(]a,+\infty[\times F_Z,L_Z) \\
& \mapsto \mc{E}_Z  \chi \mc{P}_Zf \in  H^{s',\rho'_0,\rho_\bot}(N,L)
\end{split}
\]
and use the fact that the injection in the middle is compact. (Indeed, forgetting about the transverse manifold $F_Z$, this boils down to a rather elementary statement on functions defined on $\R$. By making the change of variable $y = e^r$, one can check that this is implied by the following statement: the injection $e^{\rho r} H^s(\R_r) \hookrightarrow e^{\rho' r} H^{s'}(\R_r)$ is compact for $\rho < \rho', s > s'$ (here $\R$ is equipped with the usual Lebesgue measure $\dd r$). The latter statement can be reduced to proving that $e^{-\rho r} H^s(\R_r) \hookrightarrow L^2(\R_r, \dd r)$ is compact. This is a classical result on $\R$.)
\end{proof}

Most of the time, it will be sufficient to deal with the non-zero Fourier modes as part of the black box. However, it will prove useful to know that some regularity can be traded off for decay in those modes; it is the content of the following lemma:
\begin{lemma}
\label{lemma:embedding-nonzero}
Consider $\rho_\bot, \rho_0, s \in \R, m > 0$. Then $H^{s+m,\rho_0,\rho_\bot+m} \hookrightarrow H^{s,\rho_0,\rho_\bot}$ is a continuous embedding.
\end{lemma}

Observe that there is no reason for this embedding to be compact.

\begin{proof}
Decomposing once again in zero and non-zero Fourier modes and using interpolation estimates, it is sufficient to prove that $yH^{1} \hookrightarrow L^2$ is a continuous embedding on functions with vanishing zeroth Fourier mode. But:
\[
\begin{split}
\|f\|_{yH^1}^2 & = \|y^{-1}f\|_{H^1}^2 \\
& \asymp \|y^{-1}f\|^2_{L^2} + \|y\partial_y(y^{-1}f)\|^2_{L^2} + \|y\partial_\theta(y^{-1}f)\|^2_{L^2} + \|\partial_\zeta (y^{-1}f)\|^2_{L^2}
\end{split}
\]
Using Wirtinger's inequality for functions with zero integral, we can control the term $\|y\partial_\theta(y^{-1}f)\|^2_{L^2} = \|\partial_\theta f\|^2_{L^2} \geq \|f\|_{L^2}^2$ and this provides the sought estimate.
\end{proof}

If $f \in \mc{D}'(N,L)$, then $\chi \mc{P}_Z f \in \mc{E}'_0(]a,+\infty[ \times F_Z, L_Z)$ and this can naturally be seen as an element $\chi \mc{P}_Z f \in \mc{D}'(]0,+\infty[_y \times F_Z,L_Z)$ (here $\chi$ is extended by $0$ for $y < a$). Making the change of variable $y = e^r$, where $r \in \R_r$, we can see that as an element $\chi \mc{P}_Z f \in \mc{D}'(\R_r \times F_Z,L_Z)$. We then have the following correspondence:

\begin{lemma}
\label{lemma:relation-projected-spaces}
The following maps are bounded
\[
H^{s,\rho_0,\rho_\bot}(N,L) \owns f \mapsto \chi \mathcal{P}_Z f \in e^{(\rho_0+d/2)r} H^s(\R\times F_Z, L_Z);
\]
\[
e^{\rho_0 r} H^s(\R\times F_Z, L_Z) \owns f \mapsto \mathcal{E}_Z( \chi f) \in H^{s,\rho_0-d/2,-\infty}(N,L),
\]
where $r=\log y$, and $H^s(\R\times F_Z, L)$ is the usual Sobolev space, built from the $L^2$ space induced by the measure $\dd r \dd\vol_{F_Z}(\zeta)$.
\end{lemma}

The spaces $H^s(\R \times F_Z)$ are the ones induced by the Laplacian obtained from the product metric $\dd r^2 + g_F$. We insist on the fact that there is a shift of $-d/2$ due to the fact that we are considering the usual euclidean measure when working in the $r$-variable. The proof of Lemma \ref{lemma:relation-projected-spaces} is postponed: we will prove it below at the end of the next paragraph.

\subsubsection{Admissible operators}

We can now introduce the general class of \emph{$\R$-admissible operators} (also called \emph{admissible operators} below).

\begin{definition}\label{def:R-L^2-admissible-operator}
Consider $A\in \Psi^m_{\text{small}}(N,L_1 \to L_2)$. We will say that $A$ is a \emph{$\R$-admissible operator} if the following holds:

\begin{enumerate}
\item There exists a pseudo-differential operator $I_Z(A) \in \Psi^m(\R_r \times F_Z, L_Z)$ (in the usual Kohn-Nirenberg class with $\rho=1,\delta=0$, see \cite[Definition 1.1]{Shubin-01} for instance) of order $m$, called the \emph{indicial operator of $A$}, such that $[I_Z(A),\partial_r] = 0$ (in other words, $I_Z(A)$ is a convolution operator\footnote{By convolution operator, we mean the following: there exists a Schwartz kernel $K \in \mc{D}'(\R_r \times F_Z \times F_Z)$ such that:
\[
I_Z(A)f (r',\zeta') = \int_{\R \times F_Z} K(r'-r,\zeta',\zeta) f(\zeta) \dd r \dd \vol_{g_F}(\zeta),
\] 
for all $f \in C^\infty_{\mathrm{comp}}(\R \times F_Z)$.} 
in the $r$-variable),
\item There exists a smooth cutoff function $\chi_A \in C^\infty(]a,+\infty[)$ (depending on $A$), such that $\chi_A$ is supported in $\left\{y>2a\right\}$ and equal to $1$ in $\left\{y> C_A\right\}$, where $C_A > 2a$ is a constant (depending on $A$), 

\item We have:
\begin{equation}
\label{eq:preserving-theta-modes}
\chi_A  [A,\partial_\theta] \chi_A \text{ and } \mathcal{E}_Z\chi_A \left[\mathcal{P}_Z A \mathcal{E}_Z - I_Z(A) \right] \chi_A \mathcal{P}_Z,
\end{equation}
are $\R$-residual operators (i.e. bounded from $y^k H^{-k}$ to $y^{-k}H^{k}$, for all $k \in \N$, or equivalently from $y^k C^{-k}_*$ to $y^{-k}C^k_*$ by Proposition \ref{prop:maximallyresidualholder}). 
\end{enumerate}
(Here, by abuse of notations, $\chi_A$ is identified with a function on $N$, supported in the cuspidal parts and only depending on the $y$ variable.)
\end{definition}

When $\rho > \rho'$, the unique convolution operator that is bounded from $e^{\rho r}L^2(dr)$ to $e^{\rho' r}L^2(dr)$ is the null operator and it follows that \emph{the indicial operator associated with an $\R$-admissible operator is necessarily unique}. Modulo compact remainders, the first condition in \eqref{eq:preserving-theta-modes} means that the operator $A$ preserves the $\theta$-Fourier modes; the second condition implies that sufficiently high in the cusp\footnote{By ``sufficiently high in the cusp", we mean that the behaviour of the pseudo-differential $A$ can be arbitrary on some compact set $\left\{y < C_A\right\}$, hence the introduction of the cutoff $\chi_A$.}, $A$ is a convolution operator in the $r=\log y$ variable when acting on the zeroth Fourier mode. In particular, if $B$ is a compactly supported pseudo-differential operator, $B$ is admissible, and $I_Z(B)=0$.

Observe that in general, if $\sigma\in S^m(T^*N)$, then $\partial_\theta \sigma \in y^{-\infty} S^m(T^*N)$. Indeed, since $\partial_\theta \sigma$ has mean value $0$ in the $\theta$ variable, it is controlled by $\partial_\theta^N \sigma$, and $(y\partial_\theta)^N \sigma \in S^m$ for each $N\in \N$. It follows that if $P\in \Psi^m_{\text{small}}$, then in the cusp, $\chi [P, \partial_\theta]\chi$ is in $y^{-\infty} \Psi^m_{\mathrm{small}}$. Indeed, its symbol can be expressed with derivatives of the symbol of $P$, that include at least one derivative $\partial_\theta$. What we gain with our assumption is that on top of being decaying the cusp, the operator is smoothing. %the order of the operator becomes $\Psi^{-\infty}_{\mathrm{small}}$.

An important consequence of the definition is the following:

\begin{lemma}
\label{lemma:admissible-properties}
If $A$ is $\R$-admissible, then
\begin{equation}
\label{eq:restriction-to-zeroth-mode}
\mc{E}_Z \chi_A \mathcal{P}_Z A[\mathbbm{1} - \mathcal{E}_Z\chi_A \mathcal{P}_Z],\text{ and } [\mathbbm{1}-\mathcal{E}_Z \chi_A \mathcal{P}_Z]A \mathcal{E}_Z \chi_A \mc{P}_Z
\end{equation}
are both $\R$-residual.
\end{lemma}

\begin{proof}For the first one, let $\mathbf{R}_\theta : C^\infty(\R^d/\Lambda) \rightarrow C^\infty(\R^d/\Lambda)$ be the inverse of $\partial_\theta$ on functions with vanishing zeroth Fourier mode, i.e. $\{ f\in L^2(\R^d/\Lambda), \int_{\R^d/\Lambda} f(\theta) \dd \theta = 0\}$ (and $\mathbf{R}_\theta \mathbf{1} = 0$, where $\mathbf{1}$ denotes the constant function equal to $1$). Write $\Pi_0$, the orthogonal projection onto the zeroth Fourier mode. Then: $\partial_\theta \mathbf{R}_\theta = \mathbf{R}_\theta \partial_\theta = \mathbbm{1} - \Pi_0$. The operator $\mathbf{R}_\theta$ acts naturally on functions that are supported in the cusp. Indeed, if $f \in C^\infty_{\mathrm{comp}}(N,L)$ is supported in $\left\{y > a\right\}$, then $(\mathbf{R}_\theta f)(y,\theta,\zeta) := (\mathbf{R}_\theta f(y,\cdot,\zeta))(\theta)$. The operator $\chi_A \mathbf{R}_\theta \chi_A$ is then bounded on every $H^{s,\rho_0,\rho_\bot}(N,L)$, $s, \rho_0,\rho_\bot \in \R$.

Now, taking $\chi'_A$ supported in $\left\{y > a \right\}$ such that $\chi'_A \equiv 1$ on $\supp(\chi_A)$, we obtain:
\begin{align*}
0 	&= \partial_\theta \mc{E}_Z \chi_A \mathcal{P}_Z A  \mathbf{R}_\theta \chi'_A [\mathbbm{1} - \mathcal{E}_Z \chi_A \mathcal{P}_Z] \\
	&= \mc{E}_Z \chi_A \mathcal{P}_Z A (\mathbbm{1}-\Pi_0) \chi'_A [\mathbbm{1} - \mathcal{E}_Z \chi_A \mathcal{P}_Z] + \mc{E}_Z \chi_A \mathcal{P}_Z[\partial_\theta, A] \mathbf{R}_\theta \chi'_A  [\mathbbm{1} - \mathcal{E}_Z \chi_A \mathcal{P}_Z] \\
	& = \mc{E}_Z \chi_A \mathcal{P}_Z A [\mathbbm{1} - \mathcal{E}_Z \chi_A \mathcal{P}_Z] \chi'_A  - \mc{E}_Z \chi_A \mathcal{P}_Z A [\mathbbm{1} - \mathcal{E}_Z \chi_A \mathcal{P}_Z]  \Pi_0 \chi'_A \\
	& ~~~~~~ + \mc{E}_Z \chi_A \mathcal{P}_Z[\partial_\theta, A]  \mathbf{R}_\theta \chi'_A  [\mathbbm{1} - \mathcal{E}_Z \chi_A \mathcal{P}_Z]. 
\end{align*}
Using that $[\mathbbm{1} - \mathcal{E}_Z \chi_A \mathcal{P}_Z]  \Pi_0 f = 0$ for any $f \in C^\infty(N,L)$ such that $\supp(f) \subset \left\{y > C_A \right\}$, together with the assumption \eqref{eq:preserving-theta-modes} on the bracket $[\partial_\theta,A]$, we obtain the desired result for the operator $ \mc{E}_Z \chi_A \mathcal{P}_Z A [\mathbbm{1} - \mathcal{E}_Z \chi_A \mathcal{P}_Z] \chi'_A$, namely it is bounded as an operator $y^k H^{-k}(N,L) \rightarrow y^{-k} H^{k}(N,L)$ for any $k \in \N$. To get rid of the $\chi'_A$ in front, it suffices to observe that:
\[
\begin{split}
\mc{E}_Z \chi_A \mathcal{P}_Z A [\mathbbm{1} - \mathcal{E}_Z \chi_A \mathcal{P}_Z] (\mathbbm{1}-\chi'_A) & = \mc{E}_Z \underbrace{\chi_A (\mathbbm{1}-\chi'_A)}_{=0} \mathcal{P}_Z A [\mathbbm{1} - \mathcal{E}_Z \chi_A \mathcal{P}_Z] \\
& + \mc{E}_Z \chi_A \mathcal{P}_Z [A,(\mathbbm{1}-\chi'_A)] [\mathbbm{1} - \mathcal{E}_Z \chi_A \mathcal{P}_Z].
\end{split}
\]
Due to the support property of the kernel of $A$ (more precisely, the introduction of the cutoff function in the hyperbolic quantization, see \S\ref{sssection:hyperbolic-quantization}), it is not difficult to check that this operator has compact support. It is also clearly regularizing (in the sense that it maps any compactly supported distribution to smooth functions) as $\chi_A$ and $\mathbbm{1}-\chi_A'$ have disjoint (micro)support. Hence, it is $\R$-residual.
\end{proof}

Lemma \ref{lemma:admissible-properties} says that, up to compact remainders, $A$ acts diagonally on the decomposition between zero and non-zero Fourier mode. However, we observe that the sole conditions in \eqref{eq:restriction-to-zeroth-mode} are not necessarily stable under products, nor under taking parametrices. If we want to consider a class of operators enjoying such stability, we need to assume \eqref{eq:preserving-theta-modes}.

\begin{proposition}\label{prop:relation-symbol-geometric}
Consider $A=\Op(\sigma)$. Then:
\begin{enumerate}
\item The first boundedness property in \eqref{eq:preserving-theta-modes} is satisfied if $\partial_\theta \sigma =0$.
\item The second one in \eqref{eq:preserving-theta-modes} also does in each cusp if, 
\[
\widetilde{\sigma}:\ (r,z;\lambda,\eta)\ \mapsto \int_{\R^d/\Lambda} \sigma_{|Z}( e^r, \theta, \zeta; e^{-r}\lambda, J=0,\eta) \dd\theta,
\]
does \emph{not} depend on $r$. In that case, the operator $I_Z(A)$ is pseudo-differential, properly supported, and its principal symbol is $\widetilde{\sigma}$. Both these conditions are satisfied when $\sigma$ is invariant by local isometries of the cusp.

\item An operator $A$ is $\R$-admissible if and only if it is of the form $\Op(\sigma) + B + R$, where $\sigma$ satisfies the conditions above, $R$ is $\R$-admissible smoothing, and $B$ is a compactly supported pseudo-differential operator. In particular, the set of $\R$-admissible operators is stable by composition.
\end{enumerate}
\end{proposition}

\begin{proof}[Proof of Proposition \ref{prop:relation-symbol-geometric}]
Again, it suffices to work in local charts for $F$, i.e with with $\Op_U$ on $Z\times U\subset \R^k$. First, we observe directly from formula \eqref{equation:cutoff} that $[\partial_\theta, \Op(\sigma)] = \Op(\partial_\theta \sigma)$, so that when $\partial_\theta \sigma =0$, $\Op_U(\sigma)$ commutes with $\partial_\theta$. Reciprocally, if $[\partial_\theta,\Op(\sigma)]$ is bounded from $y^{N} H^{-N}$ to $y^{-N}H^{N}$, it implies that $\partial_\theta \sigma \in y^{-\infty} S^{-\infty}$ --- expressing $\sigma$ in terms of the kernel of $\Op(\sigma)$. In particular, we can replace $\sigma$ by $\int \sigma d\theta$, and this only adds a negligible correction. For the second condition, one has to do a change of variables. For details, we refer to \cite[Section 4.1]{Bonthonneau-Weich-17}.
\end{proof}

To apply the machinery on H\"older-Zygmund spaces later on, we need to make an important observation. In the Definition \ref{def:R-L^2-admissible-operator} of admissible operators, we consider small operators $A\in \Psi^m_{\text{small}}$ (see Definition \ref{definition:small}) and in the very definition of this algebra, we allowed for remainders that are $\R$-smoothing, a notion that is defined in terms of Sobolev spaces: $R$ is $\R$-smoothing if it bounded as
\[
R  : y^\rho H^{-s} \to y^\rho H^s,
\]
for all $s,\rho \in \R$ (see \S\ref{sssection:smoothing}). Let us even say for a moment that such operators are \emph{Sobolev $\R$-smoothing}. We have also introduced Hölder-Zygmund $\R$-smoothing operators, i.e those $R$'s that are bounded as
\[
R : y^\rho C^{-s}_\ast \to y^{\rho} C^s_\ast,
\]
for all $\rho,s\in\R$ and considered, similarly to the Definition \ref{definition:small} of small pseudo-differential operators, the algebra of operators of the form
\[
\widetilde{\Psi}^m_{\text{small}} := \left\{ \Op(a) + \text{Hölder-Zygmund } \R\text{-smoothing} ~|~ a \in S^m \right\}.
\]
As mentioned in \S\ref{sssection:smoothing}, we insist on the fact that \emph{it is not clear to us that this algebra is the same as the algebra $\Psi^m_{\text{small}}$ of small pseudo-differential operators introduced in Definition \ref{definition:small}.} Nevertheless, it turns out that in the specific case where the operator is admissible, \emph{these two notions agree}.

\begin{proposition}\label{prop:equivalenceL2Linfty-admissible}
If $A \in \widetilde{\Psi}^m_{\text{small}}$ and $A$ is $\R$-admissible, then $A \in \Psi^m_{\text{small}}$. Conversely, if $A \in \Psi^m_{\text{small}}$ and $A$ is $\R$-admissible, then $A \in \widetilde{\Psi}^m_{\text{small}}$
\end{proposition}

As we will only work with admissible operators in the following, we can simply drop the notation $\widetilde{\Psi}^m_{\text{small}}$.

\begin{proof}
We will only prove that if $A \in \Psi^m_{\text{small}}$ is $\R$-admissible in the original sense, then $A \in \widetilde{\Psi}^m_{\text{small}}$, i.e. it also is in the ``H\"older-Zygmund'' sense. The other direction is completely similar.

According to Proposition \ref{prop:relation-symbol-geometric}, we can directly assume that $A$ is $\R$-smoothing, and we have to prove that it is Hölder-Zygmund $\R$-smoothing. According to Lemma \ref{lemma:admissible-properties}, we can always replace $A$ by $\chi_A \mathcal{E}_Z I_Z(A) \mathcal{P}_Z \chi_A$, since $\R$-residual operators also map $y^kC^{-k}_* \to y^{-k}C^k_*$ for all $k \in \N$ (see Proposition \ref{prop:maximallyresidualholder}). In this way, we reduce our problem completely to a study of operators on $\R_r\times F$, which are convolution operators in the $r$ variable.

As a consequence, we have to prove that $I_Z(A)$ is bounded from $e^{\rho r}C^{-s}_\ast(\R\times F) \to e^{\rho r}C^{s}_\ast(\R\times F)$ for all $\rho \in \R,s> 0$. From the Sobolev embeddings, we already know that $e^{\rho r} C^{-s}_\ast$ is mapped to $e^{\rho r+\epsilon\langle r\rangle}C^s_\ast$ for all $\epsilon>0$, but we need to avoid this loss. Here, the presence of $F$ is not essential, and we can assume that $F$ is a point. Indeed, denoting by $\mathcal{H}:= \mathcal{L}(C^{-s}_\ast(F),C^s_\ast(F))$, we can see $I_Z(A)$ as convolution operator on $\R$ with values in $\mathcal{H}$ (using the Sobolev embeddings in $F$).

Let us denote by $K(r)$ the convolution kernel of $I_Z(A)$. If $\delta$ is the dirac at $0$, $K(r) = [I_Z(A)\delta](r)$. In particular, since $\delta \in e^{\rho r}H^{-N}$ for all $\rho\in \R$ and $N>0$ large enough, we deduce that $K \in e^{\rho r}H^N$ for all $\rho\in \R$, and $N>>1$. In particular, $K(r)\in e^{\rho r} C^N_\ast$ for all $\rho, N\in \R$, again according to the Sobolev embeddings. This is sufficient to conclude.

\end{proof}

Let us give some examples of admissible (pseudo)differential operators. If $(L,\nabla^L)$ is an admissible vector bundle in the sense of Definition \ref{definition:admissible-bundle}, then from the decomposition \eqref{eq:structure-connection}, we deduce that $\nabla^L$ is an admissible (pseudo)differential operator. More generally, all the differential operators that can be defined completely locally using only the metric structure are bound to be properly supported admissible differential operators (e.g. the Laplacian or the Levi-Civita connection). In the following, the operator $D$ — the symmetric derivative of symmetric tensors, see \S\ref{ssection:xray} for a definition — and $D^*D$ will be local differential operators constructed from the metric, so they will be properly supported admissible differential operators, in the sense of Definition \ref{def:R-L^2-admissible-operator} \\

To complete this section, we prove the following boundedness result:

\begin{lemma}
\label{lemma:boundedness-admissible-pseudo}
Let $A$ be an $\R$-admissible pseudo-differential operator of order $m \in \R$. Then $A$ is bounded as an operator between $H^{s+m,\rho_0,\rho_\bot}$ and $H^{s,\rho_0,\rho_\bot}$, for all $s, \rho_0,\rho_\bot \in \R$.
\end{lemma}

\begin{proof}
We decompose the operator into four terms:
\[ 
\begin{split}
A & = (\mathbbm{1}-\mc{E}_Z\chi_A\mc{P}_Z)A(\mathbbm{1}-\mc{E}_Z\chi_A\mc{P}_Z)f \\
& + \mc{E}_Z\chi_A\mc{P}_Z A(\mathbbm{1}-\mc{E}_Z\chi_A\mc{P}_Z)f + (\mathbbm{1}-\mc{E}_Z\chi_A\mc{P}_Z)A\mc{E}_Z\chi_A\mc{P}_Zf + \mc{E}_Z\chi_A\mc{P}_Z A\mc{E}_Z\chi_A\mc{P}_Zf
\end{split}
\]
The first term is bounded as a map
\[
\begin{split}
H^{s+m,\rho_0,\rho_\bot} & \overset{\mathbbm{1}-\mc{E}_Z\chi_A\mc{P}_Z}{\rightarrow} H^{s+m,-\infty,\rho_\bot} \\ 
& \hookrightarrow y^{\rho_\bot} H^{s+m} \overset{A}{\rightarrow} y^{\rho_\bot} H^{s}  \overset{\mathbbm{1}-\mc{E}_Z\chi_A\mc{P}_Z}{\rightarrow} H^{s,-\infty,\rho_\bot} \hookrightarrow H^{s,\rho_0,\rho_\bot},
\end{split}
\]
where we have used the boundedness of $A$ obtained in Proposition \ref{prop:microlocal-calculus}. By (\ref{eq:restriction-to-zeroth-mode}), the second and third terms are immediately bounded. As to the last term, it is dealt exactly like the first term.
\end{proof}

We can now complete the proof of Lemma \ref{lemma:relation-projected-spaces}, asserting that the following maps are bounded:
\[
H^{s,\rho_0,\rho_\bot}(N,L) \owns f \mapsto \chi \mathcal{P}_Z f \in e^{(\rho_0+d/2)r} H^s(\R\times F_Z, L_Z),
\]
\[
e^{\rho_0 r} H^s(\R\times F_Z, L_Z) \owns f \mapsto \mathcal{E}_Z( \chi f) \in H^{s,\rho_0-d/2,-\infty}(N,L).
\]

\begin{proof}[Proof of Lemma \ref{lemma:relation-projected-spaces}]
For this lemma, the presence of the compact part is inessential, so we can work completely in the cusps. We can decompose the statement into two parts. The first one is that the map which associates a section $f$ with its zeroth Fourier mode $\mathcal{E}\mathcal{P} f$ is continuous on the weighted Sobolev spaces. 

To prove this, we first observe that the parameter $\rho^\perp$ is not essential here, since the differentiated weight for zeroth and non-zeroth Fourier mode was introduced directly using the Fourier decomposition. In particular, we can assume that $\rho^\perp = \rho_0$. Next, we recall that the spaces $H^s$ are equal to $\Op(\sigma_{-s}) L^2$, where $\sigma_{-s}$ is a symbol, and $\Op(\sigma_{-s})$ is a parametrix for $(-\Delta + C)^{-s/2}$. Since $\Delta$ is $\R$-admissible, this parametrix construction can be done also in an admissible fashion. In particular we can apply Lemma \ref{lemma:admissible-properties} to $\Op(\sigma_{-s})$, and deduce that the action of $\Op(\sigma_{-s})$ and the taking of the zeroth Fourier mode commute modulo $\R$-residual operators, hence proving that the map $\mathcal{E}\mathcal{P}$ is bounded on the relevant spaces. 

Next, the second prove of the argument is to prove that $y^\rho H^s$ intersected with the space of function not depending on $\theta$ can be identified with the usual Sobolev spaces $e^{(\rho_0+d/2)r}H^s(\R_r \times F)$. For this it suffices to observe that the $d/2$ shift is due to the measure $y^{-d}dyd\theta d\vol_F$, and then observe that $y\partial_y = \partial_r$, so that the vector fields one can use to define $H^s$ regularity on the cusp are essentially the same as the ones used to define regularity on $\R\times F$.

%
%
%The first observation is that the statement for $s=0$ (i.e for $L^2$ spaces) follows directly from the definition of those spaces. 
%
%Next, we recall that the spaces $H^s$ can be expressed as $\Op(\sigma_{-s})L^2$, for some symbol $\sigma_{-s}$ which is constructed as a parametrix for $(-\Delta+ C)^{-s/2}$.  The parametrix construction yields a symbol which satisfies the conclusions of Proposition \ref{prop:relation-symbol-geometric}. In particular, since $\Op(\sigma_{-s})$ is $\R$-admissible, we can apply Lemma \ref{lemma:admissible-properties}. This enables us to deal with the $\rho^\perp$ parameter, and we can assume without loss of generality that $\rho^\perp = \rho_0$.
%
%Finally, we are left with th
%
%Recall that $H^s = \Lambda_{-s} L^2$ with $\Lambda_s = \Op(\sigma_s)$. Since the symbols $\sigma_s$ were built in a parametrix construction for the Laplacian, they are invariant under local isometries. In particular, $\Lambda_s$ is an $\R$-admissible operator in $\Psi^{-s}_{\mathrm{small}}$. Additionally, by the previous Lemma \ref{lemma:indicial-pseudo}??, its restriction to the zeroth Fourier mode acts as a pseudo-differential operator of order $-s$. Since it is uniformly properly supported, conjugation with $y^\rho$ does not change these properties. It follows that one can restrict to the case that $s=\rho_0=\rho_\bot = 0$.
%
%
%Now, it boils down to the observation that the volume measure on the cusp is $y^{-d-1}d\theta dy= e^{-r d }d\theta dr$ with $r=\log y$.
\end{proof}

\subsubsection{Admissible operators on an interval}

When dealing with differential operators, whose kernel is supported exactly on the diagonal, the assumption that one can work with spaces $H^{s,\rho_0,\rho_\bot}$ for any $\rho_0,\rho_\bot \in \R$ is not very important. However, we will be dealing with pseudo-differential operators that are not properly supported. Moreover, we will also be dealing with parametrices of differential operators, which cannot be $\R$-admissible as some poles appear. As before, we consider two admissible vector bundles $L_1, L_2 \rightarrow N$.

Our more general class of operator differs only in the behaviour of its kernel far from the diagonal, i.e its smoothing part. For this reason, we emphasize the definition of the remainders. 
\begin{definition}
\label{definition:rho-admissible}
Let $\rho_+ > \rho_-$. Given $R : C^\infty_{\mathrm{comp}}(N,L_1) \rightarrow \mc{D}'(N,L_2)$, we say that $R$ is \textbf{$(\rho_-,\rho_+)$-smoothing admissible} if there exists a smooth cutoff function $\chi_R$ supported in $\left\{y > a\right\}$, taking value $1$ for $y$ large enough and such that:
\begin{enumerate}
	\item For all $\rho_0 \in ]\rho_-, \rho_+[$, $\rho_\bot \in \R$ and $k \in \N$,
	\[
	R : H^{-k,\rho_0-d/2,\rho_\bot}(N,L_1) \to H^{k,\rho_0-d/2,\rho_\bot}(N,L_2),
	\]
	is bounded, i.e. in light of \S\ref{sssection:smoothing} $R$ is $(\rho_-,\rho_+)$-smoothing,
	\item For all $\rho_\bot\in\R$ and $k \in \N,\epsilon>0$,
	\[
	\chi_R [\partial_\theta, R]\chi_R : H^{-k, \rho_+ - d/2 -\epsilon,\rho_\bot}(N,L_1) \to H^{k,\rho_- - d/2 +\epsilon,\rho_\bot}(N,L_2),
	\]
	is bounded, i.e. in light of \S\ref{sssection:smoothing} the previous operator is $(\rho_-,\rho_+)$-residual,
	\item There is a convolution operator $I_Z(R)$ such that
	\[
	\mc{E}_Z \chi_R (\mathcal{P}_Z R \mathcal{E}_Z  -  I_Z(R) )\chi_R  \mc{P}_Z
	\]
	is an operator bounded from $H^{-k,\rho_+-d/2-\epsilon,\rho_\bot}(N,L_1) \to H^{k,\rho_- -d/2+\epsilon, \rho_\bot}(N,L_2)$ for all $k \in \N ,\epsilon>0, \rho_\bot \in \R$, i.e. the previous operator is $(\rho_-,\rho_+)$-residual.
\end{enumerate}

We say that an operator $A : C^\infty_{\mathrm{comp}}(N,L_1) \rightarrow \mc{D}'(N,L_2)$ is \textbf{$(\rho_-,\rho_+)$-admissible} of order $m \in \R$ if it can be decomposed as $A = A_{\mathrm{comp}} + A_{\mathrm{cusp}} + R$, where $A_{comp}$ is a compactly supported pseudo-differential operator of order $m$, $A_{cusp} = \Op(\sigma)$ with $\sigma \in S^m(T^*N,\mathrm{Hom}(L_1,L_2))$ satisfying the conditions of Proposition \ref{prop:relation-symbol-geometric} and $R$ is $(\rho_-,\rho_+)$\emph{-smoothing admissible}.
\end{definition}

The difference between being $\R$-admissible and $(\rho_-,\rho_+)$-admissible lies only in the behaviour on the zeroth Fourier mode in the cusps, where certain asymptotic behaviour is allowed. More specifically, it is only related to the decay of the kernel far from the diagonal. In the other Fourier modes in $\theta$, all exponential behaviours are allowed. Each $(\rho_-, \rho_+)$ admissible operator $A$ is associated with a convolution operator $I_Z(A) = I_Z(A_{\mathrm{cusp}}) + I_Z(R)$ in each cusp. Note that $I_Z(A_{\mathrm{cusp}})$ is well-defined since $A_{\mathrm{cusp}}$ is $\R$-admissible.

We have the following equivalent of Lemma \ref{lemma:admissible-properties}, Propositions \ref{prop:relation-symbol-geometric} and \ref{prop:equivalenceL2Linfty-admissible}:
\begin{proposition}\label{prop:algebra-general}
The set of $(\rho_-,\rho_+)$-admissible operators is an algebra of operators. If in Definition \ref{definition:rho-admissible} one replaces the spaces $y^{\rho-d/2}H^s$ by the spaces $y^\rho C^s_\ast$ everywhere, the class of admissible operators obtained is the same. Moreover, if one replaces $\R$ by $I$ in Lemma \ref{lemma:admissible-properties}, the statement still holds
\end{proposition}

The proof is the same as that of Lemma \ref{lemma:admissible-properties} and Propositions \ref{prop:relation-symbol-geometric} and \ref{prop:equivalenceL2Linfty-admissible}.

\subsubsection{Main result}

We can now state our main Theorem, using the notions introduced in the previous paragraph:

\begin{theorem}\label{theorem:parametrix-compact}
Let $L$ be an admissible bundle in the sense of Definition \ref{definition:admissible-bundle}. Assume that $L$ is endowed with a $(\rho_-,\rho_+)$-admissible pseudo-differential operator $P$ of order $m \in \R$, in the sense of Definition \ref{definition:rho-admissible}. Also assume that it is uniformly elliptic in the sense of Definition \ref{def:elliptic}. Then there is a discrete set $S(P)\subset (\rho_-,\rho_+)$ such that for each connected component $I := (\rho^I_-,\rho^I_+) \subset (\rho_-,\rho_+) \setminus S(P)$, there is a $I$-admissible pseudo-differential operator $Q_I$ such that
\[
P Q_I = \mathbbm{1} \text{ mod } \dot{\Psi}^{-\infty}_{I}, ~~~~Q_I P  = \mathbbm{1} \text{ mod } \dot{\Psi}^{-\infty}_{I}
\] 
i.e. both $P Q_I - \mathbbm{1}$ and $Q_I P - \mathbbm{1}$ are bounded as operators
\[
H^{-N,\rho^I_+-\epsilon-d/2, \rho_\bot}(L) \to H^{N,\rho^I_-+\epsilon-d/2,\rho_\bot}(L),
\]
(resp. $y^{\rho^I_+-\epsilon}C^{-N}_\ast \to y^{\rho^I_- + \epsilon}C^{N}_\ast$) for all $N>0$ and $\epsilon>0$ small enough. In particular, $P$ is Fredholm with same index on each space $H^{s,\rho_0-d/2,\rho_\bot}$ (resp. $y^{\rho_0}C^s_\ast$) for $s \in \R, \rho_0 \in I, \rho_\bot \in \R$.
\end{theorem}

\subsection{The case of elliptic operators and Sobolev spaces}

In this section, we will prove Theorem \ref{theorem:parametrix-compact} in the case where the operators acts on Sobolev spaces. 

\subsubsection{Indicial family}

\label{sssection:indicial-family}

We now introduce the notion of indicial family. Recall that if $L\to N$ is an admissible bundle, in the cusps it decomposes as $L_{|N} = Z \times L_Z$, $L_Z$ being a bundle over $F_Z$. Given $\rho_-< \rho_+$, it will be convenient to introduce the notation 
\[
\C_{(\rho_-,\rho_+)}:= \{\lambda\in \C,\ \Re\lambda \in (\rho_-, \rho_+) \}.
\] 
In this section, we will fix an interval $I= (\rho_-, \rho_+)\subseteq \R$.

\begin{definition}
\label{definition:indicial-family}
Let $A$ be a $I$-admissible operator of order $m$. We introduce for $\lambda \in \C_I$ the operator acting as a map $C^\infty(F_Z,L_Z) \rightarrow C^\infty(F_Z,L_Z)$ and defined for $f \in C^\infty(F_Z,L_Z)$ by:
\begin{equation}
\label{equation:indicial-family}
I_Z(A,\lambda)f(\zeta) = e^{-\lambda r} I_Z(A)\left[ e^{\lambda r'}f(\zeta')\right].
\end{equation}
The family $(I_Z(A,\lambda))_{\lambda \in \C_I}$ is called the \emph{indicial family associated with $A$.}
\end{definition}

For each $\lambda \in \C_I$, $I_Z(A,\lambda)$ is well-defined as an operator acting on $C^\infty(F_Z,L_Z)$, due to the fact that $I_Z(A)$ is a convolution operator in the $r$-variable. Interpreting $I_Z(A)$ as a convolution operator on $\R$ valued in operators on $F$, we see that \emph{$I_Z(A,\lambda)$ is the Fourier transform of the convolution kernel of $I_Z(A)$.}

More precisely, let $\mc{S}(\R \times F_Z)$ be the space of functions such that $f \in \mc{S}(\R \times F_Z)$ if and only if for all $i,j,k \in \N$:
\[
\sup_{r \in \R} \langle r \rangle^i \|\partial_r^j f (r)\|_{C^k(F_Z)} < \infty.
\]
Given $f \in e^{\rho r} \mc{S}(\R \times F_Z, L)$, we can write
\[
f = \int_{\Re(\lambda) = \rho} e^{\lambda r} \hat{f}(\lambda) \dd \lambda,
\]
where $\hat{f}(\lambda) \in C^\infty(F_Z,L)$ is defined by
\[
\hat{f}(\lambda) = \dfrac{1}{2 \pi} \int_{\R_r} e^{- \lambda r} f(r, \cdot) \dd r,
\]
and we have a Plancherel-type formula
\begin{equation}
\label{equation:plancherel}
\|f\|^2_{L^2(\R \times F_Z, \dd r \dd \zeta)} = \dfrac{1}{2\pi} \|\hat{f}\|^2_{L^2(\R_\lambda ; L^2(F_Z))}.
\end{equation}
Then:
\begin{equation}
\label{equation:diagonal-action}
I_Z(A) f = \int_{\Re(\lambda)=\rho} e^{\lambda r} I_Z(A,\lambda) \hat{f}(\lambda) \dd \lambda.
\end{equation}

To give an accurate meaning to the previous integral, recall that the spaces $H^s(\R \times F_Z)$ are the ones induced by the Laplacian $\Delta := \partial_r^2 + \Delta_{\zeta}$ obtained from the metric $\dd r^2 + g_{F_Z}$, i.e.
\[
H^s(\R \times F_Z) := (\mathbbm{1}-\Delta)^{-s/2}L^2(\R \times F_Z, \dd r \dd \zeta).
\]
This shows that the previous partial Fourier transform in the $r$-variable allows to express the Sobolev norms:
\[
\begin{split}
\|f\|^2_{H^s(\R \times F_Z)} & = \|(\mathbbm{1}-\Delta)^{s/2}f\|^2_{L^2(\R \times F_Z)}  = \int_{\Re(\lambda) = 0} \|(\mathbbm{1}-\lambda^2-\Delta_\zeta)^{s/2}\hat{f}(\lambda)\|^2_{L^2(F_Z)} \dd \lambda.
\end{split}
\]

We have a precise description for the indicial family of an admissible operator. In the following, we denote by $\Op^F$ a quantization on $F_Z$.

\begin{lemma}\label{lemma:indicial-pseudo}
Let $A$ be a $I$-admissible operator. Then $I_Z(A,\lambda)$ is $\langle \Im \lambda\rangle^{-1}$-semi-classical, i.e it can be written as $\Op^F_{h}(\tilde\sigma_\lambda)+\mathcal{O}_{\Psi^{-\infty}}(h^\infty)$, where $h = \langle \Im \lambda\rangle^{-1}$. The estimates are locally uniform in $\Re \lambda \in I$.
\end{lemma}

\begin{proof}
For this, we decompose $A = A_{comp}+ \Op(\sigma) + R$, with $A_{comp}$ is a compactly supported pseudo-differential operator, $\Op(\sigma)$ is supported in the cusp satisfying (i) and (ii) of Proposition \ref{prop:relation-symbol-geometric}, and $R$ is $I$-admissible smoothing.

Certainly, $A_{comp}$ does not contribute to $I_Z(A)$, and we start by studying the contribution from $\Op(\sigma)$. Let us express the kernel of $I_Z(\Op(\sigma))$ (in local charts in $F_Z$) as
\[
\int_{\R^k \times\R} e^{i\Phi(r,r',\lambda,z,\eta)} \chi(r-r')\tilde\sigma\left(\frac{z+z'}{2},\lambda,\eta\right)\frac{2 e^{(r+r')/2} }{e^r + e^{r'}} \frac{\dd \eta \dd \lambda}{(2\pi)^{1+k}}.
\]
($\tilde{\sigma}$ was defined in Proposition \ref{prop:relation-symbol-geometric}) with
\[
\Phi = \langle z-z',\eta\rangle + 2 \lambda \tanh\frac{r-r'}{2}.
\]
As a consequence, $I_Z(\Op(\sigma),\lambda)= \Op(\sigma_\lambda)$ with
\[
\sigma_\lambda = \frac{1}{2\pi} \int_{\R^2} e^{-\lambda u + 2i \mu \tanh\frac{u}{2}}\frac{\chi(u)}{\cosh\frac{u}{2}}\tilde{\sigma}(z,\mu,\eta) \dd u \dd \mu.
\]
This integral is stationary at $\mu=i\Im\lambda$, $u=0$, with compact support in $u$, and symbolic estimates in $\mu$. So we get $\sigma_\lambda \in S^m$, with the refined estimates
\begin{equation}
\label{equation:symbol-sigma-lambda}
|\partial_z^\alpha \partial_\eta^\beta \partial_\lambda^\gamma \sigma_\lambda| \leq C_{\alpha,\beta,\gamma} (1+ |\Im\lambda|^2 + |\eta|^2)^{(m-|\beta|-\gamma)/2},
\end{equation}
with constants $C_{\alpha,\beta,\gamma}$ locally uniform in $\Re \lambda$. We deduce from this that $I_Z(\Op(\sigma),\lambda)$ is semi-classical with parameter $h=\langle \Im\lambda\rangle^{-1}$. 

Let us turn now to $I_Z(R,\lambda)$. We need to prove that its kernel, seen as a smooth function on $C^\infty(F_Z \times F_Z)$ is uniformly bounded, as well as all its derivatives, by $\mathcal{O}(\langle\Im\lambda\rangle^{-\infty})$. Since its kernel is the Fourier transform of the convolution kernel of $I_Z(R)$, this polynomial decay at high frequency is a direct consequence of the smoothness of the kernel of $I_Z(R)$. Indeed, to obtain our bound, it suffices to prove for all $k\geq 0$ that uniformly in $\zeta_1,\zeta_2\in F$, and locally uniformly in $\rho\in I$ that
\[
r \mapsto \langle \delta^{(k)}_{\zeta_1}, I_Z(R) \delta_0(r) \delta^{(k)}_{\zeta_2}\rangle_{F} \in e^{\rho r} C^\infty. 
\]
But this follows from the fact that $R$ is $I$-smoothing, and usual Sobolev injections on $\R$. 
\end{proof}

We can now prove
\begin{lemma}
\label{lemma:homomorphism}
Let $A$ be a $I$-admissible operator. Its indicial family
\[
\C_I \ni \lambda \mapsto I_Z(A,\lambda) \in \Psi^m(F_Z,L_Z)
\]
is holomorphic (as a map taking values in a Fréchet space) and is a homomorphism in the sense that for all $I$-admissible operators $P$ and $Q$, for all $\lambda \in \C_I$:
\[
I_Z(PQ,\lambda) = I_Z(P,\lambda) I_Z(Q,\lambda) \qquad I_Z(P+Q,\lambda) = I_Z(P,\lambda) + I_Z(Q,\lambda).
\]
\end{lemma}

\begin{proof}
That $I_Z(A,\lambda)$ is defined for $\lambda\in \C_I$ follows from the description of $I_Z(A,\lambda)$ as the ``Fourier transform'' of $I_Z(A)$. This is explained in much detail in Lemmas 3.16 and 3.17 of \cite{Bonthonneau-Weich-17}.

We now show that if $P,Q$ are $I$-admissible, $I_Z(PQ)=I_Z(P)I_Z(Q)$. To this end, we consider a cutoff function $\chi_{PQ}$ supported on $\left\{\chi_P = 1\right\} \cap \left\{\chi_Q =1 \right\}$ taking constant value $1$ as $y \rightarrow +\infty$. Then, using Definition \ref{definition:rho-admissible} of admissible operators, together with Proposition \ref{prop:algebra-general}:
\begin{align*}
\mc{E}_Z \chi_{PQ} \mathcal{P}_Z PQ \mathcal{E}_Z \chi_{PQ} \mc{P}_Z	&= \mc{E}_Z \chi_{PQ} \mathcal{P}_Z P \mc{E}_Z \chi_{P} \mathcal{P}_Z  Q \mathcal{E}_Z \chi_{PQ} \mc{P}_Z \\
&   +  \underbrace{\left(\mc{E}_Z \chi_{PQ} \mathcal{P}_Z P  (\mathbbm{1}-\mc{E}_Z \chi_{P} \mathcal{P}_Z)\right)}_{\text{in }\dot{\Psi}^\infty_I \text{, by Proposition \ref{prop:algebra-general}} } Q \mathcal{E}_Z \chi_{PQ} \mc{P}_Z\\
& \hspace{-10pt}= \chi_{PQ} \underbrace{\left(\mc{E}_Z \chi_P \mc{P}_Z P \mc{E}_Z \chi_P \mc{P}_Z \right)}_{=\mc{E}_Z \chi_P I_Z(P) \chi_P \mc{P}_Z  \mod \dot{\Psi}^\infty_I} Q \mathcal{E}_Z \chi_{PQ} \mc{P}_Z  \mod \dot{\Psi}^\infty_I \\
& \hspace{-10pt}= \chi_{PQ} \mc{E}_Z \chi_P I_Z(P) \chi_P \mc{P}_Z  \mc{E}_Z \chi_Q I_Z(Q) \chi_Q \mc{P}_Z \mathcal{E}_Z \chi_{PQ} \mc{P}_Z \ \mod \dot{\Psi}^\infty_I \\
& \hspace{-10pt}=  \mc{E}_Z \chi_{PQ}  I_Z(P) \chi_P  \chi_Q I_Z(Q)  \chi_{PQ} \mc{P}_Z  \mod \dot{\Psi}^\infty_I.
\end{align*}
We now need to explain how to get rid off the cutoff function $\chi_P  \chi_Q $ in the middle. To that end, it suffices to show that both
\[
\mc{E}_Z \chi_{PQ}  I_Z(P)(\mathbbm{1}- \chi_P)  \chi_Q I_Z(Q)  \chi_{PQ} \mc{P}_Z \text{ and } \mc{E}_Z \chi_{PQ}  I_Z(P)\chi_P  (\mathbbm{1}- \chi_Q) I_Z(Q)  \chi_{PQ} \mc{P}_Z
\]
are $I$-residual. We deal with the first term for instance. First of all, observe that due to the support properties of $\chi_{PQ}$ and $\chi_P$ and the pseudo-locality of $I_Z(P)$, the operator $\chi_{PQ} I_Z(P)(\mathbbm{1}- \chi_P)$ has smooth Schwartz Kernel and maps 
\[
e^{r\rho} H^{-k}(\R \times F_Z) \rightarrow e^{r \rho'} H^k(\R \times F_Z)
\]
for any $k \in \N$, $\rho \in \R$, $\rho'\in I$. Moreover, it is immediate that $(\mathbbm{1}- \chi_P)  \chi_Q I_Z(Q)  \chi_{PQ} \mc{P}_Z  : y^\rho H^s(N,L) \rightarrow \min (e^{\rho' r}, 1) H^s(\R_r \times F_Z)$ is bounded for all $k\in \N$, $\rho\in I$, $\rho'\in \R$. Combining these two facts, we obtain that $\mc{E}_Z \chi_{PQ}  I_Z(P)(\mathbbm{1}- \chi_P)  \chi_Q I_Z(Q)  \chi_{PQ} \mc{P}_Z$ is $I$-residual.
\end{proof}

It is good to have in mind the following bound:
\begin{lemma}
\label{lemma:l2-bound}
Consider $A \in \Psi^0(N,L)$, and $a < b$. Then, there exists $C = C(A)$ such that for all $a < \Re(\lambda) < b$, $\|I_Z(A,\lambda)\|_{\mc{L}(L^2,L^2)} \leq C$.
\end{lemma}

\begin{proof}
This follows from Lemma 3.16 in \cite{Bonthonneau-Weich-17}.
\end{proof}

\subsubsection{Indicial resolvent}

\label{sssection:indicial-resolvent}

In this section, we now pick $I$ as before, $A$ a $I$-admissible operator. Additionally, we assume that $A$ is \emph{uniformly elliptic}, and we find an inverse for $I_Z(A)$.
\begin{proposition}\label{prop:indicial-resolvent-elliptic}
Let $A$ be an $I$-admissible operator. Then $I_Z(A,\lambda)$ is Fredholm of index $0$ on every space $H^s(F_Z,L_Z)$, $s \in \R$, and invertible for $\Im\lambda$ large enough, locally uniformly in $\Re \lambda\in I$. We call \emph{indicial roots} of $A$ (at cusp $Z$) the $\lambda$'s in $\C_I$ such that $I_Z(A,\lambda)$ is not invertible. We also let
\begin{equation}\label{eq:def-S}
S(A):= \left\{ \rho \in I \ \middle|\ \text{ there exists an indicial root of $A$ in $\rho + i \R$} \right\} \subset I.
\end{equation}
This set is discrete in $I$.
\end{proposition}

\begin{proof}
For this we recall that since $A$ is uniformly elliptic, we have built a parametrix $Q$ such that $QA = \mathbbm{1} + R$, where $R$ is a $I$-smoothing. Since $Q$ is constructed as $\Op(\sigma)$, and $\sigma$ is built with symbolic operations from the symbol of $A$, we can impose that $Q$ is $\R$-admissible. We deduce that 
\[
I(Q,\lambda)I(A,\lambda) = \mathbbm{1} + I_Z(R).
\]
Since $R$ is $I$-admissible smoothing, $I_Z(R) = \mathcal{O}(\langle\Im \lambda\rangle^{-\infty})$, and this implies that $I(A,\lambda)$ is left invertible for $\Im \lambda$ large enough, locally uniformly in $\Re \lambda$. Since $Q$ is also a right parametrix, we deduce that $I(A,\lambda)$ is invertible for $\Im \lambda$ large enough, and
\begin{equation}\label{eq:parametrix-better-highlambda}
I(A,\lambda)^{-1}  = I(Q,\lambda) + \mathcal{O}_{\Psi^{-\infty}}(\langle \Im \lambda\rangle^{-\infty}).
\end{equation}
Using the analytic Fredholm theorem, we deduce that for $\lambda \in \C_I$, $I(A,\lambda)^{-1}$ is a meromorphic family of pseudo-differential operators, with poles of finite rank. The discreteness of $S(A)$ follows from the invertibility of $I(A,\lambda)$ for $\Im \lambda$ large enough.
\end{proof}

\begin{proposition}
Let $A$ be an $I$-admissible elliptic operator of order $m$. Then for $\rho\in I\setminus S(A)$, $s\in \R$, $I_Z(A)$ is invertible as a map $e^{\rho r} H^s(\R\times F,L_F) \to e^{\rho r} H^{s-m}(\R\times F,L_F)$. The inverse can be expressed as
\[
I_Z(A)^{-1}_\rho = \int_{\Re \lambda = \rho} e^{\lambda r} I_Z(A,\lambda)^{-1} d\lambda. 
\]
(this is an equality of convolution kernels). It is pseudo-differential, and as an operator $C^\infty_{\mathrm{comp}}(\R \times F, L_F) \to C^\infty(\R \times F, L_F) $, it does not depend on $\rho$, as long as $\rho$ varies in a connected component of $I\setminus S(A)$.
\end{proposition}

\begin{proof}
For this, we will start by considering the operator in the right-hand side of the formula (denote it $W$), and prove that this is a bounded operator from $e^{\rho r} H^{s-m}(\R\times F,L_F)$ to $H^{s}(\R\times F,L_F)$. Up to conjugating by some powers of the Japanese bracket of the Laplacian, the statement boils down to $s=m=0$ and one can also take $\rho=0$.\footnote{Indeed, using the fact that $\Delta_{r,\zeta} := \partial_r^2 + \Delta_\zeta = I_Z(\Delta_g)$, namely it is the indicial operator associated to the Laplacian on $N$, we obtain by a straightforward computation that the bound
\[
\|I_Z(A)f\|_{H^s} \lesssim \|f\|_{H^{s+m}}
\]
is equivalent to the bound
\[
\|I_Z((\mathbbm{1} - \Delta_g)^{s/2}A(\mathbbm{1} - \Delta_g)^{-(s+m)/2})f\|_{L^2} \lesssim \|f\|_{L^2}
\]
and this last operator is of order $0$.} Then, we observe that on the line $\Re \lambda = 0$, the operator $I_Z(A,\lambda)^{-1}$ is uniformly bounded on $L^2$ by Lemma \ref{lemma:l2-bound}. We can thus apply Plancherel's formula \eqref{equation:plancherel}, and obtain:
\[
\|Wf\|^2_{L^2} = \int_{\Re(\lambda)=0} \|I_Z(A,\lambda)^{-1}\hat{f}(\lambda)\|^2_{L^2(F_Z)} \dd \lambda \lesssim  \int_{\Re(\lambda)=0} \|\hat{f}(\lambda)\|^2_{L^2(F_Z)} \dd \lambda = \|f\|^2_{L^2}.
\]
This bound also ensures the validity of the computation
\[
I_Z(A) W = W I_Z(A) = \int_{\Re \lambda =0} e^{\lambda r} I_Z(A,\lambda)^{-1} I_Z(A,\lambda) d\lambda = \mathbbm{1}.
\]

Next, to check that $I_Z(A)^{-1}_{\rho}$ is pseudo-differential, we observe that since 
\[
I_Z(Q) I_Z(A) = \mathbbm{1} \mod \Psi^{-\infty}
\]
($Q$ being the parametrix already defined), we have necessarily
\[
I_Z(A)_\rho^{-1} = I_Z(Q) \mod \Psi^{-\infty}.
\]

Finally, since $I_Z(A)^{-1}_\rho$ is expressed as a contour integral, testing against $C^\infty_{\mathrm{comp}}$ function (whose Fourier transform in the $r$ variable decay very fast), we can shift contours as long as we do not encounter a pole of $I_Z(A,\lambda)^{-1}$, using Cauchy's formula.
\end{proof}

\subsubsection{Improving Sobolev parametrices}

\label{sec:improving-I}

In this section, we will complete the proof of Theorem \ref{theorem:parametrix-compact} in the case that the operator $A$ is $(\rho_-,\rho_+)$-admissible elliptic (except the part about the Fredholm index \eqref{equation:fredholm-index-theorem} that we will deal with in the next section). Recall that $A = A_{\mathrm{comp}} + A_{\mathrm{cusp}} + R$. Both $A_{\mathrm{cusp}}, R$ are associated with some cutoff function $\chi_{A_{\mathrm{cusp}}},\chi_R$ (as it was introduced in Definitions \ref{def:R-L^2-admissible-operator} and \ref{definition:rho-admissible}). In the following, we will write $\chi_A$ for a smooth cutoff function such that $\chi_A$ is supported in $\left\{\chi_R \equiv 1\right\} \cap \left\{\chi_{A_{\mathrm{cusp}}}\right\}$ and it takes value $1$ near $y=+\infty$.

According to Proposition \ref{prop:smoothing-remainder}, we have a symbol $q$ such that $\Op(q)A - \mathbbm{1}$ is a smoothing operator (but not necessarily compact). From Lemma \ref{lemma:compact-injection}, we deduce that it would suffice to improve $\Op(q)$ only with respect to the action on the zeroth Fourier coefficient in the cusps. As we observed in the previous section, since the symbol $q$ was built using symbolic calculus, we deduce directly that $\Op(q)$ is $\R$-admissible. Consider an open interval $J$ which is a connected component of $I\setminus S(A)$. Denote by $I_Z(A)^{-1}_{J}$ the corresponding inverse of $I_Z(A)$. Then set
\[
Q_J := \Op(q) + \sum_Z \mathcal{E}_Z \chi_A \left[ I_Z(A)^{-1}_{J} - I_Z(\Op(q)) \right] \chi_A\mathcal{P}_Z,
\]
which is now a $J$-admissible pseudo-differential operator. Indeed, to prove this it suffices to prove that the correction we added is $J$-smoothing. Since $I_Z(A)^{-1}_{J} - I_Z(\Op(q)) \in \Psi^{-\infty}$ is a convolution operator, using the Plancherel formula, this follows from the fact that the estimate
\[
I_Z(A,\lambda)^{-1} - I_Z(\Op(q),\lambda) = \mathcal{O}_{\Psi^{-\infty}}(\langle\Im \lambda\rangle^{-1}),
\]
is valid locally uniformly for $\Re \lambda \in J$.

We now write $Q_J A = \mathbbm{1} + R_J$ and we aim to prove that $R_J$ is compact on $H^{s,\rho_0-d/2,\rho_\bot}$ for all $s\in \R$, $\rho_0 \in J$, $\rho_\bot \in \R$. By stability by composition of admissible pseudo-differential operators (see Proposition \ref{prop:algebra-general}), we know that $R_J$ is a smoothing admissible operators. Moreover, the operator $Q_I$ was chosen so that $I_Z(R_J) = 0$ (this can be checked using the calculation rules of Lemma \ref{lemma:homomorphism}). As a consequence, thanks to Lemma \ref{lemma:compact-injection}, the proof of Theorem \ref{theorem:parametrix-compact} (except the Fredholm properties) now boils down to the following Lemma:

\begin{lemma}
\label{lemma:refinement}
Let $A$ be a $(\rho_-,\rho_+)$-admissible pseudo-differential operator of order $-m$ such that $I_Z(A) = 0$. Then $A$ is bounded from $H^{s,\rho_0-d/2,\rho_\bot}$ to $H^{s+m,\rho_0' - d/2,\rho_\bot}$ for $\rho_0,\rho_0' \in (\rho_-,\rho_+)$, $\rho_\bot\in \R$.
\end{lemma}

\begin{proof}
We have:
\[
\begin{split}
Af & = (\mathbbm{1}-\mc{E}_Z \chi_A \mc{P}_Z)A(\mathbbm{1}-\mc{E}_Z\chi_A\mc{P}_Z)f \\
& + \mc{E}_Z\chi_A\mc{P}_Z A(\mathbbm{1}-\mc{E}_Z\chi_A \mc{P}_Z)f + (\mathbbm{1}-\mc{E}_Z\chi_A\mc{P}_Z)A\mc{E}_Z\chi_A\mc{P}_Zf + \mc{E}_Z\chi_A\mc{P}_Z A\mc{E}_Z\chi_A\mc{P}_Zf
\end{split}
\]
By definition of being admissible, the first three terms directly satisfy the announced bounds. The last one also does since we have assumed that $I_Z(A)=0$.
\end{proof}

\subsubsection{Fredholm index of elliptic operators}

We first prove the following identification:

\begin{lemma}
\label{lemma:l2-identification}
For all $s,\rho_0,\rho_\bot \in \R$, one can identify via the $L^2$ scalar product the spaces $(H^{s,\rho_0,\rho_\bot})' \simeq H^{-s,-\rho_0,-\rho_\bot}$.
\end{lemma}

\begin{proof}
We have to prove that the bilinear map
\begin{equation}
\label{equation:bilinear-mapping}
C^\infty_{\mathrm{comp}}(N,L) \times C^\infty_{\mathrm{comp}}(N,L) \ni (u,v) \mapsto \langle u,v \rangle = \int_{N} g^L(u, v) d\vol_N(z)
\end{equation}
extends boundedly as a non-degenerate bilinear map $H^{s,\rho_0,\rho_\bot} \times H^{-s,-\rho_0,-\rho_\bot} \rightarrow \C$. As usual, non-degeneracy is trivial. Up to a smoothing order modification of $\Lambda_s$ which we denote by $\Lambda'_s$, we can assume that $\Lambda_{-s} \Lambda'_s = \mathbbm{1}$. Then, for $u, v \in C^\infty_{c}(N,L)$, one has $\langle u,v \rangle = \langle \Lambda_{-s} \Lambda'_s u,v \rangle = \langle \Lambda_s u, \Lambda'_{-s} v \rangle$. By Lemma \ref{lemma:boundedness-admissible-pseudo}, since $\Lambda_{\pm s}$ is admissible, $\Lambda_{\pm s} : H^{\pm s,\rho_0,\rho_\bot} \rightarrow H^{0,\rho_0,\rho_\bot}$ is bounded. The boundedness of (\ref{equation:bilinear-mapping}) on $H^{0,\rho_0,\rho_\bot} \times H^{0,-\rho_0,-\rho_\bot} \rightarrow \C$ is immediate (these are $L^2$ spaces with weight $y^{\rho_0}$ on the zeroth Fourier mode and $y^{\rho_\bot}$ on the non-zero modes) and thus:
\[
|\langle \Lambda_s u, \Lambda'_{-s} v \rangle| \lesssim \|\Lambda_s u\|_{H^{0,\rho_0,\rho_\bot}} \|\Lambda'_{-s} v\|_{H^{0,-\rho_0,-\rho_\bot}} \lesssim \|u\|_{H^{s,\rho_0,\rho_\bot}} \|v\|_{H^{-s,-\rho_0,-\rho_\bot}}.
\]
We then conclude by density of $C^\infty_{\mathrm{comp}}(N,L)$.
\end{proof}

An immediate computation shows that 
\begin{equation}
\label{equation:indice-adjoint}
I_Z(P^*,\lambda) = I_Z(P,d-\overline{\lambda})^*.
\end{equation}
As a consequence, $\lambda$ is an indicial root of $P^\ast$ if and only if $d-\overline{\lambda}$ is an indicial root of $P$.

\begin{proposition}
\label{proposition:fredholm-index}
Let $P$ be a $(\rho_-,\rho_+)$-admissible elliptic pseudo-differential operator of order $m \in \R$. Let $I$ be a connected component in $(\rho_-,\rho_+)$ not containing the real part of any indicial root. Then $P$ is Fredholm as a bounded operator $H^{s+m,\rho_0 - d/2,\rho_\bot} \to H^{s,\rho_0 - d/2,\rho_\bot}$ with $s \in \R, \rho_0 \in I$, $\rho_\bot \in\R$. The index does not depend on $s,\rho_0,\rho_\bot$ in that range.
\end{proposition}

\begin{proof}
We write $I = (\rho^I_-, \rho^I_+)$. First, from the parametrix construction, and the compactness of the relevant spaces, we deduce that the kernel of $P$ is finite dimensional on each of those spaces (and is actually always the same). Indeed, we have 
\[
Q P = \mathbbm{1} + K,
\]
with $K$ mapping $H^{-k, \rho^I_+ -\epsilon - d/2,\rho_\bot} \to H^{k, \rho^I_- +\epsilon - d/2,\rho_\bot}$ for any $k \in \N$, any $\epsilon>0$ small enough and any $\rho_\bot \in \R$, by Lemma \ref{lemma:refinement}. This can actually be upgraded to 
\[
K : H^{-k, \rho^I_+ -\epsilon - d/2,\rho_\bot} \to H^{k, \rho^I_- +\epsilon - d/2,-\infty}
\]
is bounded (in the sense that $K : H^{-k, \rho^I_+ -\epsilon - d/2,\rho_\bot} \to H^{k, \rho^I_- +\epsilon - d/2,-k}$ is bounded for all $k \in \N, \eps > 0$). Indeed, by Lemma \ref{lemma:embedding-nonzero}, we know that 
\[
H^{k, \rho^I_- +\epsilon - d/2,\rho_\bot} \hookrightarrow H^{k/2,  \rho^I_- +\epsilon - d/2, \rho_\bot - k/2},
\]
is a continuous embedding and this proves the claim.

As a consequence, if $Pu = 0$ with $u \in H^{s+m,\rho_0 - d/2,\rho_\bot}$ for $s \in \R$ and $\rho_0 \in I, \rho_\bot \in \R$, then $u = -Ku$, hence
\[
u \in \cap_{k \in \N, \eps > 0,\rho_\bot \in \R} H^{k, \rho^I_- +\epsilon - d/2,\rho_\bot}(N,L_1),
\]
that is the kernel of $P$ is independent of the space (as long as the weight $y^\rho_0$ on the zeroth Fourier mode is in the window $\rho_0 \in I$). Moreover, by compactness of $K$, the kernel of $\mathbbm{1}+K$ is finite-dimensional, and so is the kernel of $P$.

Eventually, using Lemma \ref{lemma:l2-identification}, we can consider the same argument for the adjoint $P^\ast$ (to obtain the codimension of the image of $P$), and this closes the proof.
\end{proof}

\begin{remark}
If on the other hand, $\rho\in S(A)$, then the operator $A : H^{s+m,\rho,\rho_\bot} \rightarrow H^{s,\rho,\rho_\bot}$ is \emph{not Fredholm} for any values $s,\rho_\bot \in \R$. Although we do not make a formal statement of this fact, it could be deduced from the \emph{relative Fredholm index Theorem} (see \eqref{equation:fredholm-index-theorem}) obtained in the next paragraph: indeed, the family $\rho_0 \mapsto y^{\rho_0} A y^{-\rho_0} \in \Psi^m$ is continuous and the relative Fredholm index Theorem indicates that the Fredholm index jumps when crossing an indicial root. By a mere continuity argument, this prevents the operator from being Fredholm.
\end{remark}

\subsubsection{Jumps of Fredholm index}

Let $A$ be an elliptic pseudo-differential operator of order $m > 0$ and assume that it is $\R$-admissible. We want to investigate what happens when one crosses an indicial root: the operator may fail to be injective and/or surjective. For the sake of simplicity, we assume that the operator $A$ has no indicial root on $\Re(\lambda)=d/2$ and that it is an isomorphism as a map $H^{s,\rho,\rho_\bot} \rightarrow H^{s-m,\rho,\rho_\bot}$ for all $s \in \R, \rho_\bot \in \R$ and $\rho$ in a neighbourhood of $0$. Let us investigate its kernel: we consider $u \in H^{0,\rho_0,\rho_\bot}$ such that $Au=0$, where $\rho_0 > 0$ and we assume that $\rho_0+d/2$ is not an indicial root. By Proposition \ref{proposition:fredholm-index}, it implies in particular that $u \in H^{+\infty,\rho_0,-\infty}$ and we recall that this notation means that $u \in H^{k,\rho_0,-k}$ for all $k \in \N$. Moreover, we have
\[
\begin{split}
Au & = 0 =  (\mathbbm{1}-\mc{E}_Z\chi_A \mc{P}_Z)A(\mathbbm{1}-\mc{E}_Z\chi_A\mc{P}_Z)u \\
& + \mc{E}_Z\chi_A \mc{P}_Z  A(\mathbbm{1}-\mc{E}_Z \chi_A \mc{P}_Z)u + (\mathbbm{1}-\mc{E}_Z \chi_A \mc{P}_Z )A\mc{E}_Z \chi_A \mc{P}_Z  u  + \mc{E}_Z\chi_A \mc{P}_Z A\mc{E}_Z \chi_A \mc{P}_Z  u
\end{split}
\]
Since $A$ is $\R$-admissible, the first three terms are respectively in
\[
H^{+\infty,-\infty,\rho_\bot}, H^{+\infty,-\infty,-\infty}, H^{+\infty,-\infty,-\infty}.
\]
In particular, this implies that
\[
\mc{E}_Z \chi_A \mc{P}_Z A \mc{E}_Z \chi_A \mc{P}_Z u =\mc{E}_Z \chi_A I_Z(A) \chi_A \mc{P}_Z u + \mc{O}_{y^{-\infty}H^{\infty}}(1) = \mc{O}_{y^{-\infty}H^{\infty}}(1).
\]
that is $I_Z(A) \chi_A \mc{P}_Z u \in \cap_{\rho \in \R} e^{\rho r} H^\infty(\R \times F_Z, L)$. Since $\rho_0+d/2$ was assumed not to be an indicial root, $I_Z(A)$ is invertible on $e^{\rho_0+d/2} H^{s+m}(\R \times F_Z, L) \rightarrow e^{\rho_0+d/2} H^{s}(\R \times F_Z, L)$, for all $s \in \R$ with inverse $I_Z(A)^{-1}_{\rho_0+d/2}$ and the Schwartz kernel of this inverse does not depend on a small perturbation on $\rho_0$. By Lemma \ref{lemma:relation-projected-spaces}, we have 
\[
f:=\chi_A \mc{P}_Z  u \in e^{(\rho_0+d/2)r} H^{\infty}.
\]
In particular, we deduce that $I_Z(A)^{-1}_{\rho_0+d/2} I_Z(A) f=f$. We also have that $I_Z(A) f \in e^{(\rho+d/2)r}H^\infty$ for all $\rho\in\R$. On the other hand, we know by a classical contour integration argument that 
\[
I(A)^{-1}_{\rho_0+d/2} = I(A)^{-1}_{d/2} + 2i\pi\sum_{\substack{\lambda\text{ indicial root of }A \\ \Re \lambda\in ]d/2, d/2+\rho_0[}} \Pi_\lambda.
\]
Here, $\Pi_\lambda$ is the convolution operator whose kernel is the residue of $e^{\lambda'(r-r')}I_Z(A,\lambda')^{-1}$ at $\lambda'=\lambda$. It is a finite rank operator, whose image is the linear span of sections of the form
\begin{equation}\label{eq:eigensections-at-infinity}
e^{\lambda r} r^k f_k(\zeta),
\end{equation}
($k$ being at most the order of the pole of $I_Z(A,\lambda)$). This implies, using the boundedness properties of $I(A)^{-1}_{d/2}$ that
\[
\begin{split}
f & = I(A)^{-1}_{d/2}(A)I_Z(A)f +  \sum_{\substack{\lambda\text{ indicial root of }A \\ \Re \lambda\in ]d/2, \rho_0[}} \Pi_\lambda I_Z(A)f \\
& =  \sum_{\substack{\lambda\text{ indicial root of }A \\ \Re \lambda\in ]d/2, \rho_0[}} \Pi_\lambda I_Z(A)f \mod e^{dr/2}H^\infty.
\end{split}
\]
Going back to $u$ and writing (note that we put twice the cutoff function here so that each term below makes sense):
\[
\begin{split}
u & = (\mathbbm{1}- \mc{E}_Z\chi_A  \chi_A \mc{P}_Z)u + \mc{E}_Z \chi_A   f \\
& =  (\mathbbm{1}- \mc{E}_Z \chi_A \chi_A \mc{P}_Z)u + \mc{E}_Z  \chi_A \mc{O}_{e^{\frac{d}{2} r}H^\infty}(1) +  \sum_{\substack{\lambda\text{ indicial root of }A \\ \Re \lambda\in ]d/2, \rho_0[}} \mc{E}_Z \chi_A  \Pi_\lambda I_Z(A)f,
\end{split}
\]
we see that the first term is $H^{+\infty, - \infty,-\infty}$ and the second term is in $H^{+\infty,0,-\infty}$. In other words, we can write $u =u_0 + u_1$, where $u_0 \in H^{+\infty, 0,\rho_\bot}$ and
\begin{equation}
\label{equation:u1}
 u_1 = \mc{E}_Z  \sum_{\substack{\lambda\text{ indicial root of }A \\ \Re \lambda\in ]d/2, \rho_0[}} \chi_A \Pi_\lambda \left(I_Z(A) \mc{P}_Z \chi u\right),
\end{equation}
which belongs to a finite-dimensional space. On the other hand,
\[
 Au_1 = Au - Au_0=- Au_0 \in H^{+\infty,-\infty,-\infty}.
\]
By invertibility of $A$ for sections in $H^{s,0,\rho_\bot}, s\in\R,\rho_\bot \in \R$, we obtain that $u_0 = -A^{-1}(Au_1) \in H^{+\infty,0,-\infty}$. In other words, any solution to $Au = 0$, where $u \in H^{+\infty,\rho_0,-\infty}$, can be written as $u = u_0 + u_1$, where $u_0 \in H^{+\infty,0,-\infty}$ and $u_1$ has the explicit form \eqref{equation:u1}.

Now, the converse is also true. For any element in the range of $\Pi_\lambda$, where $0 < \Re(\lambda) < \rho_0$, we can build a solution to $Au = 0$. Indeed, first of all observe that, taking $\lambda$ an indicial root and $\rho > \Re(\lambda)$, and considering $I_Z(A)$ acting on the spaces $e^{\rho r} H^s$, we have $I_Z(A) \Pi_\lambda = 0$ by a simple contour integral argument. As a consequence, defining $u_1 := \chi_A \mc{E}_Z u'$, where $u' \in \ran(\Pi_\lambda)$, we have by construction $A u_1 \in H^{+\infty,-\infty,-\infty}$. We then look for $u= u_0 + u_1$ such that $Au = 0$, $u_0 \in H^{+\infty,0,-\infty}$. For that, it is sufficient to solve $Au_0 = -Au_1 \in H^{+\infty,-\infty,-\infty} \hookrightarrow H^{+\infty,0,-\infty}$. Now, since $A$ is assumed to be invertible on spaces with weight $y^0$ on the zeroth Fourier mode, we obtain $u_0 = -A^{-1}(Au_1) \in H^{+\infty,0,-\infty}$ and $u=u_0+u_1$ solves $Au=0$. To sum up the discussion, we have proved the

\begin{proposition}\label{prop:bigger-kernel}
Let $A$ be a $\R$-admissible elliptic operator of order $m > 0$. Assume that $A$ has no indicial root on $\left\{ \Re(\lambda) = d/2 \right\}$ and that it is invertible as a map $A : H^{s+m, \rho_0, \rho_\bot}(N,L) \rightarrow  H^{s, \rho_0, \rho_\bot}(N,L)$ for $\rho_0$ near $0$ ($s,\rho_\bot \in \R$ being arbitrary). Assume that $Au=0, u \in H^{0,\rho_0,\rho_\bot}$ with $\rho_\bot \in \R$ and $\rho_0+d/2$ not being the real part of an indicial root. Then $u=u_0+u_1$ with
\[
u_1  =   \sum_{\substack{\lambda\text{ indicial root of }A \\ \Re \lambda\in ]d/2, d/2+\rho_0[}} \mc{E}_Z \chi_A \Pi_\lambda I_Z(A) \mc{P}_Z \chi_A u
\]
and $Au_1 \in H^{+\infty,-\infty,-\infty}$, and $u_0=-A^{-1}(Au_1) \in H^{+\infty,0,-\infty}$ (in particular, $u\in H^{\infty,\rho_0,-\infty}$). Conversely, for each indicial root $\lambda$ with $\Re \lambda>d/2$, and each element in the range of $\Pi_\lambda$, we can build such a solution.
\end{proposition}

We also have a similar statement for the resolution of the equation $Au=v$ on smaller spaces than $H^{s,0,\rho_\bot}$.

\begin{proposition}\label{prop:smaller-image}
Let $A$ be a $\R$-admissible elliptic operator of order $m > 0$. Assume that $A$ has no indicial root on $\left\{ \Re(\lambda) = d/2 \right\}$ and that it is invertible as a map $A : H^{s+m, \rho_0, \rho_\bot}(N,L) \rightarrow  H^{s, \rho_0, \rho_\bot}(N,L)$ for $\rho_0$ near $0$ ($s,\rho_\bot \in \R$ being arbitrary). Let $\rho_0 <0$ and assume that $\rho_0+d/2$ is not the real part of an indicial root (in particular, there is no indicial root on $(\rho_0-\eps,\rho_0+\eps)$ for some $\eps >0$). Then, there exists $S \in \Psi^{-m}$, an $(\rho_0-\eps,\rho_0+\eps)$-$L^2$-admissible operator, a linear mapping $G : H^{s,\rho_0,\rho_\bot} \rightarrow e^{\rho r}H^{+\infty}$, bounded on these spaces for all $s, \rho , \rho_\bot \in \R$, such that for all $v \in H^{s,\rho_0,\rho_\bot}, s\in\R, \rho_\bot \in \R$, one has:
\[
A^{-1}v = S v + \chi \mc{E}_Z \sum_{\substack{\lambda\text{ indicial root of }A \\ \Re \lambda\in ]\rho_0, d/2+\rho_0[}} \Pi_\lambda (\mc{P}_Z\chi + G)v.
\]
Moreover, one has $A \chi \mc{E}_Z \Pi_\lambda : e^{r \rho}H^{s} \rightarrow H^{+\infty,-\infty,-\infty}$ for all $\rho < \Re(\lambda)$. 
\end{proposition}

\begin{proof}
Since $A$ is assumed to be invertible on the spaces $H^{s,0,\rho_\bot}, s\in\R,\rho_\bot \in \R$, given $v \in H^{s,\rho_0,\rho_\bot}$ for $\rho_0 < 0$, the equation $Au=v$ admits a solution $u \in H^{s+m,0,\rho_\bot}$ and one needs to prove that $u$ is actually more decreasing than this. The proof follows the same arguments as the ones given in the proof of Proposition \ref{prop:bigger-kernel}, namely one has to solve in the full cusp the equation $I_Z(A)\tilde{u}=\tilde{f}$, where $\tilde{f} \in e^{r(\rho_0+d/2)}H^{s}$ and $\tilde{u}$ is a priori in $e^{rd/2}H^{s+m}$.
\end{proof}

Finally, putting together Propositions \ref{prop:bigger-kernel} and \ref{prop:smaller-image}, we deduce that the Fredholm index of $A$ acting on $H^{s,\rho,\rho_\bot}$, when $\rho>0$ is not the real part of an indicial root, is 
\begin{equation}
\label{equation:fredholm-index-theorem}
\mathrm{ind}(A,H^{s,\rho,\rho_\bot}) = \sum_{\Re \lambda\in ]d/2, d/2+\rho[} \mathrm{rank}(\Pi_\lambda).
\end{equation}
If $\rho<0$, this is minus the sum for $\Re \lambda\in ]d/2+\rho,d/2[$.

\subsection{Generalizations}

The results of the previous section can be extended in several direction that we will explore. We will start by considering the case where the operator is only left or right elliptic. Then, we will consider the action on H\"older-Zygmund spaces. Throughout the section, $A$ denotes an $I$-admissible operator of order $m$, $I$ being some open interval in $\R$.

\subsubsection{The case of left/right-elliptic operators}

\label{sec:leftright-parametrix}

In this section, we will consider the case of left-elliptic operators, the case of right-elliptic operators being mostly similar. We obtain the following extension of Theorem \ref{theorem:parametrix-compact}:
\begin{theorem}\label{theorem:left-parametrix}
Let $A$ be an $I$-admissible pseudo-differential operator, uniformly left-elliptic. Then there exists a discrete set $S(A,W)\subset I$ (depending on the construction of a \emph{tempered inverse} $W$, see below) and for each connected component $J$ of $I\setminus S(A,W)$ an operator $Q_J$, which is $J$-extended admissible (in the sense of Definition \ref{definition:extension}) such that
\[
Q_J A = \mathbbm{1} \mod \dot{\Psi}^{-\infty}_J.
\]
\end{theorem} 

In the constructions of the previous section, a crucial step was to find an exact inverse for the indicial operator of $A$. Here, we will have to settle for a left inverse. However, since the indicial family will only be left-invertible a priori, there may be \emph{several left-inverses} of interest. It follows that there are various ways of constructing the parametrix. 

We have identified three different techniques:
\begin{enumerate}
	\item We construct a $C^\infty$ parametrix $Q$, and observe that $I(Q,\lambda)I(A,\lambda) = \mathbbm{1} + \mathcal{O}_{\Psi^{-\infty}}(\langle \Im \lambda\rangle^{-\infty})$ as before, so that $(I(Q,\lambda)I(A,\lambda))^{-1} I(Q,\lambda)$ is a left inverse for $I(A,\lambda)$. We can then build the associated convolution operator.
	\item Let us assume that $d/2 \in I$. We can consider $W:= A^\ast A $, which is $J$-admissible for an interval $J$ containing $d/2$, and is elliptic. Then, we can build a parametrix $Q$ for $W$ as in the previous section and observe that $ (Q A^\ast) A  =  \mathbbm{1} \mod \dot{\Psi}^{-\infty}_{J'}$, where $J' \subset J$ is a subinterval of $J$ without any (projection on the real axis of) indicial roots for $W$.
	\item Finally, it may be more convenient to study directly $I(A,\lambda)$, and find an operator $W(\lambda)$ such that $W(\lambda) I(A,\lambda)=\mathbbm{1}$. Then, the question is to understand what are the conditions required on $W(\lambda)$ so that there exists an admissible operator $Q$ on $M$ such that $I_Z(Q,\lambda)=W(\lambda)$ and $QA = \mathbbm{1} \text{ mod } \dot{\Psi}^{-\infty}_I$. 
\end{enumerate}

This suggests important remarks. First of all, the validity of constructions (1) and (2) follows directly from the results in the previous section, as one recovers the usual elliptic case, so there is nothing to prove, and the result is an \emph{admissible} parametrix (instead of extended admissible). Nevertheless, (1) and (2) are clearly \emph{suboptimal} insofar as they may create \emph{artificial indicial roots} for the operator $A$ (this is exactly what happens for the symmetric derivative $D$ studied in Lemma \ref{lemma:D-elliptic} if one applies construction (2) for instance), and this is due to the fact that there are many possible choices for a left inverse of $I_Z(A,\lambda)$, and some have different poles than others.

We will concentrate on explaining how (3) can be made to work. We can already announce that in (3), we will indeed construct such an operator $Q$ but \emph{we are unable to build it so that it is pseudo-differential}. This is due to that fact that we will have to do some surgery on the zeroth Fourier mode and this operator $Q$ will naturally belong to an \emph{extension of small pseudo-differential operators}, denoted by $\overline{\Psi}^m_{I}$ (see Definition \ref{definition:extension}). Nevertheless, this parametrix operator for $A$ will enjoy all the main properties of pseudo-differential operators (e.g. boundedness on Sobolev spaces) and this will be sufficient for our purposes.

We start by investigating the conditions that may be sufficient for a holomorphic family of operators on $F$ to be the indicial family of an admissible operator. For this, we recall the discussion in the proof of Lemma \ref{lemma:indicial-pseudo}.

\begin{definition}
We say that a holomorphic family $\C_I \ni \lambda\mapsto W(\lambda) \in \Psi^m(F)$ is a \emph{tempered holomorphic family} of pseudo-differential operators if it can be decomposed into 
\[
W(\lambda) = W'(\lambda) + R(\lambda),
\]
where $R(\lambda)= \mathcal{O}_{\Psi^{-\infty}}(\langle\Im \lambda\rangle^{-\infty})$ is a holomorphic family of $\mathcal{O}(\langle\Im \lambda\rangle^{-\infty})$-smoothing operators on $F$, and $W'(\lambda)$ is a holomorphic family of pseudo-differential operators on $F$, supported in a small neighbourhood of the diagonal, such that in local charts, their full symbol satisfy the estimates \eqref{equation:symbol-sigma-lambda}.
\end{definition}

We have already seen that if $A$ is $I$-admissible, then $I_Z(A,\lambda)$ is a tempered holomorphic family of pseudo-differential operators (Lemma \ref{lemma:indicial-pseudo}). We now consider the converse statement

\begin{lemma}
Let $C_I \ni \lambda \mapsto W(\lambda)$ be a tempered holomorphic family of pseudo-differential operators on $F$. Then for $\rho \in I$,
\[
W = \frac{1}{2i\pi}\int_{\Re \lambda = \rho} e^{\lambda r} W(\lambda) d\lambda
\]
defines a convolution operator, pseudo-differential on $\R\times F$. We say that such a $W$ is a $I$-\emph{tempered convolution operator} on $\R \times F$.
\end{lemma}

\begin{proof}
First off, if $W$ is a $\mathcal{O}(\langle \Im \lambda\rangle^{-\infty})$ smoothing operator, that the result holds is elementary. We can thus concentrate on the case that $W(\lambda)$ is uniformly supported in a small neighbourhood of the diagonal. As a consequence, we can take local charts, and assume that we are working in $\R^k$ instead of $F$. 

Thus, we can write the kernel of $W(\lambda)$ as
\[
\frac{1}{(2\pi)^k} \int_{\R^k} e^{i\langle\zeta-\zeta', \eta\rangle} \sigma\left(\zeta,\lambda, \eta\right) d\eta.
\]
The symbol $\sigma$ is holomorphic in the $\lambda$ variable, and we have the estimates
\[
|\partial_\zeta^\alpha \partial_\eta^\beta \partial_\lambda^\gamma\sigma| \leq C_{\alpha,\beta} (1+ |\Im\lambda|^2 + |\eta|^2)^{(m-|\beta|-\gamma)/2}.
\]
 The kernel of $W$ is thus
\[
\begin{split}
\frac{1}{i(2\pi)^{k+1}}\int_{\Re \lambda = \rho} e^{i\langle\zeta-\zeta', \eta\rangle + \lambda(r-r')} &\sigma\left(\zeta,\lambda,\eta\right) d\eta d\lambda =\\
	& \frac{1}{(2\pi)^{k+1}}\int_{\R^{k+1}} e^{i(\langle\zeta-\zeta', \eta\rangle + \lambda(r-r'))} \sigma\left(\zeta,\lambda,\eta\right)e^{\rho(r-r')} d\eta d\lambda.
\end{split}
\]
We can incorporate $\exp( \rho (r-r'))$ in the symbol to find that this is a pseudo-differential operator in the usual sense on $\R\times F$.
\end{proof}

We now introduce the natural extension of admissible pseudo-differential operators:

\begin{definition}
\label{definition:extension}
We define the \emph{extended} $I$-admissible operators of order $m \in \R$ as the set $\overline{\Psi}^{m}_{I}$ of operators $P$ such that there exists a cutoff function $\chi$ (with $\chi \equiv 0$ in $\left\{y < a \right\}$ and $\chi \equiv 1$ in $\left\{y > C_P\right\}$ for some constant $C_P > 0$), a small admissible pseudo-differential operator $A \in \Psi^{m}_{\text{small}}$ and a $I$-tempered convolution operator $W$ of order $m$ on $\R \times F$, and finally a compactly supported pseudo-differential operator $W'$ of order $m$ on $\R\times F$ such that:
\[
P = A + \chi \mathcal{E}(W + W')\mathcal{P} \chi.
\]
\end{definition}

The operator $\chi \mc{E} W' \mc{P} \chi$ is actually a compactly supported pseudo-differential operator on the zeroth Fourier mode. In the Definition \ref{def:R-L^2-admissible-operator} of admissible operators , it was assumed that the operator is pseudo-differential. Extended admissible operators are not pseudo-differential anymore, but they still satisfy properties (1)-(3) of Definition \ref{def:R-L^2-admissible-operator}, with the indicial operator $I(P) = I(A) + W$. They also enjoy the boundedness properties of Lemma \ref{lemma:boundedness-admissible-pseudo}. We can now prove the following:

\begin{lemma}
Assume $A \in \Psi^m_{\mathrm{small}}$ is $I$-admissible left-elliptic, and that we have a tempered holomorphic family $\C_I \ni \lambda \mapsto W(\lambda)$ such that $W(\lambda)I(A,\lambda)= \mathbbm{1}$. Then, there exists $Q \in \overline{\Psi}^{-m}_{I}$ such that $I_Z(Q,\lambda)=W(\lambda)$ for all $\lambda \in \C_I$ and
\[
QA = \mathbbm{1}  \mod \dot{\Psi}^{-\infty}_I.
\]
\end{lemma}

In practice, given $A$, an $I$-admissible left-elliptic operator, we will construct by hand a left indicial inverse $W(\lambda)$ that is \emph{meromorphic on $I$} and without poles in $\left\{ |\Im(\lambda)| \gg 1 \right\}$ (see for instance Lemma \ref{lemma:ellipticite-nabla} where we deal with the gradient of the Sasaki metric). As a consequence, we can define the set
\[
S(A,W) := \left\{ \rho \ \middle|\ W(\lambda) \text{ has a pole on }\rho + i \R  \right\},
\]
that is $\lambda \mapsto W(\lambda)$ is holomorphic in a strip $\left\{ \rho_- < \lambda < \rho_+ \right\}$, where $\rho_\pm \in S(A,W)$ are two consecutive points in this set. The set $S(A,W)$ depends on the choice of left inverse $W$, and the parametrix construction will only work on spaces $y^{\rho-d/2}H^s$, with $\rho \in I\setminus S(A,W)$.\footnote{Although this will not be used in the following, one could go further and define a set $S(A)$ (independent of a choice of $W$) for the left elliptic operator $A$ as:
\[
S(A) := \bigcap_{W \text{ left inverse on } I} S(A,W).
\]
It should be possible to prove that $I \setminus S(A)$ coincides with the set of all $\rho \in I$ such that there exists $\eps > 0$ such that for all $\lambda \in \C$ such that $\Re(\lambda)\in (\rho-\eps,\rho+\eps)$, the operators $I(A,\lambda) \in \Psi^m(F,L_1 \to L_2)$ are injective. As this will not be pursued in the following, we do not make a formal claim out of this.}

\begin{proof}
Let us consider $Q$ a $\R$-admissible $C^\infty$ parametrix for $A$. We want to build a left parametrix for $A$ which acts essentially as $W$ on the zeroth Fourier mode at infinity. For this we start by setting
\[
\tilde{Q} : = Q + \chi \mathcal{E}\left[ W - I(Q) \right]\mathcal{P} \chi.
\]
$\tilde{Q}$ is not pseudo-differential, because we modified $Q$ by an operator which only acts on the zeroth Fourier mode in the cusps, so that $\tilde{Q}$ is an extended admissible operator. If we compute $\tilde{Q}A$, we will obtain 
\[
\tilde{Q}A = \mathbbm{1} + \chi \mathcal{E}\left[ W - I(Q)\right] \mathcal{P} [\chi, A] \mod \dot{\Psi}^{-\infty}_I.
\]

Since $W-I(Q)$ is not smoothing, this is not a parametrix modulo smoothing operators. However, since $[\chi, A]$ is of order $m-1$, and since $A$ is admissible, we can decompose $\chi \mathcal{E}\left[ W - I(Q)\right] \mathcal{P} [\chi, A]$ into the sum of a $I$-residual operator and an operator acting only on the zeroth Fourier mode, which is pseudo-differential as such, and of order $-1$. It further decomposes into the sum of a $I$-residual operator and a compactly supported pseudo-differential operator of order $-1$ (acting only on the zeroth Fourier mode). To improve $\tilde{Q}$ to a parametrix modulo compact \emph{smoothing} remainder, we have to add another modification in the zeroth Fourier mode, via a parametrix construction similar to the original construction of $Q$ in the proof of Proposition \ref{prop:smoothing-remainder}. This will add to $\tilde{Q}$ a compactly supported pseudo-differential operator of order $-m-1$ acting only on the zeroth Fourier mode in the cusps, and we thus obtain as announced an $I$-extended admissible operator.  

\end{proof}

We end this paragraph with the following important comment. In the case where $A$ is only left elliptic, Proposition \ref{prop:smaller-image} can be extended and boils down to saying the following. Assume that $A$ is left elliptic, and invertible on $L^2$. Consider $v \in y^{- \rho_0}L^2$, and $u$ such that $Au=v$, with a priori $u$ in $H^{m}$, then $u$ is actually of the form $u=u_0+u_1$, where $u_0 \in y^{-\rho_0} H^{m}$ and
\[
u_1 = \chi \mc{E}_Z \sum_{\substack{\lambda\text{ left indicial root of }A \\ \Re \lambda\in ]\rho_0, d/2[}} \Pi_\lambda (\mc{P}_Z\chi + G) v,
\]
where $G$ maps into $e^{r\rho}H^{+\infty}$ for all $\rho \in \R$.

\subsubsection{Action on H\"older-Zygmund spaces}

In this section, we will explain how one can prove Theorem \ref{theorem:parametrix-compact} in the case of operators acting on Hölder-Zygmund spaces on cusps.

In the proof of Theorem \ref{theorem:parametrix-compact} in the case of Sobolev spaces, the main ingredients were the existence of the inverse of the indicial operator and the compactness of some injections. Translating the proof to the case of Hölder-Zygmund spaces, we need to check that both ingredient are still available. We start by the compactness argument. As in \S\ref{ssection:black-box}, we consider a smooth cutoff function $\chi \in C^\infty(N,\R)$ such that $\chi|_Z \equiv 1$ for $y > 3a$ and $\chi \equiv 0$ for $y < 2a$.

\begin{lemma}
For any $\rho \in \R, s > s'$, the embedding
\[
\mathbbm{1}-\mathcal{E}_Z \chi \mathcal{P}_Z : y^\rho C^s(N,L) \rightarrow y^\rho C^{s'}(N,L)
\]
is compact. 
\end{lemma}

In other words, the restriction of the injection $y^\rho C^s(N,L) \hookrightarrow y^\rho C^{s'}(N,L)$ to functions with vanishing zeroth Fourier mode is compact. 

\begin{proof}
We follow the proof of Lemma \ref{lemma:compact-injection-non-constant-mode}. As in that proof, it is sufficient to prove that $\|(1-\psi_n)f\|_{C^0_*} \leq C/n \|f\|_{C^{s_0}_*}$ for some $s_0 > 0, C >0$ and then to conclude by interpolation. Since $L^\infty \hookrightarrow C^0_*$ and $C^{1+\epsilon}_* \hookrightarrow C^1$ (for any $\epsilon > 0$), it is therefore sufficient to prove that $\|(1-\psi_n)f\|_{L^\infty} \leq C/n \|f\|_{C^1}$. By Poincaré-Wirtinger's inequality, there exists a constant $C > 0$ (only depending on the lattice $\Lambda$) such that for any $f$ such that $\int f d\theta = 0$, $\|f(y)\|_{L^\infty(\T^d)} \leq C \|\partial_\theta f(y)\|_{L^\infty(\T^d)}$, for all $y > a$. Thus, $\|(1-\psi_n)f(y)\|_{L^\infty(\T^d)} \leq C/n \|y\partial_\theta f(y)\|_{L^\infty(\T^d)}$ and passing to the supremum in $y$, we obtain the result we yearned for.
\end{proof}

Next, we turn to the fact that the indicial operator has a bounded inverse. This is a bit more subtle. For simplicity, assume there are no indicial roots in $\{ \Re \lambda \in I\}\supset i\R$, and consider the action of
\begin{equation}
\label{equation:inverse-holder-zygmund}
I_Z(A)^{-1} =\int_{i\R} e^{\lambda(r-r')} (I_Z(A,\lambda))^{-1} \dd \lambda,
\end{equation}
on $C^s_\ast(\R \times F_Z)$. While the action of convolution operators on $L^2$ spaces is very convenient to analyze, it is not so easy for Hölder-Zygmund spaces. First, from the computations in the proof of Lemma \ref{lemma:relation-projected-spaces}, we deduce that the $C^s_\ast$ spaces of $L\to N$, correspond with the usual $C^s_\ast$ spaces of $L_Z \to\R \times F_Z$.

Next, we recall that according to Lemma \ref{lemma:indicial-pseudo}, $I_Z(A,\lambda)^{-1}$ is a tempered family of holomorphic operators in $\C_I$. In particular, we can decompose it into the sum of a smoothing family, and a pseudo-differential family supported in a small neighbourhood of the diagonal. We denote by $S_I^{(1)}$ the contribution of the pseudo-differential family, and $S_I^{(2)}$ that of the smoothing family.

Choosing local patches in $F_Z$, we can write
\[
S_I^{(1)} f(r,\zeta) = \int_{\R \times \R^n} e^{i\lambda(r-r')} e^{i \langle \zeta-\zeta', \eta \rangle} \widetilde{\sigma}(\lambda,z,\xi) f(r',\zeta') \dd r' \dd \zeta' \dd\lambda \dd\eta,
\]
and this is a classical pseudo-differential operator of order $-m$ on $R \times F_Z$ which is bounded as a map $C^s_*(\R \times F_Z) \rightarrow C^{s+m}_*(\R \times F_Z)$. 

It remains to study $S_I^{(2)}$. For the sake of simplicity, we will identify in our notations the operator and its kernel. We pick $z,z'\in F_Z$ and $r>1$. When $|\rho|<\epsilon$,
\[
S_I^{(2)}(r,z,z') =  \int_\R  e^{ it r } R_{it} (z,z') dt = e^{\rho r} \int_{\R} e^{itr} R_{it+\rho} (z,z') \dd t,
\]
where $R_{it+\rho}$ is $\mathcal{O}( \langle t \rangle^{-\infty})$ in $C^\infty(F_Z\times F_Z)$, for $|\rho|<\epsilon$. We deduce that $S_I^{(2)}(r,z,z')$ is $\mathcal{O}( e^{-\epsilon|r|})$ in $C^\infty(\R\times F_Z\times F_Z)$. In particular, $S_I^{(2)}$ acts boundedly as a map $C_*^s(\R \times F_Z) \rightarrow C_*^{s+m}(\R \times F_Z)$.
Now that we have checked that $I_Z(A)^{-1}$ is bounded on the appropriate spaces, the proof of Section \S\ref{sec:improving-I} applies. 

At this point, we also observe that our arguments can be combined directly with the arguments of section \ref{sec:leftright-parametrix} to deal with the case of left/right elliptic operators on H\"older-Zygmund spaces.

To finish the proof of Theorem \ref{theorem:parametrix-compact}, we consider the Fredholm index of elliptic operators acting on Hölder-Zygmund spaces. It is similar to Proposition \ref{proposition:fredholm-index}.

\begin{proposition}
\label{proposition:fredholm-index-ii}
Let $P$ be a $(\rho_-,\rho_+)$-admissible elliptic pseudo-differential operator of order $m \in \R$. Let $I$ be a connected component in $(\rho_-,\rho_+)$ not containing any indicial root. Then, the Fredholm index of the bounded operator $P : y^{\rho} C_*^{s+m} \to y^{\rho} C_*^s$ is independent of $s \in \R, \rho \in I$. Moreover, the Fredholm index coincides with that of Proposition \ref{proposition:fredholm-index}, that is of $P$ acting on Sobolev spaces $H^{s+m,\rho-d/2,\rho_\bot} \rightarrow H^{s, \rho-d/2,\rho_\bot}$, for $s, \rho_\bot \in \R$.
\end{proposition}

\begin{proof}
This is a rather straightforward consequence of Proposition \ref{proposition:fredholm-index} combined with the embedding estimates of Lemma \ref{lemma:embedding-holder-to-sobolev-spaces} and Lemma \ref{lemma:embedding-hs-in-ck}.
\end{proof}

\section{X-ray transform and symmetric tensors}

In this Section, we apply the previous theory of inversion of elliptic pseudo-differential operators to the three operators $\nabla_S, D$ and $D^*D$ and prove that the X-ray transform is solenoidal injective on $2$-tensors.

\label{section:geometry}

\subsection{Gradient of the Sasaki metric}

A first step towards the Livsic Theorem \ref{theorem:livsic} is the analytic study of the gradient $\nabla_S$ induced by the Sasaki metric $g_S$ (itself induced by $g$) on the unit tangent bundle $SM$ of $(M,g)$. Let $\pi : SM \rightarrow M$ be the projection. We recall (see \cite{Paternain-99} for further details) that the tangent bundle to $SM$ can be decomposed according to the splitting
\[ T(SM) = \mathbb{V} \oplus^\bot \mathbb{H} \oplus^\bot \R X, \]
where $\mathbb{V} = \ker d \pi$ is the \emph{vertical bundle}, $\mathbb{H} = \ker \mc{K} \cap (\R X)^\bot$ is the \emph{horizontal bundle}\footnote{We use the convention that $\HH := (\V \oplus \R X)^\bot$, and not $\V^\bot$ as usual. In particular, if $M$ is $(d+1)$-dimensional, then $\HH$ is $d$-dimensional.}, $\mc{K} : T(SM) \rightarrow TM$ is the \emph{connection map} defined as follows: consider $(x,v) \in SM, w \in T_{(x,v)}(SM)$ and a curve $(-\eps,\eps) \ni t \mapsto z(t) \in SM$ such that $z(0)=(x,v), \dot{z}(0)=w$; write $z(t)=(x(t),v(t))$; then $\mc{K}_{(x,v)}(w) := \nabla_{\dot{x}(t)} v(t)|_{t=0}$. Note that $d \pi : \V \rightarrow TM, \mc{K} : \HH \oplus \R X \rightarrow TM$ are both isometries. We denote by $g_S$ the Sasaki metric on $SM$ defined by:
\[
g_S(w,w') := g(d \pi(w), d\pi(w')) + g(\mc{K}(w),\mc{K}(w')).
\]
Let $\nabla_S$ be the Levi-Civita connection induced by the Sasaki metric $g_S$ on $SM$. Given $u \in C^\infty(SM)$, one can decompose its gradient according to:
\begin{equation}
\nabla_S u = \nabla^v u + \nabla^h u + Xu \cdot X, 
\end{equation}
where $\nabla^{v,h}$ are the respective vertical and horizontal gradients (the orthogonal projection of the gradient on the vertical and horizontal bundles), that is $\nabla^v u \in C^\infty(SM,\mathbb{V})$ and $\nabla^h u \in C^\infty(SM,\mathbb{H})$. 

\begin{lemma}
\label{lemma:ellipticite-nabla}
The gradient $\nabla_S : C^\infty(SM) \rightarrow C^\infty(SM, T(SM))$ is a left-elliptic $\R$-admissible differential operator of order $1$. Its only indicial root is $0$. Moreover, there exists a $]0,+\infty[$-extended admissible operator $Q$ of order $-1$ and $R$ a $]0;+\infty[$-residual operator such that:
\[
Q\nabla_S = \mathbbm{1} + R 
\]
\end{lemma}

\begin{proof}
As before, we let $Z := [a,+\infty[ \times \R^d/\Lambda$ be (one of) the cuspidal part of the manifold $(M,g)$. The fact that $\nabla_S$ is an elliptic admissible differential operator of order $1$ is immediate. Observe that, if $f \in C^\infty(SM)$, then $f$ can be extended as a function $\tilde{f}$ in $C^\infty(TM \setminus \left\{ 0 \right\})$ by $0$-homogeneity. Then, writing $\nabla_S^{TM}$ for the gradient on $TM$ (induced by the Sasaki metric), we have $\nabla_S f = (\nabla_S^{TM} \tilde{f})|_{SM}$. As a consequence, it is equivalent to study the action of $\nabla_S^{TM} : C^\infty(TM) \rightarrow C^\infty(T(TM))$ on $0$-homogeneous functions, which we are going to do from now on. Taking advantage of the global trivialization over the cusp, we can write $TZ \simeq Z \times \R^{d+1}$; we use coordinates $(y,\theta, v_y, v_\theta)$.

We now compute the indicial operator of the gradient, following the techniques of the previous sections. In order to do so, we let $f \in C^\infty(\R^{d+1})$ be a function depending on the variables $(v_y,v_\theta)$ which is \emph{$0$-homogeneous}. In other words, $f$ is a smooth function on the quotient $F_Z := \R^{d+1}/\R_+^*$ which is a smooth \emph{compact} manifold diffeomorphic to the $d$-dimensional sphere. Therefore, we are really in the setting of the previous two sections. We introduce the vector fields $U := d \pi^{-1}(y\partial_y), V_\ell := d\pi^{-1}(y\partial_{\theta_\ell})$ for $\ell=1,...,d$. They belong to the space $\HH \oplus \R X$. An elementary (although tedious) computation using Christoffel symbols in coordinates (see \cite[Appendix A.3.2]{Bonthonneau-thesis}) allows to show that: 
\[
U \propto y \partial_y + v_y \partial_{v_y} + \sum_\ell v_\theta \partial_{v_\theta} = y\partial_y + \mc{E},
\]
(modulo normalization), where $\mc{E}$ is the Euler vector field and
\[
V_\ell \propto y \partial_{\theta_\ell} + v_y \partial_{v_{\theta_\ell}} - v_{\theta_\ell} \partial_{v_y},
\]
(modulo normalization). Then, writing $\nabla_{TM}^v$ for the (total) vertical gradient on $TM$, and using the formula $d\pi^{-1}(Y) \cdot \pi^*h = \pi^* (Y \cdot h)$ for $Y \in C^\infty(M,TM), h \in C^\infty(M)$, we obtain:
\[
\begin{split}
y^{-\lambda}\nabla^{TM}_S(f y^\lambda) & = \nabla_{TM}^v f + y^{-\lambda} \left( U(fy^\lambda) \cdot U + \sum_\ell V_\ell(f y^\lambda) \cdot V_\ell \right) \\
& =  \nabla_{TM}^v f + (Uf + y^{-\lambda} f y \partial_y(y^\lambda)) \cdot U + \sum_\ell (V_\ell f + y^{-\lambda} f y \partial_{\theta_\ell}(y^\lambda) ) \cdot V_\ell \\
& = \nabla_{TM}^v f + \lambda f \cdot U + \sum_\ell V_\ell f \cdot V_\ell,
\end{split}
\]
because $Uf = 0$ (note that the $0$-homogeneity condition is used here in $\mc{E} f = 0$). We then set $W(\lambda)(w) := \lambda^{-1} g_S(w,U)$, for $w \in C^\infty(\R \times \R^{d+1}, T(\R \times \R^{d+1}))$. Then:
\[
W(\lambda) I(\nabla_S,\lambda) f = f
\]
The only indicial root of $\nabla_S$ is thus $\lambda=0$. As a consequence, Theorem \ref{theorem:left-parametrix} applies immediately (using construction (3), presented after Theorem \ref{theorem:left-parametrix}) and yields the results of Lemma \ref{lemma:ellipticite-nabla}.
\end{proof}

\subsection{Exact Livsic theorem}

We recall that $\mathcal{C}$ is the set of hyperbolic free homotopy classes on $M$ and that for each such class $c\in \mathcal{C}$ of $C^1$ curves on $M$, there is a unique representant $\gamma_g(c)$ which is a \emph{geodesic} for $g$. 

In this section, we prove an exact \textit{Livsic theorem} asserting that a function whose integrals over closed geodesic vanish is a \emph{coboundary}, namely a derivative in the flow direction. For $f\in C^0(SM)$, we can define
\[
I^gf(c) = \frac{1}{\ell(\gamma_g(c))} \int_{0}^{\ell(\gamma_g(c))} f(\gamma(t),\dot{\gamma}(t)) \dd t,
\]
for $c \in \mc{C}$.

\begin{theorem}[Livsic theorem]
\label{theorem:livsic}
Let $(M^{d+1},g)$ be a negatively-curved complete manifold whose ends are real hyperbolic cusps. Denote by $-\kappa_0$ the maximum of the sectional curvature. Let $0 < \alpha < 1$ and $0 < \beta < \sqrt{\kappa_0}\alpha$. Let $f \in y^\beta C^\alpha(SM) \cap H^1(SM)$ such that $I^g f = 0$. Then there exists $u\in y^\beta C^\alpha(SM) \cap H^1(SM)$ such that $f=Xu$. Moreover, $\nabla^v X u, \nabla_X \nabla^v u \in L^2(SM, \V)$ and $u$ thus satisfies the Pestov identity (Lemma \ref{lemma:Pestov}).
\end{theorem}

We will denote by $\N_\bot$ the subbundle of $\pi^* TM \rightarrow SM$ (where $\pi : SM \rightarrow M$ denotes the projection) whose fiber at $(x,v) \in SM$ is given by $\N_\bot(x,v) := \left\{v\right\}^{\bot}$. The picture to have in mind is the following: above each point $(x,v) \in SM$, we glue the fiber $(\R \cdot v)^\bot$. Using the maps $d\pi$ and $\mc{K}$, the vectors $\nabla^{v,h} u$ can be identified with elements of $\N_\bot$, i.e. $\mc{K}(\nabla^v u), d\pi(\nabla^h u) \in \N_\bot$. For the sake of simplicity, we will drop the notation of these projection maps in the following and consider $\nabla^{v,h} u$ as elements of $\N_\bot$. We have a natural $L^2$-scalar product on $L^2(SM,\N_\bot)$ given by:
\[
\langle w, w' \rangle := \int_{SM} g_x(w(x,v),w'(x,v)) \dd \mu(x,v),
\]
where $\mu$ stands for the Liouville measure. We refer to \cite[Section 2]{Paternain-Salo-Uhlmann-15} for further details. Before starting with the proof of the Livsic Theorem \ref{theorem:livsic}, we recall the celebrated \emph{Pestov identity}:

\begin{lemma}[Pestov identity]
\label{lemma:Pestov}
Let $(M^{d+1},g)$ be a cusp manifold. Let $u\in H^2(SM)$. Then
\[
\begin{split}
\|\nabla^v X u \|_{L^2(SM,\N_\bot)}^2 \hspace{-35pt}& \\
&=\|\nabla_X \nabla^v u  \|_{L^2(SM,\N_\bot)}^2  	-\int_{SM} \kappa(v, \nabla^vu) \|\nabla^v u\|^2 \dd \mu(x,v)  + d  \| X u \|_{L^2(SM)}^2,
\end{split}
\]
where $\kappa$ is the sectional curvature.
\end{lemma}

In the compact case, the proof is based on the integration of local commutator formulas and clever integration by parts (see \cite[Proposition 2.2]{Paternain-Salo-Uhlmann-15}). Since the manifold has finite volume and no boundary, the proof is identical and we do not reproduce it here. By a density argument and using the fact that the sectional curvature is pinched negative, assuming only $\nabla^v X u \in L^2(SM)$, we deduce that $\nabla_X \nabla^v u, \nabla^v u \in L^2(SM)$ and
\[
\|\nabla_X \nabla^v u  \|,\ \|\nabla^v u\| \lesssim \|\nabla^v X u \|.
\]

\begin{proof}[Proof of Theorem \ref{theorem:livsic}]
In this proof, we will first build $u$, and then determine its exact regularity. For the construction, we follow the usual tactics, but we give the details since we want to let the H\"older constant grow at infinity. For the sake of simplicity, we will denote by $y : M \rightarrow \R_+$ a smooth extension of the height function (initially defined in the cusps) to the whole unit tangent of the manifold, such that $0 < c < y$ is uniformly bounded from below and $y \leq a$ on $M \setminus \cup_\ell Z_\ell$. The case of uniformly H\"older functions was dealt with in \cite[Remark 3.1]{Paulin-Pollicott-Schapira-15}. Since the flow is transitive, we pick a point with dense orbit $x_0$, and define
\[
u(\varphi_t(x_0)) = \int_0^t f(\varphi_s(x_0))\dd s.
\]
Obviously, we have $X u = f$, so it remains to prove that it is locally uniformly H\"older to consider the extension of $u$ to $SM$. Pick $x_1 =\varphi_t(x_0)$ and $x_2 = \varphi_{t'}(x_0)$, with $t'>t$. Pick $\epsilon>0$, and assume that $d(x_1,x_2)=\epsilon$. By the Shadowing Lemma, there is a periodic point $x'$ with $d(x_1,x')<\epsilon$ and period $T<|t'-t| + C\epsilon$, for some uniform constant $C > 0$ depending on the dynamics, which shadows the segment $(\varphi_s(x_0))_{s \in [t,t']}$. Moreover, there exists a time $\tau \leq C\epsilon$ such that we have the following estimate:
\begin{equation}
\label{equation:hyperbolicite}
d(\varphi_s(\varphi_\tau(x_1)), \varphi_s(x')) \leq C\epsilon e^{-\sqrt{\kappa_0} \min(s, |t'-t|-s)}
\end{equation}
This is a classical bound in hyperbolic dynamics (see \cite[Proposition 6.2.4]{Fisher-Hasselblatt} for instance). The constant $\sqrt{\kappa_0}$ follows from the fact the maximum of the curvature is related to the lowest expansion rate of the flow (see \cite[Theorem 3.9.1]{Klingenberg-95} for instance). 

Then, using the assumption that $\int_0^T f(\varphi_s(x')) ds = 0$, we write:
\[
\begin{split}
u(x_2)&-u(x_1)  = \int_0^{t'-t} f(\varphi_s(x_1)) \dd s \\
&\hspace{-10pt} = \int_0^{t'-t-\tau} f(\varphi_s(\varphi_\tau(x_1))) - f(\varphi_s(x')) \dd s - \int_{t'-t-\tau}^{T} f(\varphi_s(x')) \dd s + \int_0^{\tau} f(\varphi_s(x_1)) \dd s
\end{split}
\]
The last two terms are immediately bounded by $\lesssim \epsilon y(x_1)^\beta$ because $\tau < C \epsilon$ and $T-(t-t') < C \eps$. As to the first one, using \eqref{equation:hyperbolicite}, it is controlled by
\[
\left| \int_0^{t'-t-\tau} f(\varphi_s(\varphi_\tau(x_1))) - f(\varphi_s(x')) \dd s\right| \lesssim \int_0^{t'-t} y(\varphi_s(x'))^\beta d(\varphi_s(x_1), \varphi_s(x'))^\alpha \dd s,
\]
thanks to the assumption on $f$ (namely it is $y^\beta C^\alpha(SM)$). (Note that changing $y(\varphi_s(x'))^\beta$ by $y(\varphi_s(x_1))^\beta$ or $y(\varphi_s(\varphi_\tau(x_1)))^\beta$ does not change anything in the previous integral on the right-hand side as the ratios
\[
y(\varphi_s(x'))/y(\varphi_s(x_1)), y(\varphi_s(x'))/y(\varphi_s(\varphi_\tau(x_1)))
\]
are uniformly contained in an interval $[1/C,C]$ for some constant $C > 0$). Let us find an upper bound on $y(\varphi_s(x'))$. Of course, when a segment of the trajectory $(\varphi_s(x'))_{s \in [0,T]}$ is included in a compact part of the manifold (say of height $y \leq a$), $y(\varphi_s(x'))$ is uniformly bounded by $a$, so the only interesting part is when the trajectory is contained in the cusps. In time $|t'-t|$, the segment $(\varphi_s(x'))_{s \in [0,T]}$ has started and returned at height $y(x_1)$. Thus, it can only go up to a height
\begin{equation}
\label{equation:hauteur-borne}
y(\varphi_s(x')) \leq e^{\min(s,|t'-t|-s)} y(x_1).
\end{equation}

Combining (\ref{equation:hyperbolicite}) and (\ref{equation:hauteur-borne}), this leads to:
\[
\begin{split}
\int_0^{t'-t} y(\varphi_s(x'))^\beta & d(\varphi_s(x_1), \varphi_s(x'))^\alpha \\
& \lesssim \int_0^{t'-t} y(x_1)^\beta e^{\beta \min(s,|t'-t|-s)} d(x_1,x_2)^\alpha e^{-\alpha \sqrt{\kappa_0} \min(s,(t'-t)-s)}  \dd s \\
& \lesssim y(x_1)^\beta d(x_1,x_2)^\alpha \int_0^{t'-t} e^{(\beta -\alpha \sqrt{\kappa_0}) \min(s,|t'-t|-s)} \dd s
\end{split}
\]
As long as $\sqrt{\kappa_0} \alpha>\beta$, this is uniformly bounded as $|t'-t|\to +\infty$. In particular, we conclude that $u$ is $y^\beta C^\alpha$, and we can thus extend it to a global $y^\beta C^\alpha$ function on $SM$.

We now have to prove that $u \in H^1(SM)$ and to this end, we will use a kind of bootstrap argument. Since $f \in H^1(SM)$ and $f = Xu$, we obtain that $\nabla^vXu \in L^2(SM)$. Moreover, as discussed after the Pestov identity, we obtain directly that $\nabla_X \nabla^vu, \nabla^v u \in L^2(SM)$.

By using the commutator identity $[X,\nabla^v] = - \nabla^h$ (see \cite[Lemma 2.1]{Paternain-Salo-Uhlmann-15}), we deduce $\nabla^h u \in L^2(SM)$. Thus, $\nabla_S u \in L^2$. By Lemma \ref{lemma:ellipticite-nabla}, we deduce that $u \in H^{1}(SM)$
\end{proof}

\subsection{X-ray transform and symmetric tensors}

\label{ssection:xray}

Although we will mostly use $1$- and $2$-tensors, it is convenient to introduce notations for general symmetric tensors. We will be using the injection
\[
\pi_m : v \in C^\infty(M, SM) \to v \otimes \dots \otimes v \in C^\infty(M, SM^{\otimes m}).
\]
Given a symmetric $m$-tensor $h\in C^\infty(M,S^m(T^\ast M))$, we can define a function on $SM$ by pulling it back via $\pi_m$:
\[
\pi_m^\ast h: (x,v) \mapsto h_x ( v \otimes \dots \otimes v).
\]

\begin{definition}
\label{definition:xray}
The X-ray transform on symmetric $m$-tensors is defined in the same way as for $C^0$ functions on $SM$: if $h$ is a symmetric $m$-tensor, 
\[
I_m^g h(c) = \frac{1}{\ell(\gamma_g(c))} \int_{0}^{\ell(\gamma_g(c))} \pi_m^\ast h(\gamma(t),\dot{\gamma}(t)) dt,
\]
where $t \mapsto \gamma(t)$ is a parametrization by arc-length, $c \in \mc{C}$.
\end{definition}

Given a symmetric $m$-tensor $h$, we can consider its covariant derivative $\nabla h$, which is a section of 
\[
T^\ast M \otimes S^m( T^\ast M)\to M.
\]
If $\mc{S}$ denotes the symmetrization operator from $\otimes^{m+1} T^\ast M$ to $S^{m+1}(T^\ast M)$, we define the \emph{symmetric derivative} as
\[
D h = \mc{S} (\nabla h) \in C^\infty( M, S^{m+1}( T^\ast M)).
\]
Given $x \in M$, the pointwise scalar product for tensors in $\otimes^m T^\ast_xM$ is defined by
\[
\langle v_1^* \otimes ... \otimes v_m^*, w_1^* \otimes ... \otimes w_m^* \rangle_x = \prod_{j=1}^m g(v_j,w_j),
\]
where $v_j, w_j \in T_xM$ and $v_j^*,w_j^*$ denotes the dual vector given by the musical isomorphism. We can then endow the spaces $C^\infty(M, S^m (T^\ast M))$ with the scalar product 
\begin{equation}
\label{equation:metric-tensors}
\langle h_1, h_2 \rangle =  \int_{M} \langle h_1(x), h_2(x) \rangle_x d\vol(x)
\end{equation}
We obtain a global scalar product on $\oplus_{m \geq 0} C^\infty(M, S^m(T^\ast M))$ by declaring that whenever $m\neq m'$, $C^\infty(M, S^m (T^\ast M))$ is orthogonal to $C^\infty(M, S^{m'} (T^\ast M))$. Following conventions we denote by $-D^\ast$ the adjoint of $D$ with respect to this scalar product. One can compute that for a tensor $T$, for any orthogonal frame $e_1,\dots,e_{d+1}$,
\[
D^\ast T(\cdot) = \Tr (\nabla T)(\cdot) = \sum_i  \nabla_{e_i} T(e_i, \cdot).
\]
The operator $D^\ast$ is called the \emph{divergence}, and one can check that it maps symmetric tensors to symmetric tensors.
\begin{definition}\label{def:solenoidal}
Let $f$ be a tensor so that $D^\ast f =0$. Then we say that $f$ is \emph{solenoidal}.
\end{definition}
We can also define $\pi_{m\ast}$, which is the formal adjoint of $\pi_m^\ast$ --- with respect to the usual scalar product on $L^2(SM)$. Moreover, one can check\footnote{The fastest way to check this is the following: first of all, observe that $\pi^*_{m+1} D = \pi^*_{m+1} \nabla$ (all the antisymmetric parts vanish); then consider at a point $p \in M$ local normal coordinates $(x_1,...,x_{d+1})$; in these coordinates at $p$, $\nabla f (p)= \sum_i \partial_{x_i} f d x_i$ and $X|_{T_pM}=\sum_i v_i \partial_{x_i}$ which is now sufficient the claim.}
\[
\pi^\ast_{m+1} D = X\pi^\ast_m,
\]
see \cite[Lemma B.1.4]{Lefeuvre-thesis} for instance. Through $\pi_m^\ast$ we obtain another scalar product on symmetric tensors:
\[
[ u,v] = \int_{SM} \pi_m^\ast u \overline{\pi_m^\ast v}.
\]
Representing $[u,v] = \langle A u, v\rangle$, one can check that there are universal constants $C_m>0$ such that $\| A\| \leq C_m$, $\| A^{-1}\|\leq C_m$ when restricted to $m$-tensors. (This simply follows from the fact that this is a statement pointwise in $x \in M$; as a consequence, the constant $C_m > 0$ is universal and does not depend on the geometry.)  In the following, we will restrict our study to the map
\[
D : C^\infty(M,T^*M) \rightarrow C^\infty(M,S^2(T^*M))
\]
but it is very likely that most of the results still hold for tensors of general order $m \in \N$. As in the compact case, we obtain:

\begin{lemma}\label{lemma:D-elliptic}
The symmetric derivative of $1$-forms
\[
D : C^\infty(M,T^*M) \rightarrow C^\infty(M,S^2(T^*M))
\]
is $\R$-admissible and left uniformly elliptic. Its only indicial root is $-1$. Additionally, it is injective on $y^{\rho}H^s$ and $y^\rho C^s_\ast$ for all $\rho,s \in \R$. Moreover, there is $]-1,+\infty[$-extended admissible pseudo-differential operators $Q$ of order $-1$ and $R$, a $]-1,+\infty[$-residual operator such that:
\[
QD = \mathbbm{1} + R.
\]
In particular, the image of $D(H^{s+1}(M,T^*M)) \subset H^s(M,S^2(T^*M))$ is closed, for all $s \in \R$.
\end{lemma}

\begin{proof}
For the moment, we deal with the general case
\[
D : C^\infty(M,S^m(T^*M)) \rightarrow C^\infty(M,S^{m+1}(T^*M)).
\]
Since $D$ is a differential operator, it makes no difference to work with Sobolev or Hölder-Zygmund spaces. The first step is to prove that $D$ is uniformly elliptic. By taking local coordinates around a point $(x,\xi) \in T^*M \setminus \left\{ 0 \right\}$ for instance, one can compute the principal symbol of the operator $D$ which is $\sigma(D)(x,\xi) : u \mapsto \mc{S}(\xi \otimes u)$, where $u \in S^m(T^*_xM)$  (see \cite[Theorem 3.3.2]{Sharafutdinov-94}). Then, using the fact that the antisymmetric part of $\xi \otimes u$ vanishes in the integral:
\[ 
\| \sigma(D)u \|^2 \geq C_m^{-1} \int_{\Ss^d} \langle\xi,v\rangle^2 {\pi_m^\ast u}^2(v) \dd v = C_m^{-1}|\xi|^2 \int_{\Ss^d} \langle\xi/|\xi|,v\rangle^2 {\pi_m^\ast u}^2(v) \dd v > 0, 
\]
unless $u \equiv 0$. Since $S^m(T^\ast_x M)$ is finite dimensional, the map
\[ 
(u,\xi/|\xi|) \mapsto \| \sigma(D)(x,\xi/|\xi|)u \|,
\]
defined on the compact set $\left\{u\in S^m, |u|^2=1\right\} \times \Ss^d$ is bounded and attains its lower bound $C^2 > 0$ (which is independent of $x$). Thus $\|\sigma(x,\xi)u\| \geq C|\xi| \| u\|$. It is then not hard to check from this lower bound that the operator is uniformly left elliptic in the sense of Definition \ref{def:elliptic} (actually, this could have been taken as an equivalent definition of left ellipticity).

We now assume $m=1$ and consider
\[
D : C^\infty(M,T^*M) \rightarrow C^\infty(M,S^2(T^*M)).
\]
Let us give a word on the injectivity of this operator. Consider a $1$-form $f$ such that $Df = 0$, and $f$ is either in some $y^\rho H^s$ or some $y^\rho C^s_\ast$. Then $f$ is smooth by the elliptic regularity Theorem. As a consequence $\pi_1^\ast f$ is a smooth function on $SM$. Recall that $X \pi_1^*f  = \pi_2^* Df = 0$. Additionally, the geodesic flow admits a dense orbit; we deduce that $\pi_1^\ast f$ is a constant. However, since $f$ is a $1$-form, $\pi_1^*f(x,-v)=-\pi_1^*f(x,v)$ for all $(x,v) \in SM$, thus $f = 0$.

Now, we recall the results from Sections \S\ref{sec:parametrix-modulo-compact} and \S\ref{section:holder-zygmund}. Since $D$ is a differential operator that is invariant under local isometries, it is a $\R$-admissible (left elliptic) operator. In particular, it suffices to determine whether its associated indicial family $I_Z(D,\lambda)$ has a left inverse. In the present case, since $D$ is an operator on sections of a bundle over $M$, the indicial operator is just a matrix. We consider a $1$-form $\alpha$ in the cusp in the form
\[
y^\lambda \left[ a \frac{dy}{y} + \sum b_i\frac{d\theta_i}{y},  \right]
\]
where $a, b_i \in \R$. Then we find that
\[
D \alpha = y^\lambda\left[ a\left( \lambda \frac{dy^2}{y^2} - \sum \frac{d\theta_i^2}{y^2}\right) + \sum b_i (\lambda+1)\frac{ d\theta_i dy + dy d\theta_i}{y^2}\right].
\]
The matrix $I_Z(D,\lambda)$ is thus the transpose of
\[
\begin{pmatrix}
\lambda 	& - 1 	& -1 	&\dots 	& - 1 	& 0 		& 0 		& \dots 	& 0 \\
0 		& 0 	& 0 	&0  	&0 		& 2(\lambda+1) &0   	 	& \dots 	& 0 \\
0 		& 0 	& 0 	&0  	&0 		& 0			&2(\lambda+1) 	& \dots 	& 0 \\
\dots 	& \dots	& \dots	&\dots 	&\dots	& \dots		&\dots 		& \dots 	& \dots \\
0 		& 0 	& 0 	&0  	&0 		& 0			&0 			& \dots 	& 2(\lambda + 1)
\end{pmatrix}
\]
In particular, with 
\[
J(\lambda) = \begin{pmatrix}
(\lambda+1)^{-1} 	& -(\lambda+1)^{-1}	& 0 	&\dots  & 0 				& \dots 	& 0 \\
0 					& 0 				& 0 	&0  	& (2(\lambda+1))^{-1}& \dots 	& 0 \\
\dots 				& \dots				& \dots	&\dots 	& \dots	 			& \dots 	& \dots \\
0 					& 0 				& 0 	&0   	& 0		 			& \dots 	& (2(\lambda + 1))^{-1}
\end{pmatrix}
\]
we get
\[
J(\lambda)I_Z(D,\lambda) = \mathbbm{1};\quad \|J(\lambda)\| = \mathcal{O}(|\lambda|^{-1}) \text{ as }\Im \lambda \to \pm \infty.
\]
We deduce that $D$ has $-1$ for sole indicial root. As a consequence, we can apply Theorem \ref{theorem:left-parametrix} (using again construction (3)):
\begin{equation}\label{eq:parametrix-D}
Q D = \mathbbm{1} + R,
\end{equation}
with $R$ bounded from $H^{s,\rho}$ to $H^{N, -d/2 -1 + \epsilon}$ and from $C^{s,\rho}_\ast$ to $C^{s, -1 + \epsilon}_\ast$, for all $d/2 >\epsilon>0$, $s\in \R$, $\rho > -d/2 + 1$.

Let us now prove that the image of $D$ is closed (for the Sobolev spaces, the case of Hölder-Zygmund spaces is similar). This is rather classical argument once one has an inverse for the operator modulo a compact remainder, but we reproduce it here for the reader's convenience. For a sequence $(u_n)$ of elements of $H^{1+s}$ such that $Du_n \to f \in H^{s}$, $QD u_n = u_n + R u_n$ also converges since $Q$ is continuous. By extraction, since $R$ is compact we can assume that $R(u_n/\|u_n\|)$ converges also, to some $v$. Then, we have
\[
u_n + \|u_n\|(v+ o(1)) =Q f + o(1).
\]
Assume that $\|u_n\|$ is bounded. Then we obtain that $u_n$ itself converges in $H^{1+s}$, to some $u$, and $Du = f$. Otherwise, we can decompose $u_n = \lambda_n v + w_n$, with $w_n\perp v$, $w_n$ bounded and $\lambda_n \to \infty$. We deduce that $R v = -v$, and $QD u_n = QD w_n$, so that we can extract $w_n$ to make it converge to some $w$, and $Dw=f$.
\end{proof}

Since the image of $D$ is closed, it is the orthogonal of the kernel of $D^\ast$, and each $f\in H^s(M, S^2(T^\ast M))$ can be written as
\[
f = f^s + Du,
\]
with $D^\ast f^s = 0$, and $f^s \in H^s(M, S^2(T^\ast M))$, $u\in H^{1+s}(M, S^1(T^\ast M))$. The tensor $f^s$ is called the \emph{divergence-free} part or the \emph{solenoidal part} of $f$, and $Du$ the \emph{exact part} or the \emph{potential part} of $f$. This can be naturally generalized to tensors of any order and Hölder-Zygmund spaces, following the same scheme of proof.

To close this section, remark that the X-ray transform satisfies $IX = 0$ and thus $0 = I X \pi_m^\ast = I \pi_m^\ast D = I_m D$. Thus in general it is impossible to recover the exact part $Dp$ of a tensor $f$ from the knowledge of $I_m f$. We will say that the X-ray is \emph{solenoidal injective} on smooth symmetric $m$-tensors if it is injective when restricted to $\ker D^*$.

\subsection{Projection on solenoidal tensors}

In this section, we will study the symmetric Laplacian on $1$-forms, that is the operator $\Delta := D^*D$ acting on sections of $S^1(T^\ast M) \rightarrow M$. We will denote by $\lambda_d^\pm = d/2 \pm\sqrt{d+d^2/4}$. Observe that $\lambda_d^- < 0$ (this will be used later).

\begin{lemma}\label{lemma:laplacian-symmetric}
For all $s \in \R,\rho\in ]\lambda^-_d,\lambda^+_d[, \rho_\bot \in \R$, the operator $\Delta$ is invertible on the spaces $H^{s,\rho - d/2,\rho_\bot}(M,S^1(T^\ast M))$ and on $y^\rho C^s_\ast(M,S^1(T^\ast M))$. Its inverse $\Delta^{-1}$ is a pseudo-differential operator of order $-2$.
\end{lemma}

\begin{proof}
The operator $\Delta = D^*D$ is elliptic since $D$ is elliptic, and it is also invariant under local isometries, and differential. In particular, it is $\R$-admissible, so we can apply Theorem \ref{theorem:parametrix-compact}. Let us compute its indicial operator: we find
\[ 
\begin{array}{l} 
I(\Delta,\lambda)\left(a \dfrac{dy}{y}\right) = (\lambda^2-\lambda d -d) a \dfrac{dy}{y} \\
I(\Delta,\lambda)\left(b_i \dfrac{d\theta_i}{y}\right) = \dfrac{1}{2}(\lambda+1)(\lambda-(d+1)) b_i \dfrac{d\theta_i}{y} \end{array} 
\]
$I(\Delta,\lambda)$ is a diagonal matrix which is invertible for
\[ 
\lambda \notin \left\{-1,d+1,\underbrace{d/2\pm \sqrt{d+d^2/4}}_{=\lambda^\pm_d}\right\} 
\]
The interval $]\lambda^-_d,\lambda^+_d[$ does not contain other any root, so we can apply directly Theorem \ref{theorem:parametrix-compact}, and get a pseudo-differential operator $Q$ of order $-2$, bounded on the relevant Sobolev and Hölder-Zygmund spaces such that
\begin{equation}
\label{equation:parametrix-laplacien-1}
Q\Delta = \mathbbm{1} + K,
\end{equation}
with $K$ bounded from $y^{\rho} H^{-N}$ to $y^{-\rho}H^{N}$, $y^{\rho+d/2}C^{-N}_\ast$ to $y^{d/2-\rho}C^{N}_\ast$ for all $\rho\in [0,\lambda^+_d-d/2[$. We can also do this on the other side:
\begin{equation}
\label{equation:parametrix-laplacien-2}
\Delta Q = \mathbbm{1} + K',
\end{equation}
$K'$ satisfying the same bounds. We deduce that $\Delta$ is Fredholm. Additionally, from the parametrix equation, we find that any element of its kernel (on any Sobolev or Hölder-Zygmund space we are considering) has to lie in $L^2(SM)$. However, on $L^2$, $\Delta u = 0$ implies $Du =0$, and $u=0$. Additionally, on $L^2$, $\Delta$ is self-adjoint, so it is invertible and its Fredholm index is $0$. We then conclude using Propositions \ref{proposition:fredholm-index} and \ref{proposition:fredholm-index-ii}.

\end{proof}

As a consequence, we obtain the

\begin{lemma}\label{lemma:projection-solenoidal}
$\pi_{\ker D^*} = \mathbbm{1}-D \Delta^{-1} D^*$ is the orthogonal projection on solenoidal tensors. It is a $] \lambda_d^-, \lambda_d^+[$-admissible pseudo-differential operator of order $0$.
\end{lemma}

We also observe that if we had used the construction (2) (see section \ref{sec:leftright-parametrix}) to find a parametrix for $D$, we would obtain a set of indicial roots strictly larger than necessary, since it would include the roots of $\Delta$, and not only $-1$.

\subsection{Solenoidal injectivity of the X-ray transform}

We now prove Theorem \ref{theorem:xray-injectivite}. As usual, the proof relies on the \textit{Pestov identity} combined with the Livsic theorem. It follows exactly that of \cite{Croke-Sharafutdinov-98}; nevertheless, we thought it was wiser to include it insofar as we only work in $H^1$ regularity on a noncompact manifold (where as \cite{Croke-Sharafutdinov-98} is written in smooth regularity on a compact manifold). 

We recall that there exists a canonical splitting
\[ T_{(x,v)}(TM) = \V_{(x,v)} \oplus^\bot \HH_{(x,v)}, \]
where $(x,v) \in TM$ which is orthogonal for the Sasaki metric. We insist on the fact that \emph{we now work on the whole tangent bundle $TM$ and no longer on the unit tangent bundle $SM$}. As a consequence, the horizontal space $\HH$ is the same but the vertical space $\V$ sees its dimension increased by $1$. These two spaces are identified to the tangent vector space $T_xM$ via the maps $d\pi$ and $\mc{K}$.

Given $u \in C^\infty(TM)$, we can write $\nabla_S u = \nabla^v u + \nabla^h u$, where $\nabla^v u \in \V, \nabla^h u \in \HH$. We denote by $\Div^{v,h}$ the formal adjoints of the operators $\nabla_S^{v,h}$ (see \cite[Section 2]{Paternain-Salo-Uhlmann-15} for further details).

\begin{proof}
We first start with an elementary inequality. Let $u \in C^\infty(SM)$. We extend $u$ to $TM \setminus \left\{0\right\}$ by $1$-homogeneity. The local Pestov identity \cite[Equation (2.14)]{Croke-Sharafutdinov-98} at $(x,v) \in TM$ reads:
\[
2 \langle \nabla^h u, \nabla^v(Xu) \rangle = |\nabla^h u|^2 + \Div^h Y + \Div^v Z - \langle R(v,\nabla^v u)v,\nabla^v u \rangle
\]
where
\[
Y := \langle \nabla^h u, \nabla^v u \rangle v - \langle v, \nabla^h u \rangle \nabla^v u \qquad  Z := \langle v, \nabla^h u\rangle \nabla^h u
\]
Moreover, $\langle v , Z \rangle = |Xu|^2$. Integrating over $SM$ and using the Green-Ostrogradskii formula \cite[Theorem 3.6.3]{Sharafutdinov-94} together with the assumption that the curvature is nonpositive, we obtain:
\begin{equation}
\label{equation:pestov-integree}
\int_{SM} \|\nabla^h u\|^2 d\mu \leq 2 \int_{SM} \langle \nabla^h u, \nabla^v(Xu) \rangle d\mu - (3+d) \int_{SM} \underbrace{\langle v, Z \rangle}_{=|Xu|^2} d\mu
\end{equation}
Note that by a density argument, the previous formula extends to functions $u \in H^1(SM)$ such that $\nabla^v(Xu) \in L^2(SM)$.

We now consider the case where $\pi_m^* f = Xu$ with $f \in H^1$ (and thus $u \in H^1$ and $\nabla^v(Xu) \in L^2$ by the arguments given in the proof of Livsic theorem). Following \cite[Equation (2.18)]{Croke-Sharafutdinov-98}, one obtains the following equality almost-everywhere in $TM$:
\[
2 \langle \nabla^h, \nabla^v(Xu) \rangle = \Div^h W - 4 \times u \pi_m^*(D^*f),
\]
with $W(x,v) = 4u(x,v) (f_x(\cdot,v,...,v))^\sharp$ (where $\sharp : T^*M \rightarrow TM$ is the musical isomorphism). In (\ref{equation:pestov-integree}), this yields
\begin{equation}
\label{equation:pestov-cle}
\int_{SM} \left( |\nabla^h u|^2 + (3+d)|Xu|^2 \right) d\mu \leq - 4 \int_{SM} u \pi_m^*(D^*f) d\mu
\end{equation}

We now assume that $f$ is a symmetric $m$-tensor in
\[
y^\beta C^\alpha(M,S^m(T^*M)) \cap H^1(M,S^m(T^*M)),
\]
such that $D^\ast f = 0$ and $I_m(f)=0$. By the Livsic Theorem \ref{theorem:livsic}, there exists $u\in y^\beta C^\alpha (SM) \cap H^1(SM)$ such that $\pi_m^*f=Xu$. By (\ref{equation:pestov-cle}), we obtain $Xu=0$, thus $f=0$.
\end{proof}

\subsection{Proof of the spectral rigidity}

This section is devoted to the proof of Corollary \ref{corollary:xray-injectivite}. We consider thus as in the statement a family $(g_{\lambda})_{\lambda\in[-1,1]}$ of cusp metrics on a given cusp manifold $(M,g_0)$. We assume that the resonant set of $g_\lambda$ does not depend on $\lambda$. The standard argument in the compact case (see \cite{Guillemin-Kazhdan-80, Guillemin-Kazhdan-80-2} for the original reference) is to say that the trace of the wave-group $(e^{it\sqrt{-\Delta}})_{t \in \R}$ is exactly singular at the length of the periodic orbits (this follows from the trace formula of Duistermaat-Guillemin \cite{Duistermaat-Guillemin-75} or also by earlier work of Colin de Verdière \cite{ColinDeVerdiere-73}), and determined by the spectrum, so that the lengths of the periodic orbits do not change when $\lambda$ varies. 

In this non-compact case, the same argument applies, once the suitable modifications have been made. We explain this without entering into much detail, as it would require to rewrite the content of other articles.

The first point is that the trace of the wave-group is not well defined; it has to be replaced by the \emph{$0$-trace} of the wave-group, a sort of Hadamard regularization in the cusp. In \cite[Section 2.1]{Bonthonneau-4}, it is explained how this regularized trace can be expressed in terms of the discrete spectrum and the scattering determinant. The latter is a meromorphic function which is almost entirely determined by the resonances, as explained in \cite[Theorem 3.32]{Muller-92} for surfaces. The same facts also holds for higher dimension with a similar proof (see \cite[Proposition 1.1.3]{Bonthonneau-thesis}). In any case, the part of the scattering determinant which is \emph{not} determined by the resonances contributes only to a Dirac mass (or derivative of a Dirac mass) at $0$ in the $0$-trace of the wave group. 

On the other hand, the regularization at infinity introduced by the $0$-trace does not introduce some exotic singularities (as explained in section 2.4.2 of \cite{Bonthonneau-4}), so that the arguments of Duistermaat-Guillemin \cite{Duistermaat-Guillemin-75}, which can certainly be carried out in any compact part of the manifold, imply that the singularities of the $0$-trace take place exactly at the algebraic lengths of the periodic orbits of the geodesic flow.

From these considerations, we conclude that the marked length spectrum is constant as $\lambda$ varies. However, as we have seen, for any free hyperbolic homotopy class $c \in \mc{C}$:
\[
\frac{d}{d\lambda} L_{g_\lambda}(c) = \int_{\gamma_{g_\lambda}(c)} \partial_\lambda g_\lambda = 1/2 \times L_{g_\lambda}(c) I_2^{g_\lambda} [\partial_\lambda g_\lambda](c).
\]
Now, since $\partial_\lambda g_\lambda$ is assumed to have compact support, it certainly is an element of $y^{-\infty} C_\ast^N$ for any $N\in \R$. Writing the $L^2$-orthogonal decomposition
\[
\partial_\lambda g_\lambda = f^s_\lambda + D_{g_\lambda} u_\lambda,
\]
where $D_{g_\lambda}$ is the symmetric derivative induced by the metric $g_\lambda$ (see \S\ref{ssection:xray}), $u_\lambda$ is a $1$-form and $f^s \in \ker D_{g_\lambda}^*$ is the solenoidal part of the symmetric $2$-tensor $\partial_\lambda g_\lambda$, we can apply the Lemma \ref{lemma:projection-solenoidal} to deduce that $f^s_\lambda$ and $u_\lambda$ are in $y^{-\epsilon}C_\ast^N$ for some $\epsilon>0$, and every $N>0$ (using here that $\lambda_d^- < 0$). In particular, we can now use the injectivity of the X-ray transform Theorem \ref{theorem:xray-injectivite} to deduce that $f^s_\lambda = 0$. 

Identifying the $1$-form $u_\lambda$ with a vector field $X_\lambda$, we deduce that $\partial_\lambda g_\lambda = \mathcal{L}_{X_\lambda} g_\lambda$ (where $\mc{L}$ stands for the Lie derivative), so that defining the isotopy $(\phi_\lambda)_{\lambda \in [-1,1]}$ such that $\phi_0$ is the identity and
\[
\partial_\lambda \phi_\lambda = X_\lambda \circ \phi_\lambda,
\]
we have $\phi_\lambda^\ast g_\lambda = g_0$. It may be worthwhile to observe here that a priori, $\phi_\lambda$ is not compactly supported, even though we only allowed a compactly supported deformation in the first place. This is somehow a linear version of the solenoidal reduction which will be a key object in our second article.

\bibliographystyle{alpha}
\bibliography{biblio}

\end{document}